\documentclass[11pt,reqno]{amsproc}
\usepackage{amsmath,latexsym,amssymb,verbatim,amsbsy,amsthm,mathrsfs,tikz,times}
\usepackage[margin=1in]{geometry}
\usepackage[shortlabels]{enumitem}
\usepackage{empheq}
\usepackage{enumerate}
\title{On the well-posedness of stochastic Boussinesq equations with cylindrical multiplicative noise}

\author[D. Alonso-Or\'an]{Diego Alonso-Or\'an}
\address{Instituto de Ciencias Matem\'aticas CSIC-UAM-UC3M-UCM, 28049 Madrid, Spain.}
\email{diego.alonso@icmat.es}

\author[A. Bethencourt de Le\'on]{Aythami Bethencourt de Le\'on}
\address{Department of Mathematics, Imperial College, London SW7 2AZ, UK. }
\email{ab1113@ic.ac.uk}

\usepackage{color}
\usepackage[colorlinks=true, pdfstartview=FitV, linkcolor=blue, citecolor=blue, urlcolor=blue]{hyperref}

\theoremstyle{plain}
\newtheorem{theorem}{Theorem}[section]
\newtheorem{definition}[theorem]{Definition}
\newtheorem{lemma}[theorem]{Lemma}
\newtheorem{proposition}[theorem]{Proposition}

\theoremstyle{definition}
\newtheorem{remark}[theorem]{Remark}

\def\tilde{\widetilde}
\numberwithin{equation}{section}
\renewcommand\theequation{\thesection.\arabic{equation}}

\renewcommand\hat{\widehat}

\DeclareMathOperator{\diff}{d\!}
\newcommand{\norm}[1]{\left\lVert#1\right\rVert}    
\newcommand\abs[1]{\left|#1\right|}    

\begin{document}

%%%%%%%%%%%%%%%%%%%%%%%%%% THE ABSTRACT %%%%%%%%%%%%%%%%%%%%%%%%
\begin{abstract}
The Boussinesq equations are fundamental in meteorology. Among other aspects, they aim to model the process of front formation. We use the approach presented in \cite{Principal} to introduce stochasticity into the incompressible Boussinesq equations. This is, we introduce cylindrical transport noise in a way that the geometric properties in the Euler-Poincar\'{e} formulation are preserved. 
%namely, the new stochastic equations will retain a Kelvin circulation theorem, with circulation velocity conserved along stochastic paths and conservation of potential vorticity, also along stochastic paths.\textcolor{blue}{La ultima frase es muy rara y larga}.
One of our main results establishes the local well-posedness of regular solutions for these new stochastic Boussinesq equations. We also construct a blow-up criterion and derive some general estimates, which are crucial for showing well-posedness of a wide range of similar SPDEs. \hfill \today
\end{abstract}

\maketitle

\setcounter{tocdepth}{1}
\tableofcontents
\section{Introduction}
The Boussinesq equations are widely considered as a fundamental model for the study of large scale atmospheric and oceanic flows, built environment, dispersion of dense gases, and internal dynamical structure of stars, \cite{Pedlosky,Ric,Busse}. Beyond its considerable physical relevance, this system of equations has generated substantial interest in the mathematics community due to its rich nonlinear and coupling structure. The physical derivation of the Boussinesq equations dates back to the early 1900's, and more concretely, to the work of Rayleigh \cite{Rayleigh}. He proposed the study of buoyancy driven fluid convection problems by using the equations of Boussinesq \cite{Bouss} in order to explain the experimental work of B\'{e}nard \cite{Ben}.

The 2D Boussinesq equations retain some key features of the well-known Euler and Navier-Stokes equations, as for instance, the vortex stretching mechanism. Moreover, the Boussinesq flow can be interpreted as an analogue of the 3D axisymmetric flow, where vortex stretching terms appear in the vorticity formulation. The Boussinesq equations with various different boundary conditions, on bounded and unbounded domains, have attracted considerable attention and there is a vast literature on the mathematical theory of the deterministic Boussinesq equations \cite{LuoHou1,LuoHou2,ElgJeo,CanDiBenn,Chae,HouLi,ChaeWu,ConDoe}. The fundamental issue of whether classical solutions of the 2D Boussinesq equations can develop finite time singularities remains an outstanding open problem, which is even discussed in Yudovich's ``eleven great problems of mathematical hydrodynamics'' \cite{Yudo03}.

Therefore, the Boussinesq equations encompass tremendously fundamental physical laws, and hence, it is natural to wonder what the motivations to introduce stochasticity in deterministic models like this one are, an idea which has been lately popularised. In particular, introducing stochasticity in a deterministic model in an intelligent fashion can help account for two things: \\
\begin{itemize}[leftmargin=*]
\item Since most deterministic atmospheric models are numerically intractable, they cannot be solved accurately enough with the computer power we have nowadays at our disposal. Moreover, the sensitivity with respect to initial conditions often makes them unreliable, unstable, and unleashes chaotic behaviour. However, the statistical averages and properties of the solutions are typically much more robust. For this reason, this kind of stochastic solutions is incredibly useful to deal with the small unresolved scales. The process of treating this type of problems is called stochastic parameterisation in the literature (see, for example, \cite{LesQua,ZidiFre}). \\
\item Nowadays, the uncertainty due to the radiation phenomena coming from the clouds is considered to be the most drastic source of unpredictability in weather and climate modelling, and it is not fully understood yet. Adding stochasticity might help account for the uncertainty coming from these events and shed some light onto this unknown and complex process. \\
\end{itemize} 

A vast literature exists concerning stochastic fluid dynamics equations. For instance, the stochastic Navier-Stokes equation has been largely studied, starting with the pioneering works of \cite{BenTem72,BenTem73}, and substantial developments have been carried out, see \cite{BenFre,FlaGat,FraRom,GlaZia} and the references therein. Several results have also been established regarding the two and three-dimensional Euler equation \cite{Bes99,BesFla,Kim02,Kim09,CruFlaMal07,GlaVic}. Recently, solution properties of a newly derived stochastic model of the Euler equation were investigated in \cite{CriHolFla,FlaLuo}. This model was proposed by D. Holm in \cite{Principal} and presents an innovative geometric approach for including stochastic processes as cylindrical transport noise in PDE systems via a stochastic variational principle. This new formulation endows the stochastic equations with a rich geometric structure inherited from its deterministic counterpart.  \\

The goal of this paper is manifold: first, we establish local existence and uniqueness of solutions of the system \eqref{mainmain1}-\eqref{mainmain2}, for initial vorticity $\omega_0$ and potential temperature $\theta_0$ in the spaces $H^{2}(\mathbb{T}^{2}, \mathbb{R})$ and $H^{3}(\mathbb{T}^{2}, \mathbb{R}),$ respectively. Second, we prove a blow-up criterion, which partially recovers the most general blow-up criterion for the deterministic case. Finally, we will provide some important derivative estimates, and comment on why they are key when showing local well-posedness of a wide range of stochastic fluid problems (where the stochastic noise depends on the gradient of the velocity).  \\ 
\subsection*{Main results.} In this paper we address the well-posedness of a stochastic version of the 2D incompressible Boussinesq equations,
given by\\
\begin{eqnarray}
 \diff\omega + \mathcal{L}_{u} \omega \diff t+\displaystyle\sum_{i=1}^{\infty}  \mathcal{L}_{\xi_{i}}\omega\circ \diff B^{i}_{t}&=& \partial_{x} \theta \diff t,  \label{mainmain1} \\
 \diff \theta + \mathcal{L}_{u} \theta \diff t + \displaystyle\sum_{i=1}^{\infty}  \mathcal{L}_{\xi_{i}}\theta\circ \diff B^{i}_{t}&=& 0,  \label{mainmain2}
\end{eqnarray}
where $\omega=\nabla^{\perp}\cdot u=\partial_{x}u_{2}-\partial_{y}u_{1}$ is the vorticity, $u$ is the velocity field, and $\theta$ denotes the potential temperature. We assume $\xi_{i},$ $i \in \mathbb{N},$ are prescribed divergence-free vector fields, and $B^{i},$ $i\in\mathbb{N},$ are a family of scalar independent Brownian motions. The system is defined on $\mathbb{T}^{2}\times[0,\infty)$, with $\mathbb{T}^{2}=[-\pi,\pi]^{2}=\mathbb{R}^{2}/(2\pi\mathbb{Z}^{2})$ being the two-dimensional square torus (therefore supplemented with periodic boundary conditions). The derivation of the stochastic 2D Boussinesq equations is carried out in full detail in Section \ref{derivation:sto:bouss}.	Concretely, the aim of the present paper is to prove the following three theorems: 
\begin{theorem}\label{mainth:intro}
Let $(\omega_{0},\theta_{0})\in H^{2}(\mathbb{T}^{2},\mathbb{R})\times H^{3}(\mathbb{T}^{2},\mathbb{R})$, then there exists a unique local solution  to the stochastic 2D Boussinesq equations  \eqref{mainmain1}-\eqref{mainmain2} in $H^{2}(\mathbb{T}^{2},\mathbb{R})\times H^{3}(\mathbb{T}^{2},\mathbb{R})$. Namely, if $\omega^1,\omega^2:\mathbb{T}^{2}\times \Xi \times [0,\tau]\to \mathbb{R}$, $\theta^1,\theta^2:\mathbb{T}^{2}\times \Xi\times [0,\tau]\to \mathbb{R}$ are two solutions defined up to the same stopping time $\tau,$ then $\omega^{1}=\omega^{2}$ and $\theta^{1}=\theta^{2}$, a.s. 
\end{theorem}

\begin{theorem}\label{mainth2:intro} 
Let $(\omega_{0},\theta_{0})\in H^{2}(\mathbb{T}^{2},\mathbb{R})\times H^{3}(\mathbb{T}^{2},\mathbb{R})$. Then there exists a stopping time $\tau_{max}$ and processes $\omega:\mathbb{T}^{2}\times \Xi \times [0,\tau_{max})\to \mathbb{R}, \theta:\mathbb{T}^{2}\times \Xi\times [0,\tau_{max})\to \mathbb{R},$ with trajectories in $C([0,\tau_{max});H^{2}(\mathbb{T}^{2},\mathbb{R})\times H^{3}(\mathbb{T}^{2},\mathbb{R}))$. Moreover, if $\tau_{max}<\infty$, then
\[ \int_{0}^{\tau_{max}}  \left( \norm{\nabla u(t)}_{L^{\infty}}+\norm{\nabla \theta(t)}_{L^{\infty}} \right) \ \diff t = \infty. \]
In particular, $\displaystyle\lim\sup_{t\nearrow \tau_{max}} (\norm{\nabla u(t)}_{L^{\infty}}+\norm{\nabla \theta(t)}_{L^{\infty}}) = \infty$. $\tau_{max}$ is the largest time satisfying the aforementioned properties.
\end{theorem}

Finally, we show the following theorem, which will be paramount when deriving some of the Sobolev estimates we need throughout this paper.
\begin{theorem}\label{generalcancellations}
Let $\mathcal{Q}$ be a linear operator of first order with smooth bounded coefficients. Then for $f\in H^{2}(\mathbb{T}^{2},\mathbb{R})$ we have 
\begin{equation}\label{eq:cancellation1:thm}
\langle \mathcal{Q}^2 f, f \rangle_{L^2} +  \langle \mathcal{Q} f, \mathcal{Q} f \rangle_{L^2} \lesssim ||f||_{L^2}^2.
\end{equation}
Moreover, if $f\in H^{2+k}(\mathbb{T}^{2},\mathbb{R})$, and $\mathcal{P}$ is a pseudodifferential operator of order $k,$
\begin{equation}\label{eq:cancellation2:thm}
\langle \mathcal{P} \mathcal{Q}^2 f, \mathcal{P} f \rangle_{L^2} +  \langle \mathcal{P} \mathcal{Q} f, \mathcal{P} \mathcal{Q} f \rangle_{L^2} \lesssim ||f||_{H^k}^2, 
\end{equation}
for every $k\in[1, \infty)$.
\end{theorem}
\begin{remark}
Inequalities in Theorem \ref{generalcancellations} turn out to be fundamental for closing energy estimates in a very general type of  stochastic fluid problems. We discuss this in the appendix.
\end{remark}

\subsection*{Plan of the paper.} The paper is organised along the following lines: \\
\begin{itemize}[leftmargin=*]
\item
In Section \ref{derivation:sto:bouss} we show how to derive the 2D deterministic Boussinesq equations from a Clebsch-type variational principle and use this approach to construct the stochastic version we will focus our study on.  \\
\item
In Section \ref{prel:not:main} we review some basic mathematical background, both deterministic and stochastic, establish key notation and introduce our main assumptions. We also present the main mathematical results of this article. \\ 
\item
In Section \ref{proofs:main:results} we provide the proof of the first key result of this paper. We start by showing local uniqueness of solutions defined up to a certain stopping time. Then we introduce a truncated version of the stochastic Boussinesq equations and show uniqueness of maximal solutions. The next subsection deals with the global existence of mild solutions of the hyper-regularised truncated Boussinesq equations. \\
\item
In Section \ref{comp:lim:proc} we discuss the required compactness methods and limiting procedure. To that purpose, one has to overcome some technical difficulties, which we treat in great detail. \\
\item
In Section \ref{blowup:crit} we prove the second fundamental result of this paper, namely the blow-up criterion. We also discuss the main obstacles regarding the derivation of sharper versions of this blow-up criterion, which are actually available for the deterministic model but seemingly not for its stochastic counterpart.\\

\item In Section \ref{conclusions} we propose some possible future research lines and comment on several open problems which are left to study. \\

\item Appendix \ref{appendixA} gathers the proof of our third main result and probably the most important one, namely, our general derivative estimates. These are needed in a simpler Lie-derivative form throughout the body of this paper. This simplified version of Theorem \ref{generalcancellations} is presented in Section \ref{prel:not:main}. 
\end{itemize}

\newpage

\section{The stochastic Boussinesq equations}\label{derivation:sto:bouss}
\subsection{Clebsch approach derivation} \label{5}
The Boussinesq equations for inviscid, incompressible, 2D fluid flow in a smooth domain $\Omega \subset \mathbb{R}^2$, first derived in \cite{Bouss} are given by
\begin{align}
&\partial_t u + u \cdot \nabla u = - \nabla p + \theta \hat{e}_2, \label{Bous1} \\
&\partial_t \theta  + u \cdot \nabla \theta = 0, \label{Bous2} \\
&\nabla \cdot u = 0, \label{Bous3}
\end{align}
where $u=u(x,y)$ represents the fluid velocity, $\theta=\theta(x,y)$ is the potential temperature, and $p=p(x,y)$ denotes the pressure of the fluid. Also, we have used the notation $\hat{e}_2 = \nabla y = (0,1).$ Due to their variational character, the Boussinesq equations enjoy several important conservation laws.
\begin{theorem}
The Boussinesq system \eqref{Bous1}-\eqref{Bous3} conserves energy and generalised enstrophy.
\begin{align}
h=\int_\Omega \left \{  \frac12 |u|^2  + \theta y \right \} \diff V
\quad\hbox{(energy),} \quad
\,\label{EPEady-erg1}
\end{align}
\begin{align}
C_\Phi=\int_\Omega \Phi(\theta) \diff V
\quad\hbox{(generalised enstrophy),}\quad
\label{EPEady-erg2}
\end{align}
for any differentiable function $\Phi$ of the potential temperature.
\end{theorem}
The following theorem shows that the Boussinesq equations can be derived from a Clebsch-type approach (this kind of approach is presented in \cite{Kuper}). This is the key tool for introducing stochasticity as explained in \cite{Principal}.
\begin{theorem}\label{Theorem:clebschtype}
Let $\Omega$ be a smooth domain. Consider a Lagrangian function $l[u, \theta, D]$ and construct the following action functional
\[ S\left[u,\theta, D, \phi, \beta\right] = \int_a^b  l[u, \theta, D] \diff t 
+ \int_a^b \int_\Omega \left \{ \phi (D_t + \text{div} (D u)) + \beta (\partial_t \theta + u \cdot \nabla \theta ) \right \}\diff V \diff t. \]
Here, $D$ represents a density, and $\beta,$ $\phi$ are multipliers with respect to which we will also take variations. $\diff V$ denotes integration in the domain $\Omega.$ If we apply Hamilton's principle for this action functional, we obtain Boussinesq equations  \eqref{Bous1}-\eqref{Bous3}. 
\end{theorem}
\begin{remark}
The quantities which are paired with $\phi$ and $\beta$ in the action functional have a geometric meaning. Indeed, if $D$ is considered as a two-form density, and $\theta$ as a scalar, their Lie derivatives with respect to a velocity vector field $u$ become
\begin{eqnarray*}
\mathcal{L}_u D &=& \text{div}  (Du), \\
\mathcal{L}_u \theta &=& u \cdot \nabla \theta.
\end{eqnarray*}
Hence, the action functional above could be rewritten as
\[ S[u,\theta, D, \phi, \beta] = \int_a^b  l[u, \theta, D] \diff t 
+ \int_a^b \int_\Omega \{ \phi (D_t + \mathcal{L}_u D) + \beta (\partial_t \theta + \mathcal{L}_u \theta ) \}\diff V \diff t. \]
\end{remark}
\begin{proof}[Proof of Theorem \ref{Theorem:clebschtype}]
First note that the form of the action functional is a way of imposing the constraints coming from the continuity equation and the tracer equation for $\theta.$ Let us take variations on the action functional S:
\begin{eqnarray*}
0 &=&  \delta S\left[u, \theta, D\right]  \\ &=& \delta \int_a^b  l[u, \theta, D] \diff t 
+ \delta \int_a^b  \int_\Omega \{\phi (D_t + \text{div} (D u)) + \beta (\partial_t \theta + u \cdot \nabla \theta) \} \diff V \diff t \\
&=& \int_a^b \left\langle \frac{\delta l}{\delta u} - D \nabla \phi  + \beta \nabla \theta , \delta u \right\rangle_{L^2}  \diff t 
+ \int_a^b \left\langle \frac{\delta l}{\delta D}  - \phi_t - \nabla \phi \cdot u , \delta D \right\rangle_{L^2}  \diff t \\
&+& \int_a^b  \left\langle \frac{\delta l}{\delta \theta}  - \beta_t - \text{div} (\beta u) , \delta \theta \right \rangle_{L^2} \diff t  
+ \int_a^b  \langle D_t + \text{div} (D u),  \delta \phi \rangle_{L^2} \diff t +  \int_a^b    \langle \partial_t \theta + u \cdot \nabla \theta, \delta \beta \rangle_{L^2} \diff t.
\end{eqnarray*}
Here, we have used the notation $\langle \cdot, \cdot \rangle_{L^2}$ to denote $L^2$ inner product. We obtain the following set of  equations
\begin{eqnarray*}
&&\frac{\delta l}{\delta u} = D \nabla \phi  - \beta \nabla \theta,  \quad \frac{\delta l}{\delta D}  = \phi_t + \nabla \phi \cdot u,  \quad
\frac{\delta l}{\delta \theta} = \beta_t + \text{div} (\beta u), \\ \\
  && D_t + \text{div} (D u) = 0, \quad \partial_t \theta + u \cdot \nabla \theta  =0.
\end{eqnarray*}
Now, use this last set of equations to compute 
\begin{eqnarray*}
\left( \frac{\partial}{\partial t} + \mathcal{L}_{u} \right) \left( \frac{1}{D} \frac{\delta l}{\delta u} \cdot \diff x \right) &=&  \left( \frac{\partial}{\partial t} + \mathcal{L}_{u} \right) \left( \left( \nabla \phi  - (\beta/ D) \nabla \theta \right) \cdot \diff x  \right)
= \left( \frac{\partial}{\partial t} + \mathcal{L}_{u} \right) \left( \diff \phi  - (\beta/ D) \diff \theta \right) \\
 &=& \diff \left( \frac{\partial}{\partial t} + \mathcal{L}_{u} \right) \phi  -\left( \frac{\partial}{\partial t} + \mathcal{L}_{u} \right)(\beta / D) \diff \theta - (\beta / D) \diff \left( \frac{\partial}{\partial t} + \mathcal{L}_{u} \right) \theta \\
&=& \diff \frac{\delta l}{\delta D} - \frac{1}{D} \frac{\delta l}{\delta \theta} \diff \theta.
\end{eqnarray*}
Note that if we substitute the Lagrangian function
\begin{align}
l=\int_\Omega \left \{  \frac12 D |u|^2 - D \theta y + p (1-D)\right \} \diff V,
\label{BouLag}
\end{align}
the variational derivatives become
\begin{eqnarray*}
\frac{1}{D}\frac{\delta l}{\delta u} &=& u, \\
\frac{\delta l}{\delta D} &=& \frac{1}{2} |u|^2 - \theta y - p,  \\
\frac{1}{D} \frac{\delta l}{\delta \theta} &=& -y.
\end{eqnarray*}\vspace{0.5cm}
Note that the multiplier $p$ enforces $D=1.$ Hence we obtain the equations
\[ u_t \cdot \diff x + \mathcal{L}_u (u \cdot \diff x) = \nabla (|u|^2/2 - \theta y-p)\cdot \diff x + y \nabla \theta \cdot \diff x ,\]
which can be rewritten as
$$u_t + u \cdot \nabla u = \theta \nabla y - \nabla p,$$
together with the tracer equation
$$\theta_t + u \cdot \nabla \theta = 0,$$
to close the system. Therefore, we have obtained the Boussinesq equations \eqref{Bous1}-\eqref{Bous3} by using a Clebsch-type approach with constraints to ensure conservation of mass and potential temperature. 
\end{proof}

\subsection{Stochastic equations for a general Lagrangian}   \label{sto}
In order to add stochasticity to the Boussinesq equations in a way that their geometric properties are preserved, we imitate the ideas in \cite{Principal}. In our case, there is more than one constraint of the type $q_t + \mathcal{L}_{u} q = 0,$ so we include all of them in our stochastic variational principle. \\\\ 
The new stochastic action functional will be
\begin{eqnarray*}
S[u, \theta, D] &=& \int_a^b l[u, \theta, D] \diff t 
+ \int_a^b  \left( \left\langle \phi , \frac{\diff D}{\diff t} + \mathcal{L}_{u} D \right\rangle_{L^2} + \left\langle \beta , \frac{\diff \theta }{\diff t} + \mathcal{L}_{u} \theta  \right\rangle_{L^2} \right) \diff t \\
&+& \int_a^b  \sum_{i=1}^\infty \langle \phi \diamond D, \xi_i(x,y) \rangle_{L^2} \circ \diff B_i(t) + \int_a^b  \sum_{i=1}^\infty \langle \beta \diamond \theta, \xi_i(x,y) \rangle_{L^2} \circ \diff B_i(t). 
\end{eqnarray*}
Here, $B_i,$ $i \in \mathbb{N},$ represents a countable family of independent Brownian motions, and $\xi_{i},$ $i \in \mathbb{N},$ are prescribed divergence-free vector fields. $``\circ"$ denotes Stratonovich integration. Note that we have also required the diamond operation, which is defined by
$$\langle p \diamond q , \xi \rangle_{L^2} = -\langle p  , \mathcal{L}_{\xi} q \rangle_{L^2}.$$ 
By taking variations, one obtains
\begin{eqnarray*}
\delta S[u, \theta, D] &=& 
\int_a^b \left \langle \diff D  + \mathcal{L}_{\diff X_t} D, \delta \phi \right \rangle_{L^2}  
+\int_a^b \left \langle \diff \theta  + \mathcal{L}_{\diff X_t} \theta, \delta \beta \right \rangle_{L^2} 
+ \int_a^b \left \langle  \frac{\delta l}{\delta D} \diff t - \diff \phi + \mathcal{L}^T_{\diff X_t} \phi, \delta D \right \rangle_{L^2} \\
&+& \int_a^b \left \langle  \frac{\delta l}{\delta \theta} \diff t - \diff \beta + \mathcal{L}^T_{\diff X_t} \beta , \delta \theta \right \rangle_{L^2}
+ \int_a^b \left \langle \frac{\delta l}{\delta u} - D \nabla \phi + \beta \nabla \theta, \delta u \right \rangle_{L^2} \diff t,
\end{eqnarray*}
where $X_t$ is defined by 
\[ \diff X_t = u(x,y,t) \diff t + \sum_{i=1}^\infty \xi_i(x,y) \circ \diff B_i(t),\]
which is a stochastic differential equation in Stratonovich form. Therefore, the stochastic equations of motion for a general Lagrangian depending on $u, \theta, D,$ become 
\begin{eqnarray}  \label{general1}
\left( \diff + \mathcal{L}_{\diff X_t} \right) \left( \frac{1}{D} \frac{\delta l}{\delta u} \cdot \diff x \right) = \diff \frac{\delta l}{\delta D} \diff t - \frac{1}{D} \frac{\delta l}{\delta \theta} \diff \theta \diff t,
\end{eqnarray}
plus the two imposed stochastic transport equations
$$ \diff \theta  + \mathcal{L}_{\diff X_t} \theta = 0,$$
and
$$ \diff D  + \mathcal{L}_{\diff X_t} D = 0.$$
\begin{remark}
The notation $\circ \diff B_i(t)$ represents Stratonovich integration with respect to Brownian motion. Hence, in order to rewrite these equations in It\^o form one has to use the It\^o correction. This will be done later. 
\end{remark}

\subsection{Stochastic incompressible Boussinesq model}
In Subsection \ref{sto}, we derived the stochastic Boussinesq equations for a general Lagrangian function by using the Clebsch approach. Let us now derive the stochastic incompressible Boussinesq equations, which as we have explained, have Lagrangian function (\ref{BouLag}). Note that in the deterministic case, the multiplier $p$ enforces that the velocity $u$ must be divergence-free (since it makes $D=1$). When one substitutes the Lagrangian (\ref{BouLag}) into the general stochastic equations (\ref{general1}), one realises quickly that in order for the computations to work properly (so that the geometric properties of the Boussinesq equations are not lost), one also needs to assume that the stochastic part of $u$ is divergence-free, this is
\[ \nabla \cdot \xi_i(x,y) = 0, \quad i \in \mathbb{N}.\]
We will denote by $\bar{u}$ (instead of $\diff X_t$) the stochastic velocity with noise. With this new notation, the stochastic Boussinesq equations become
\begin{eqnarray*}
\diff u + \bar{u} \cdot \nabla u + u_j \nabla \bar{u}_j &=& \nabla (|u|^2/2) \diff t - \nabla p \diff t +  \theta \hat{e}_2 \diff t, \\
\diff \theta + \bar{u} \cdot \nabla \theta   &=&0, \\
 \nabla \cdot \bar{u} &=& \nabla \cdot u \diff t + \nabla \cdot \xi_i \circ \diff B_i(t) = 0.
 \end{eqnarray*}
Here we employ the Einstein summation convention of summing over repeated indices for the term $u_j \nabla \bar{u}_j.$ Also, $D$ is moved along with the stochastic flow, namely, 
\[ \diff D + \bar{u} \cdot \nabla D = 0.\]

\subsection{Stratonovich to It\^o}\label{StatoIto}
We have obtained stochastic Boussinesq equations in Stratonovich form. This was convenient for us, since the equations in this form preserve important geometric properties we are interested in retaining, such as conservation laws. That is because the Stratonovich integral preserves the standard rules of integral calculus. The stochastic equations for a general Lagrangian in Stratonovich form are
\makeatletter
 \def\@eqnnum{{\normalsize \normalcolor (\theequation)}}
  \makeatother
 { \small	
\begin{eqnarray*}
 \diff  \left( \frac{1}{D} \frac{\delta l}{\delta u} \cdot \diff x \right) + \mathcal{L}_{u} \left( \frac{1}{D} \frac{\delta l}{\delta u} \cdot \diff x \right) \diff t - \nabla \left(\frac{\delta l}{\delta D} \right) \cdot \diff x \diff t &+& \frac{1}{D} \frac{\delta l}{\delta \theta} \nabla \theta \cdot \diff x \diff t  = -\displaystyle \sum_{i=1}^\infty \mathcal{L}_{\xi_i} \left( \frac{1}{D} \frac{\delta l}{\delta u} \cdot \diff x \right) \circ \diff B_i(t), \\
 \diff \theta  + \mathcal{L}_{u} \theta \diff t &=& -\displaystyle \sum_{i=1}^\infty \mathcal{L}_{\xi_i} \theta \circ \diff B_i(t), \\
 \diff D  + \mathcal{L}_{u} D \diff t &=& -\displaystyle \sum_{i=1}^\infty \mathcal{L}_{\xi_i} D \circ \diff B_i(t).
\end{eqnarray*}}
Of course, it is useful to be able to write these equations in It\^o form as well, which can be effected by using the It\^o correction.
\begin{proposition} \label{pasoito}
The It\^o form of our stochastic equations for a general Lagrangian is
\begin{eqnarray*}
\diff  \left( \frac{1}{D} \frac{\delta l}{\delta u} \cdot \diff x \right) &+& \mathcal{L}_{u} \left( \frac{1}{D} \frac{\delta l}{\delta u} \cdot \diff x \right) \diff t - \nabla \left(\frac{\delta l}{\delta D} \right) \cdot \diff x \diff t + \frac{1}{D} \frac{\delta l}{\delta \theta} \nabla \theta \cdot \diff x \diff t  \\
&+& \sum_{i=1}^\infty \mathcal{L}_{\xi_i} \left( \frac{1}{D} \frac{\delta l}{\delta u} \cdot \diff x \right) \diff B_i(t) = \frac{1}{2} \sum_{i=1}^\infty \mathcal{L}_{\xi_i} \left( \mathcal{L}_{\xi_i} \left( \frac{1}{D} \frac{\delta l}{\delta u} \cdot \diff x \right) \right) \diff t, \\
\diff \theta &+& \mathcal{L}_{u} \theta \diff t + \sum_{i=1}^\infty \mathcal{L}_{\xi_i} \theta \diff B_i(t)  = \frac{1}{2} \sum_{i=1}^\infty \mathcal{L}_{\xi_i}(
\mathcal{L}_{\xi_i} \theta ) \diff t  , \\
\diff D  &+& \mathcal{L}_{u} D \diff t  + \sum_{i=1}^\infty \mathcal{L}_{\xi_i} D \diff B_i(t) = \frac{1}{2} \sum_{i=1}^\infty  \mathcal{L}_{\xi_i} (\mathcal{L}_{\xi_i} D) \diff t .
\end{eqnarray*}
\end{proposition}
Note that the stochastic incompressible Boussinesq equations in Stratonovich form can be expressed as
\begin{eqnarray}
\diff u +  (u \cdot \nabla u - \theta \hat{e}_2) \diff t + u_j \nabla \bar{u}_j + \sum_{i=1}^\infty \mathcal{L}_{\xi_i} u \circ \diff B_t^i &=& \nabla (|u|^2/2)\diff t  - \nabla p \diff t, \label{StraBou1} \\
\diff \theta + u \cdot \nabla \theta \diff t + \sum_{i=1}^\infty \mathcal{L}_{\xi_i} \theta \circ \diff B_t^i  &=&0, \label{StraBou2} \\
\nabla \cdot u &=& 0,   \label{StraBou3}\\
\nabla \cdot \xi_i &=& 0. \label{StraBou4}
\end{eqnarray}
To rewrite equations \eqref{StraBou1}-\eqref{StraBou4} in It\^o form we apply Proposition \ref{pasoito}, thus obtaining 
\makeatletter
 \def\@eqnnum{{\normalsize \normalcolor (\theequation)}}
  \makeatother
 { \small	
\begin{eqnarray}
\diff u +  (u \cdot \nabla u - \theta \hat{e}_2) \diff t + u_j \nabla \bar{u}_j  + \sum_{i=1}^\infty \mathcal{L}_{\xi_i} u \diff B_t^i &=& \frac{1}{2} \sum_{i=1}^\infty \mathcal{L}^2_{\xi_i} u \diff t + \nabla (|u|^2/2)  \diff t- \nabla p \diff t, \label{ItoBou1}\\
\diff \theta + u \cdot \nabla \theta \diff t +  \sum_{i=1}^\infty \mathcal{L}_{\xi_i} \theta \diff B_t^i  &=& \frac{1}{2} \sum_{i=1}^\infty \mathcal{L}^2_{\xi_i} \theta \diff t, \label{ItoBou2}\\
\nabla \cdot u &=& 0,  \label{ItoBou3}\\
\nabla \cdot \xi_i &=& 0.  \label{ItoBou4}
\end{eqnarray}}
\begin{remark}
As a reminder, note that if one wants to prove Proposition \ref{pasoito}, or as a particular case, to derive
\eqref{ItoBou1}-\eqref{ItoBou4} from \eqref{StraBou1}-\eqref{StraBou4}, one has to calculate the cross-variational terms coming from the identity
$$\int_0^t f \circ d B_s  = \int_0^t f d B_s + \frac{1}{2} \left[f, B \right], $$
where $\left[\cdot, \cdot \right]$ represents the cross-variation between two stochastic processes. So, in our case
$$\left[\mathcal{L}_{\xi_i} u, B^i \right]_t = \mathcal{L}_{\xi_i} \left[ u, B^i \right]_t = - \mathcal{L}_{\xi_i} \int_0^t \mathcal{L}_{\xi_i} u(\cdot, s) \diff s =  -  \int_0^t \mathcal{L}^2_{\xi_i} u(\cdot, s) \diff s.$$
A similar result is also obtained for $\theta$, namely
$$\left[\mathcal{L}_{\xi_i} \theta, B^i \right]_t =  -  \int_0^t \mathcal{L}^2_{\xi_i} \theta(\cdot, s) \diff s.$$
\end{remark}
Finally, since we prefer to avoid the pressure term when working with equations \eqref{ItoBou1}-\eqref{ItoBou4}, we take the curl operator on the first equation, obtaining
\begin{eqnarray}
  \diff \omega + \mathcal{L}_{u} \omega \diff t+\displaystyle\sum_{i=1}^{\infty}  \mathcal{L}_{\xi_{i}} \omega \diff B^{i}_{t}&=& \frac{1}{2} \displaystyle\sum_{i=1}^{\infty}  \mathcal{L}^{2}_{\xi_{i}}\omega \diff t+\partial_{x} \theta \diff t,  \label{eq:vorticity:ito} \\
  \diff \theta + \mathcal{L}_{u} \theta \diff t + \displaystyle\sum_{i=1}^{\infty}  \mathcal{L}_{\xi_{i}}\theta \diff B^{i}_{t}&=& \frac{1}{2} \displaystyle\sum_{i=1}^{\infty}  \mathcal{L}^{2}_{\xi_{i}}\theta \diff t, \label{eq:theta:ito}
\end{eqnarray}
where $\omega=\nabla^{\perp}\cdot u=\partial_{x}u_{2}-\partial_{y}u_{1}$ is the vorticity. To close the system, the velocity $u$ can be calculated from $\omega$ by using the Biot-Savart law (see \ref{Biot}) and the divergence-free condition $\nabla \cdot u=0$.

\newpage

\section{Preliminaries, notation and main results} \label{prel:not:main}
\subsection{Preliminaries and notation}
The first result featured in this paper shows local existence in time and uniqueness of regular solutions of the stochastic Boussinesq equations \eqref{eq:vorticity:ito}-\eqref{eq:theta:ito}. The system is defined on $\mathbb{T}^{2}\times[0,\infty)$, with $\mathbb{T}^{2}=[-\pi,\pi]^{2}=\mathbb{R}^{2}/(2\pi\mathbb{Z}^{2})$ being the two-dimensional square torus. Note that we have chosen to focus on the periodic case for the sake of simplicity; however, results can be straightforwardly extended to the whole domain $\mathbb{R}^{2}$. In the presence of boundaries, this is, for smooth bounded domains $\Omega\subset\mathbb{R}^{2}$, a more careful analysis is required and presents a future line of research. We next introduce the functional setting and some mathematical background which will be used throughout this article.
%\begin{eqnarray}
%  \diff\omega + \mathcal{L}_{u} \omega \diff t+\displaystyle\sum_{i=1}^{\infty}  \mathcal{L}_{\xi_{i}}\omega\circ \diff B^{i}_{t}&=& \partial_{x} \theta \diff t, \label{eq:vorticity} \\ 
% \diff \theta + \mathcal{L}_{u} \theta \diff t + \displaystyle\sum_{i=1}^{\infty}  \mathcal{L}_{\xi_{i}}\theta\circ \diff B^{i}_{t}&=& 0, \label{eq:theta} 
%\end{eqnarray}
%where $\omega=\nabla^{\perp}\cdot u=\partial_{x}u_{2}-\partial_{y}u_{1}$ is the vorticity.
%The system  is defined on $\mathbb{T}^{2}\times(0,\infty)$, with $\mathbb{T}^{2}=[-\pi,\pi]^{2}=\mathbb{R}^{2}/(2\pi\mathbb{Z}^{2})$ being the two-dimensional square torus. 
%  The above Stratonovich stochastic partial differential equation can be rewritten in It\^o form (cf. subsection \ref{StatoIto}) as
%\begin{eqnarray}
%  \diff \omega + \mathcal{L}_{u} \omega \diff t+\displaystyle\sum_{i=1}^{\infty}  \mathcal{L}_{\xi_{i}} \omega \diff B^{i}_{t}&=& \frac{1}{2} \displaystyle\sum_{i=1}^{\infty}  \mathcal{L}^{2}_{\xi_{i}}\omega \diff t+\partial_{x} \theta \diff t,  \label{eq:vorticity:ito} \\
%  \diff \theta + \mathcal{L}_{u} \theta \diff t + \displaystyle\sum_{i=1}^{\infty}  \mathcal{L}_{\xi_{i}}\theta \diff B^{i}_{t}&=& \frac{1}{2} \displaystyle\sum_{i=1}^{\infty}  \mathcal{L}^{2}_{\xi_{i}}\theta \diff t. \label{eq:theta:ito}
%\end{eqnarray}
 \newline 
\subsubsection{Sobolev spaces and embeddings.} Sobolev spaces are defined as
\[ W^{k,p}:=\lbrace f\in L^{p}(\mathbb{T}^{2},\mathbb{R}):  (I-\Delta)^{k/2}f\in L^{p}(\mathbb{T}^2,\mathbb{R}) \rbrace, \]
for any $k\geq0$ and $p\in[1,\infty],$ equipped with the norm $||f||_{W^{k,p}}=||(I-\Delta)^{k/2}f||_{L^{p}}$. Here, we denote by $(I-\Delta)^{k/2}f$ to be the function having Fourier transform $(1+|\xi|^{2})^{k/2}\hat{f}(\xi),$ where $\hat{f}$ represents the Fourier transform of $f$. Sometimes we will also use the notation $\Lambda^{k}=(-\Delta)^{k/2}$. Recall that $L^{2}$ based spaces are Hilbert spaces and may alternatively be denote by $H^{k}=W^{k,2}$. For $k>0$, we also define $H^{-k}:=(H^{k})^{\star}$, i.e. the dual space of $H^k$. Along the paper we will be using different forms of Sobolev embeddings. For the sake of clarity, we collect below the ones we will most often make use of:
\begin{eqnarray}
\left\Vert f \right\Vert_{L^{4}} &\lesssim& \left\Vert f \right\Vert ^{1/2}_{L^{2}}  \left\Vert \nabla f \right\Vert ^{1/2}_{L^{2}}, \label{Sob:ine1} \\
\left\Vert \nabla f \right\Vert_{L^{4}} &\lesssim& \left\Vert f \right\Vert ^{1/2}_{L^{\infty}}  \left\Vert \Delta f \right\Vert ^{1/2}_{L^{2}}, \label{Sob:ine3} \\
\left\Vert f \right\Vert_{L^{\infty}} &\lesssim& \left\Vert f \right\Vert_{H^{1+\epsilon}}, \ \ \text{for every } \epsilon>0. \label{Sob:ine2}
\end{eqnarray} 

\subsubsection{The Biot-Savart operator.} \label{Biot} As we previously mentioned, in order to close the system of partial differential equations, we need to be able to calculate $u$ from the vorticity $\omega$. This reconstruction is obtained by means of the Biot-Savart operator, namely $u=K\star\omega=\nabla^{\perp}\Delta^{-1} \omega$. As a consequence, it is easy to check that the following inequality 
\begin{equation}\label{BiotIne}
 ||u||_{W^{k+1,p}}\leq C_{k,p} ||\omega||_{W^{k,p}},
\end{equation}
holds for all $k\geq 0, p\in(1,\infty),$ where $C_{k,p}=C(k,p)$ denotes a positive constant.
\newline 

\subsubsection{Assumptions on the vector fields $\xi_{i}$.} To give a reasonable meaning to the stochastic terms and to show certain estimates we need to impose the following assumption. The vector fields 
\\$\xi_i: \mathbb{T}^2 \rightarrow \mathbb{R}^2$ are assumed to be of class $C^4$ and to satisfy 
\begin{equation}
\displaystyle\sum_{i=1}^{\infty} \left| \left| \xi_i \right| \right|_{H^{3}}^2 < \infty. \label{assumption3}
\end{equation}
With this assumption in mind, it is easy to check that for smooth enough functions $f$:
\begin{eqnarray}
\left| \left|\displaystyle\sum_{i=1}^{\infty} \mathcal{L}^2_{\xi_i} f \right| \right|_{L^2}^2 &\lesssim & ||f||_{H^2}^2, \label{assumption1} \\ 
\displaystyle\sum_{i=1}^{\infty} \langle \mathcal{L}_{\xi_i} f, \mathcal{L}_{\xi_i} f \rangle_{L^2} &\lesssim &||f||_{H^2}^2 \label{assumption2}.
\end{eqnarray}
Inequalities \eqref{assumption3}-\eqref{assumption2} will be frequently applied throughout this article. Moreover, since during the proofs of some of the main uniqueness and existence theorems several high order terms appear in the energy estimates (similar to the ones in \cite{CriHolFla}), one needs to make use of some facts which are collected in the following proposition. Its most general version appears in Appendix \ref{appendixA}, where we also comment on how to use this result for showing existence and uniqueness results in a more general class of SPDEs.
\begin{proposition}\label{Liecancellations}
Let $f\in H^{2}(\mathbb{T}^{2}, \mathbb{R})$ and $\xi_i$ be vector fields satisfying \eqref{assumption3} . Then we have
\begin{equation}\label{eq:cancellation1}
\langle \mathcal{L}^2_{\xi_i} f, f \rangle_{L^2} +  \langle \mathcal{L}_{\xi_i} f, \mathcal{L}_{\xi_i} f \rangle_{L^2} = 0.
\end{equation}
Moreover, if $f\in H^{k+2}(\mathbb{T}^{2},\mathbb{R})$ and $\xi_{i}$ are of class $C^{k+1}$ satisfying
\[\displaystyle\sum_{i=1}^{\infty} ||\xi_{i}||^{2}_{H^{k+1}} < \infty, \]
there exists a positive constant $C=C(i)$ such that
\begin{equation}\label{eq:cancellation2}
\displaystyle\sum_{i=1}^{\infty}\langle \Lambda^k\mathcal{L}^2_{\xi_i} f, \Lambda^k f \rangle_{L^2} +  \langle \Lambda^k \mathcal{L}_{\xi_i} f, \Lambda^k \mathcal{L}_{\xi_i} f \rangle_{L^2} \leq C ||f||_{H^k}^2, 
\end{equation}
for every $k\in[1, \infty)$.
\end{proposition}
\begin{remark}
Estimates (\ref{eq:cancellation1}) and (\ref{eq:cancellation2}) are very surprising, since the terms of highest order and one order less cancel. This turns out to be a general property regarding differential operators (see Appendix \ref{appendixA}).
\end{remark}
\subsubsection{Theory of analytical semigroups.} For the sake of completeness, we also include several facts from the theory of analytic semigroups which will be useful later on. For any fixed $k \in \mathbb{N}$ let us denote $\mathcal{D}(A)= H^{2k}(\mathbb{T}^{2},\mathbb{R}),$ and define the operator $A:\mathcal{D}(A)\to L^{2}(\mathbb{T}^{2},\mathbb{R})$ by $Af=\nu\Delta^{k}f$, with $\nu$ a positive real number. This operator is self-adjoint and negative definite. Let $e^{tA}$ be the semigroup generated by the operator $A$ in $L^{2}(\mathbb{T}^{2},\mathbb{R})$. The fractional powers $(I-A)^{\alpha}$ are well-defined for every $\alpha>0$. Moreover, we have
\[ \left\Vert f \right\Vert_{H^{2k\alpha}}\leq C_{\alpha}||(I-A)^{\alpha} f||_{L^{2}}, \]
for some $C_{\alpha}>0$ and $f\in H^{2k\alpha}(\mathbb{T}^{2},\mathbb{R})$. The fractional powers commute with the semigroup $e^{tA}$ (cf. \cite{Pazy}), and have the following property
\begin{equation}\label{semigroupine}
\left\Vert (I-A)^{\alpha}e^{tA}f\right\Vert_{L^{2}} \leq \frac{C_{\alpha}}{t^\alpha} \left\Vert f \right\Vert_{L^2},
\end{equation}
for all $t \in (0,T]$ and functions $f\in L^{2}(\mathbb{T}^{2},\mathbb{R}).$ With this property in hand, let us prove the following statement.
\begin{lemma}\label{estimation:semigroup}
Let $f \in C([0,T];L^{2}(\mathbb{T}^{2},\mathbb{R})),$ $f_i \in C([0,T];L^{2}(\mathbb{T}^{2},\mathbb{R})), i \in \mathbb{N},$ and $t\in(0,T].$ We have that
\begin{equation}\label{estimation:semigroup1}
\left\Vert \int_{0}^{t}e^{\left(  t-s\right)  A}f\left(  s\right)
\diff s\right\Vert _{H^{\beta}}^{2}\lesssim  T^{2-\beta / k}\sup_{t\in\left[  0,T\right]  }\left\Vert f\left(
s\right)  \right\Vert _{L^{2}}^{2},
\end{equation}
for  $0<\beta<k$. Moreover,
\begin{equation} \label{estimation:semigroup2}
\mathbb{E}\left[  \sup_{t\in\left[  0,T\right]  }\left\Vert \sum_{i=1}^{\infty}\int_{0}^{t}e^{\left(  t-s\right)  A}f_{i}\left(  s\right)  \diff B_{s}^{i}\right\Vert
_{H^{\beta}}^{2}\right] \lesssim T^{2-\beta/k}\ \mathbb{E}\left[  \sup_{s\in\left[ 0,T\right]  }\sum_{i=1}^{\infty
}\left\Vert f_{i}\left(  s\right)  \right\Vert _{L^{2}}^{2}\right],
\end{equation}
for $0<\beta<k$.  
\end{lemma}
\begin{proof}[Proof of Lemma \ref{estimation:semigroup}]
We just show the first inequality, as the second follows analogously. Note that
\begin{eqnarray*}
\left\Vert \int_{0}^{t}e^{\left(  t-s\right)  A}f\left(  s\right)
\diff s\right\Vert _{H^{\beta}}&\leq & C_{\alpha} \left\Vert \left(  I-A\right)  ^{\beta / 2k}\int%
_{0}^{t}e^{\left(  t-s\right)  A}f\left(  s\right)  \diff s\right\Vert _{L^{2}}\leq
C_{\alpha} \int_{0}^{t}\frac{1}{\left(  t-s\right)  ^{\beta/2k}}\left\Vert f\left(
s\right)  \right\Vert _{L^{2}}\diff s,
\end{eqnarray*}
where we have used property (\ref{semigroupine}) with $\alpha = \beta / 2k$ for the second inequality. The assertion follows by Jensen's inequality.
\end{proof}
\subsubsection{Duhamel's principle and mild sense.} \label{Duhamel} In order to show global existence of the regularised equations (cf.\ \ref{globalregular}), we need to rewrite these in a convenient way, namely, as an abstract stochastic evolution equation
\begin{eqnarray}
\label{abstractformulation} 
\diff U + BU \ \diff t + \displaystyle\sum_{i=1}^{\infty}R_{i}U  \ \diff B^{i }_{t}&=& GU \ \diff t + LU \ \diff t,  \\
\label{abstractinitial} U(0)&=&U_{0},
\end{eqnarray}
where $U:=(\omega,\theta)$, $BU:=(u\cdot\nabla\omega,u\cdot\nabla\theta )$, $GU:=(\partial_{x}\theta,0)$, $LU:=\left( \dfrac{1}{2}\displaystyle\sum_{i=1}^{\infty} \mathcal{L}^{2}_{\xi_{i}}\omega,\dfrac{1}{2}\displaystyle\sum_{i=1}^{\infty} \mathcal{L}^{2}_{\xi_{i}}\theta \right),$ and $R_{i}U:= \left( \mathcal{L}_{\xi_{i}}\omega, \mathcal{L}_{\xi_{i}}\theta \right)$.
With this new formulation, we say that $U$ satisfies (\ref{abstractformulation})-(\ref{abstractinitial}) in the mild sense if
\[ U(t)=e^{tA}U_{0}-\int
_{0}^{t}e^{\left(  t-s\right)  A}(BU(s)-GU(s)) \ \diff s +\int_{0}^{t}e^{\left(  t-s\right)  A} LU(s) \ \diff s-
\displaystyle\sum_{i=1}^{\infty}\int_{0}^{t}e^{\left(  t-s\right)  A}R_{i}U(s) \diff B_{s}^{i}, \]
where $e^{tA}$ is the semigroup generated by $A$ defined previously. \\
\subsubsection{Compact embedding theorems}
We shall make use of a compact embedding result (see \cite{FlaGat}) which is a variation of the classical Aubin-Lions Lemma \cite{Lio69}. To this end we first recall some spaces of fractional in time derivative. Let $W$ be a Banach space and consider the space of functions 
\begin{eqnarray}  \label{dobleuve}
\{f:[0,T] \rightarrow W\}.
\end{eqnarray}
For fixed $p>1$ and $0<\alpha<1$, we define 
\[ W^{\alpha,p}([0,T]; W)=\bigg\{ f\in L^{p}([0,T]; W):  \int_0^T \int_0^T \frac{||f(t)-f(s)||_W^p}{|t-s|^{1+p\alpha}} \diff t \diff s< \infty \bigg\}. \] 
We endow this space with the following norm
\[ \norm{f}^{p}_{W^{\alpha,p}([0,T]; W)}:=  \int_0^T ||f(t)||_W^p \diff t + \int_0^T \int_0^T \frac{||f(t)-f(s)||_W^p}{|t-s|^{1+p\alpha}} \diff t \diff s. \]
We now have all the tools to state the compact embedding lemma. 
\begin{lemma}\label{compactlemma}Suppose that $X\subset Y\subset Z$  are Banach spaces with $X, Z$ reflexive, and that the embedding of $X$ into $Y$ is compact. Then for any $1<p<\infty$ and $0 < \alpha < 1$, the embedding:
\[ L^{p}([0,T]; X) \cap W^{\alpha,p}([0,T]; Z)\hookrightarrow L^{p}([0,T]; Y) \]
is compact.
\end{lemma}
%\begin{lemma}
%If $Y_{0}\subset  Y$ are two Banach spaces with compact embedding, and the  real numbers $0<\alpha<1$, $p>1$ satisfy
%\[ \alpha p >1\]
%then the space $W^{\alpha,p}([0,T];Y_{0})$ is compactly embedded in %$C([0,T]; Y)$.
%\end{lemma}
\subsubsection{The stochastic framework.}
We next briefly recall some notions and aspects of the theory of stochastic analysis. We refer the reader to the classical references \cite{PraZab,Fla96,Fla11} for a more thorough review. We begin by fixing a stochastic basis
 $\mathcal{S}=(\Xi,\mathcal{F},\lbrace\mathcal{F}_{t}\rbrace_{t\geq 0}, \mathbb{P},\lbrace B^{i}\rbrace_{i\in\mathbb{N}}),$ that is, a filtered probability space together with a sequence $\lbrace B^{i} \rbrace_{i\in\mathbb{N}}$ of scalar independent Brownian motions relative to the filtration $\lbrace\mathcal{F}_{t}\rbrace_{t\geq0}$.  \\ \\
Given a stochastic process $X\in L^{2}(\Xi;L^{2}([0,\infty);L^{2}(\mathbb{T}^{2},\mathbb{R}))),$ one may define the It\^o stochastic integral by
\[ M_{t}= \int_{0}^{t} X \ \diff B= \displaystyle\sum_{i=1}^\infty \int_{0}^{t} X_{i} \ \diff B^{i}, \quad t > 0, \]
where $X_{i}= Xe_{i},$ being $\lbrace e_{i} \rbrace_{i\in \mathbb{N}}$ a complete orthonormal basis in $L^{2}(\mathbb{T}^{2},\mathbb{R})$. This definition makes $\{M_t\}_{t>0}$ an element of the square integrable martingales. The process $\lbrace M_{t} \rbrace_{t>0}$ enjoys many good properties. An important one is the so called Burkholder-Davis-Gundy inequality, which in the present context reads
\begin{equation}\label{BGD:ineq}
 \mathbb{E}\left[ \displaystyle \sup_{s \in [0,T]}\left| \int_{0}^{t} X_{s} \diff B_ {s}\right|^{p} \right] \leq C_p \mathbb{E} \left[ \int_{0}^{T} |X_s|^{2} \ dt \right]^{p/2},
\end{equation}
for any $p\geq 1$ and $C_{p}$ an absolute constant depending on $p$. \\\\
Finally, we review some classical and standard convergence tools from abstract probability theory. These results will be paramount for establishing the needed convergence of the associated hyper-regularised truncated equations to a solution. Let $(X,d)$ be a separable metric space and $\mathcal{B}(X)$ the Borel $\sigma$-algebra.  Let $\mathcal{P}(X)$ denote the collection of all the probability measures that can be defined on $(X,\mathcal{B}(X)).$ A set $\Gamma\subset \mathcal{P}(X)$ is said to be tight if, for every $\epsilon >0,$ there exists a compact subset $K_\epsilon \subset X$ such that 
\[ \mu (K_\epsilon) \geq 1 - \epsilon, \quad \forall \mu\in \Gamma. \]
We say that a sequence $\{\mu_n\}_{n \in \mathbb{N}} \subset \mathcal{P}(X)$ converges weakly to a probability measure $\mu$  if 
\[ \displaystyle\lim_{n \rightarrow \infty} \int_X \varphi \diff \mu_{n} = \int_X \varphi \diff \mu, \]
for all bounded continuous functions $\varphi: X \rightarrow \mathbb{R}.$  On the other hand, a set $\Gamma\subset\mathcal{P}(X)$ is weakly compact if every sequence $\lbrace \mu_{n} \rbrace_{n \in \mathbb{N}}\subset \Gamma$ has a weakly convergent subsequence. The proofs of the following two classical results can be found in \cite{PraZab},\cite{GyoKry}.
\begin{theorem}[Prokhorov] \label{prohorov}
The collection $\Gamma \subset \mathcal{P}(X)$ is weakly compact if and only if it is tight.
\end{theorem}
\begin{theorem}[Skorokhod representation]  \label{rep}
Let $\{\mu_n\}_{n\in \mathbb{N}}$ be a sequence of probability measures that converges weakly to some measure $\mu$. Assume the support of $\mu$ is separable. Then there exists a probability space $(\Omega, \mathcal{A}, \mathbb{P})$ and random variables $\{X_n\}_{n=1}^\infty$, such that $X_n$ converges almost surely to a random variable $X$, where the laws of $X_{n}$ and $X$ are $\mu_{n}$ and $\mu,$ respectively.
\end{theorem}
Let us state the celebrated Gy\"ongy-Krylov result.
\begin{lemma}[Gy\"ongy-Krylov lemma]    \label{strong}
Let $\{X_n\}_{n\in \mathbb{N}}$ be a sequence of random variables with values in a Polish space $(E,d),$ endowed with the Borel $\sigma$-algebra. Then $X_{n}$ converges in probability to an $E$-valued random process, if and only if, for every pair of subsequences $\{X_{n_j}, X_{m_j}\}_{j \in \mathbb{N}},$ there exists a further subsequence that converges weakly to a random variable supported on the diagonal $\lbrace(x,y)\in E\times E: x=y\rbrace$. 
\end{lemma}
We conclude this subsection by recalling the following classical probability theory inequality: 
\begin{lemma}[Markov's inequality]\label{desigualdad}
Let $X$ be a nonnegative random variable and $A >0$. Then
$$\mathbb{P} (X > A)  \leq \frac{\mathbb{E} (X)}{A}.$$
\end{lemma}

\subsection{Statement of the main results}
Let us state here some fundamental definitions and the main theorems that we are going to prove in the following sections. 
\begin{definition}[Local solution] \label{localsol}
A local solution $(\omega,\theta) \in H^2(\mathbb{T}^{2},\mathbb{R}) \times H^3(\mathbb{T}^{2},\mathbb{R})$ to the Boussinesq equations (\ref{eq:vorticity:ito})-(\ref{eq:theta:ito}) is a pair of random variables $\omega: \mathbb{T}^2 \times \Xi \times [0,\tau] \rightarrow \mathbb{R},$ $\theta: \mathbb{T}^2 \times \Xi \times [0,\tau] \rightarrow \mathbb{R},$ with trajectories of class $C([0,\tau]; H^2(\mathbb{T}^{2},\mathbb{R})\times H^3(\mathbb{T}^{2},\mathbb{R}))$, together with a stopping time $\tau: \Xi \rightarrow [0, \infty]$ such that $\omega(t \wedge \tau),$ $\theta(t \wedge \tau)$ are adapted to $\lbrace\mathcal{F}_{t}\rbrace_{t\geq 0},$  and (\ref{eq:vorticity:ito})-(\ref{eq:theta:ito}) holds in the $L^2$ sense. This is
\begin{eqnarray*}
 \omega_{\tau'} - \omega_0  + \int_0^{\tau'} \mathcal{L}_{u} \omega \diff s +\displaystyle \sum_{i=1}^{\infty} \int_0^{\tau'} \mathcal{L}_{\xi_{i}}\omega  \diff B^{i}_{s}&=& \dfrac{1}{2}\displaystyle\sum_{i=1}^{\infty}\int_{0}^{\tau'}\mathcal{L}^{2}_{\xi_{i}} \omega \diff s+ \int_0^{\tau'} \partial_{x} \theta \diff s,  \\
 \theta_{\tau'} - \theta_0 +  \int_0^{\tau'} \mathcal{L}_{u} \theta \diff s + \displaystyle\sum_{i=1}^{\infty}  \int_0^{\tau'} \mathcal{L}_{\xi_{i}}\theta \diff B^{i}_{s}&=&  \dfrac{1}{2}\displaystyle\sum_{i=1}^{\infty}\int_{0}^{\tau'}\mathcal{L}^{2}_{\xi_{i}} \theta \diff s,
 \end{eqnarray*}
for finite stopping times $\tau' \leq \tau$. A pair $(\omega,\theta)\in L^{2}(\mathbb{T}^{2}\times[0,\tau])\times L^{2}(\mathbb{T}^{2}\times[0,\tau])$ is said to satisfy equations (\ref{eq:vorticity:ito})-(\ref{eq:theta:ito}) in the weak sense if
 \makeatletter
 \def\@eqnnum{{\normalsize \normalcolor (\theequation)}}
  \makeatother
 { \small	
\begin{eqnarray*}
 \langle \omega_{\tau'}, \phi \rangle_{L^2} - \langle \omega_0, \phi \rangle_{L^2} - \int_0^{\tau'} \langle  \omega,\mathcal{L}_{u} \phi \rangle_{L^2} \diff s- \displaystyle \sum_{i=1}^{\infty} \int_0^{\tau'} \langle \omega, \mathcal{L}_{\xi_{i}}\phi \rangle_{L^2} \diff  B^{i}_{s}&=& \dfrac{1}{2}\displaystyle\sum_{i=1}^{\infty}\int_{0}^{\tau'} \langle\omega, {\mathcal{L}}^2_{\xi_{i}}\phi\rangle_{L^2} \ \diff s-\int_0^{\tau'} \langle   \theta, \partial_{x} \phi \rangle_{L^2} \diff s,  \\
\langle \theta_{\tau'}, \phi \rangle_{L^2} - \langle \theta_0, \phi \rangle_{L^2}  - \int_0^{\tau'} \langle \theta,\mathcal{L}_{u}  \phi \rangle_{L^2}  \diff s - \displaystyle\sum_{i=1}^{\infty}  \int_0^{\tau'} \langle \theta, \mathcal{L}_{\xi_{i}}\phi \rangle_{L^2} \diff  B^{i}_{s}&=& \dfrac{1}{2} \displaystyle\sum_{i=1}^{\infty}\int_{0}^{\tau'} \langle\theta, {\mathcal{L}}^2_{\xi_{i}} \phi\rangle_{L^2} \ \diff s, \end{eqnarray*}}
 for all test functions $\phi \in C^\infty (\mathbb{T}^2,\mathbb{R}).$ 
\end{definition}

\begin{definition}[Maximal solution]\label{def:maximal}
A maximal solution of \eqref{eq:vorticity:ito}-\eqref{eq:theta:ito} is a stopping time $\tau_{max}:\Xi \to [0,\infty]$ and random variables $\omega: \mathbb{T}^2 \times \Xi \times [0,\tau_{max}) \rightarrow \mathbb{R}$, $\theta: \mathbb{T}^2 \times \Xi \times [0,\tau_{max}) \rightarrow \mathbb{R}$ such that: \\
\begin{itemize}[leftmargin=*]
\item $\mathbb{P} (\tau_{max} >0) = 1, \ \tau_{max} = lim_{n \rightarrow \infty} \tau_n,$ where $\tau_n$ is an increasing sequence of stopping times, i.e. $\tau_{n+1}\geq \tau_{n}$ $\mathbb{P}$ almost surely. \\
\item $(\tau_{n},\omega,\theta)$ is a local solution for every $n\in\mathbb{N}$. \\
\item If $(\tau',\omega',\theta')$ is another triplet satisfying the above conditions and $(\omega',\theta')=(\omega,\theta)$ on $[0,\tau'\wedge \tau_{max} )$, then $\tau'\leq \tau_{max}$ $\mathbb{P}$ almost surely. \\
\end{itemize}
\end{definition}
We are now ready to state the main results of this article:
\begin{theorem}\label{mainth}
Let $(\omega_{0},\theta_{0})\in H^{2}(\mathbb{T}^{2},\mathbb{R})\times H^{3}(\mathbb{T}^{2},\mathbb{R})$, then there exists a unique maximal solution $(\tau_{max},\omega,\theta)$ of the 2D stochastic Boussinesq equations  \eqref{eq:vorticity:ito}-\eqref{eq:theta:ito}. If $(\tau',\omega',\theta')$ is another maximal solution of \eqref{eq:vorticity:ito}-\eqref{eq:theta:ito}, then necessarily $\tau_{max}=\tau'$, $\omega=\omega',$ and $\theta=\theta'$ on $[0,\tau_{max})$. Moreover, either $\tau_{max}=\infty$ or $\displaystyle\lim\sup_{s\nearrow \tau_{max}}\left(||\omega(s)||_{H^{2}}+ ||\theta(s)||_{H^{3}}\right)= \infty$.
\end{theorem}
In this paper we also construct a blow-up criterion, which reads:
\begin{theorem}\label{mainth2}
Given $(\omega_{0},\theta_{0})\in H^{2}(\mathbb{T}^{2},\mathbb{R})\times H^{3}(\mathbb{T}^{2},\mathbb{R})$, if $\tau_{max}<\infty$, then
\[ \int_{0}^{\tau_{max}}  \norm{\nabla u(t)}_{L^{\infty}}+\norm{\nabla \theta(t)}_{L^{\infty}} \ \diff t = \infty, \quad a.s.\] 
\end{theorem}
\begin{remark}
Theorem \ref{mainth} answers the question of the physical validity of the stochastic Boussinesq equations and provides another example which corroborates the method for introducing stochasticity presented in \cite{Principal} as physical. Theorem \ref{mainth2} could be used to check whether data from a given numerical simulation shows blow-up in finite time.
%\textcolor{red}{Theorem \ref{mainth} is fundamental because it establishes the physical validity of the stochastic Boussinesq equations, and provides another example which validates the method for introducing stochasticity presented in \cite{Principal}, as physical. Theorem \ref{mainth2} provides a very useful tool for investigating the problem of blow-up vs global existence, and in particular, can be used to check whether a numerical simulation evinces blow-up in finite time.}
\end{remark}
\begin{remark}
The Sobolev spaces in Theorem \ref{mainth} are not sharp. One could actually prove the same local existence and uniqueness result in the Sobolev spaces $H^{s-1}(\mathbb{T}^{2},\mathbb{R})\times H^{s}(\mathbb{T}^{2},\mathbb{R})$ for $s>2$. One of the knotty and technical points for extending this result to fractional indexes hinges on the Lie derivative cancellation inequalities stated in Proposition \ref{Liecancellations}. However, as shown in Appendix \ref{appendixA}, inequality \eqref{eq:cancellation2} is satisfied for a wide class of differential operators, which in particular covers the case of fractional differential operators.  
\end{remark}

Finally, we prove a result which is tremendously useful when deriving $H^s$ estimates for a general type of equations.
\begin{theorem}\label{generalcancellations2}
Let $\mathcal{Q}$ be a linear operator of first order with smooth bounded coefficients. Then for $f\in H^{2}(\mathbb{T}^{2},\mathbb{R})$ we have 
\begin{equation}\label{eq:cancellationgeneral1}
\langle \mathcal{Q}^2 f, f \rangle_{L^2} +  \langle \mathcal{Q} f, \mathcal{Q} f \rangle_{L^2} \lesssim ||f||_{L^2}^2.
\end{equation}
Moreover, if $f\in H^{2+k}(\mathbb{T}^{2},\mathbb{R})$, and $\mathcal{P}$ is a pseudodifferential operator of order $k,$
\begin{equation}\label{eq:cancellationgeneral2}
\langle \mathcal{P} \mathcal{Q}^2 f, \mathcal{P} f \rangle_{L^2} +  \langle \mathcal{P} \mathcal{Q} f, \mathcal{P} \mathcal{Q} f \rangle_{L^2} \lesssim ||f||_{H^k}^2, 
\end{equation}
for every $k\in[1,\infty)$.
\end{theorem}
\begin{remark}
Theorem \ref{generalcancellations2} turns out to be fundamental when performing a priori estimates in $H^{s}$ spaces for similar problems where the noise is given by
$$\sum_{i=1}^\infty \mathcal{Q}_i (u) \diff B_t^i,$$
where $Q_i,$ $i \in \mathbb{N},$ represent linear differential operators of first order, and the stochastic integral is in the Stratonovich sense. We comment further on this in the appendix.
\end{remark}
\begin{remark}[Paper notation]
We mention some aspects regarding the notation we employ along the article. Normally, we will denote the Sobolev $L^2-$based spaces by 
$H^s$(domain, target space). However, we will sometimes omit the domain and target space and just write $H^s,$ when these are clear from the context. We write $\diff \ $  to indicate an integrating differential in a domain, and we also use the notation $\diff f$ to denote the exterior differential of a form or a function, hoping it causes no confusion. $a\lesssim b$ means there exists $C$ such that $a\leq Cb$, where $C$ is a positive universal constant that may depend on fixed parameters, constant quantities, and the domain itself. Note also that this constant might differ from line to line. It is also important to remind that the condition ``almost surely" is not always indicated, since in some cases it is obvious from the context.
\end{remark}
\newpage
\section{Proofs of the main statements}\label{proofs:main:results}
\subsection{Local uniqueness}
To show the local uniqueness of solutions, we argue by contradiction. We will prove that any two different local solutions of the 2D stochastic Boussinesq equations (\ref{eq:vorticity:ito})-(\ref{eq:theta:ito}) defined up to a certain stopping time must be equal (almost surely). 
\begin{proposition}\label{prop:uniqueness}
Let $\tau$ be a stopping time and $\omega_1,\omega_2:\mathbb{T}^{2}\times \Xi \times [ 0,\tau] \to \mathbb{R}$, $\theta_1,\theta_2:\mathbb{T}^{2}\times \Xi\times [ 0,\tau] \to \mathbb{R}$ two solutions of \eqref{eq:vorticity:ito}-\eqref{eq:theta:ito} with the same initial data $(\omega_{0},\theta_{0})$  and continuous paths of class $C\left( [0,\tau]; H^{2}(\mathbb{T}^{2},\mathbb{R})\times H^{3}(\mathbb{T}^{2},\mathbb{R})\right)$. Then $\omega_{1}=\omega_{2}$ and $\theta_{1}=\theta_{2}$ on $[0,\tau]$.
\end{proposition}

\begin{proof}[Proof of Proposition \ref{prop:uniqueness}]
We know that
\begin{eqnarray*}
\diff \omega_{j} + \mathcal{L}_{u_{j}} \omega_{j} \diff t+\displaystyle\sum_{i=1}^{\infty}  \mathcal{L}_{\xi_{i}} \omega_{j} \diff B^{i}_{t}&=& \frac{1}{2} \displaystyle\sum_{i=1}^{\infty}  \mathcal{L}^{2}_{\xi_{i}}\omega_{j} \diff t+\partial_{x} \theta_{j} \diff t,  \\
 \diff \theta_{j} + \mathcal{L}_{u_j} \theta_{j} \diff t + \displaystyle\sum_{i=1}^{\infty}  \mathcal{L}_{\xi_{i}}\theta_{j} \diff B^{i}_{t}&=& \frac{1}{2} \displaystyle\sum_{i=1}^{\infty}  \mathcal{L}^{2}_{\xi_{i}}\theta_{j} \diff t,
\end{eqnarray*}
for $j=1,2$. Defining the differences $\tilde{\omega} = \omega_1-\omega_2,$ $\tilde{u} = u_1-u_2,$ and $\tilde{\theta} = \theta_1 - \theta_2,$ we infer that
\begin{eqnarray*}
\diff \tilde{\omega} + \mathcal{L}_{u_{1}} \tilde{\omega} \diff t+\mathcal{L}_{\tilde{u}}\omega_{2} \diff t+\displaystyle\sum_{i=1}^{\infty}  \mathcal{L}_{\xi_{i}} \tilde{\omega} \diff B^{i}_{t}&=& \frac{1}{2} \displaystyle\sum_{i=1}^{\infty}  \mathcal{L}^{2}_{\xi_{i}}\tilde{\omega} \diff t+\partial_{x} \tilde{\theta} \diff t,  \\
 \diff \tilde{\theta} + \mathcal{L}_{u_{1}} \tilde{\theta} \diff t+ \mathcal{L}_{\tilde{u}} \theta_2 \diff t + \displaystyle\sum_{i=1}^{\infty}  \mathcal{L}_{\xi_{i}}\tilde{\theta} \diff B^{i}_{t}&=& \frac{1}{2} \displaystyle\sum_{i=1}^{\infty}  \mathcal{L}^{2}_{\xi_{i}}\tilde{\theta} \diff t.
\end{eqnarray*}
Therefore, we can write (upon using It\^o's lemma for $f(x)=x^{2}$) that
\begin{eqnarray*}
\frac{1}{2} \diff \ \langle \tilde{\omega}, \tilde{\omega} \rangle_{L^2}&+& \langle \mathcal{L}_{u_{1}} \tilde{\omega},\tilde{\omega}\rangle_{L^{2}} \diff t + \langle \mathcal{L}_{\tilde{u}} \omega_{2},\tilde{\omega}\rangle_{L^{2}} \diff t + \displaystyle\sum_{i=1}^{\infty} \langle  \mathcal{L}_{\xi_{i}} \tilde{\omega}, \tilde{\omega} \rangle_{L^{2}}\diff B^{i}_{t} \\
&=& \frac{1}{2} \displaystyle\sum_{i=1}^{\infty} \langle \mathcal{L}^2_{\xi_i} \tilde{\omega}, \tilde{\omega} \rangle_{L^2} \diff t + \frac{1}{2} \displaystyle\sum_{i=1}^{\infty}  \langle \mathcal{L}_{\xi_i} \tilde{\omega}, \mathcal{L}_{\xi_i} \tilde{\omega} \rangle_{L^2} \diff t+ \langle \partial_{x} \tilde{\theta}, \tilde{\omega} \rangle_{L^2} \diff t,
\end{eqnarray*}
and
\begin{eqnarray*}
\frac{1}{2} \diff \ \langle \tilde{\theta}, \tilde{\theta} \rangle_{L^2}&+& \langle \mathcal{L}_{u_{1}} \tilde{\theta},\tilde{\theta}\rangle_{L^{2}} \diff t + \langle \mathcal{L}_{\tilde{u}} \theta_{2},\tilde{\theta}\rangle_{L^{2}} \diff t + \displaystyle\sum_{i=1}^{\infty} \langle  \mathcal{L}_{\xi_{i}} \tilde{\theta}, \tilde{\theta} \rangle_{L^{2}} \diff  B^{i}_{t} \\
&=& \frac{1}{2} \displaystyle\sum_{i=1}^{\infty} \langle \mathcal{L}^2_{\xi_i} \tilde{\theta}, \tilde{\theta} \rangle_{L^2} \diff t + \frac{1}{2} \displaystyle\sum_{i=1}^{\infty}  \langle \mathcal{L}_{\xi_i} \tilde{\theta}, \mathcal{L}_{\xi_i} \tilde{\theta} \rangle_{L^2} \diff t.
\end{eqnarray*}
Now, one can check that for scalar functions, $\mathcal{L}_{\xi} = - \mathcal{L}^*_{\xi}$ (see Proposition \ref{Liecancellations}). Hence we obtain
\[\frac{1}{2}\displaystyle\sum_{i=1}^{\infty} \langle \mathcal{L}^2_{\xi_i} \tilde{\omega}, \tilde{\omega} \rangle_{L^2} \diff t + \frac{1}{2}\displaystyle\sum_{i=1}^{\infty} \langle \mathcal{L}_{\xi_i} \tilde{\omega}, \mathcal{L}_{\xi_i} \tilde{\omega} \rangle_{L^2} \diff t = 0\]
and
\[\frac{1}{2} \displaystyle\sum_{i=1}^{\infty} \langle \mathcal{L}^2_{\xi_i} \tilde{\theta}, \tilde{\theta} \rangle_{L^2} \diff t + \frac{1}{2}\displaystyle\sum_{i=1}^{\infty} \langle \mathcal{L}_{\xi_i} \tilde{\theta}, \mathcal{L}_{\xi_i} \tilde{\theta} \rangle_{L^2} \diff t = 0. \]
We can estimate the nonlinear terms of each equation as follows
\[ \abs{\langle \mathcal{L}_{\tilde{u}} \omega_{2},\tilde{\omega}\rangle_{L^{2}} }=\left| \langle\tilde{u}\cdot \nabla\omega_{2},\tilde{\omega} \rangle_{L^2}\right| \lesssim \norm{\tilde{u}}_{L^4}\norm{\nabla\omega_{2}}_{L^4}\norm{\tilde{\omega}}_{L^2} \leq C \norm{\omega_{2}}_{H^{2}} \norm{\tilde{\omega}}^{2}_{L^{2}},
\]
where we have used the Gagliardo-Nirenberg inequality (\ref{Sob:ine1}) and the Biot-Savart property (\ref{BiotIne}). In a similar manner, we can estimate
\[| \langle \mathcal{L}_{\tilde{u}} \theta_{2},\tilde{\theta}\rangle_{L^{2}} |=|\langle\tilde{u}\cdot \nabla\theta_{2},\tilde{\theta} \rangle_{L^2}| \lesssim ||\nabla\theta_{2}||_{L^\infty} ||\tilde{u}||_{L^2} ||\tilde{\theta}||_{L^2} \lesssim ||\theta_{2}||_{H^{3}} (||\tilde{\omega}||^{2}_{L^2} + ||\tilde{\theta}||^{2}_{L^{2}}) \]
by applying Young's inequality and the Sobolev embedding (\ref{Sob:ine2}) in the last step. We also derive
\[ |\langle \partial_{x} \tilde{\theta},\tilde{\omega} \rangle_{L^2} | \lesssim   \left(||\nabla \theta_{1}||_{L^2}+||\nabla \theta_{2}||_{L^2}\right) ||\tilde{\omega}||_{L^{2}} . \]
The following two terms are zero due to the divergence-free condition:
\[ \langle \mathcal{L}_{u_1} \tilde{\omega},\tilde{\omega} \rangle_{L^{2}}=\langle u_{1}\cdot \nabla \tilde{\omega}, \tilde{\omega} \rangle_{L^{2}}= 0, \]
\[ \langle \mathcal{L}_{u_1} \tilde{\theta},\tilde{\theta} \rangle_{L^{2}}=\langle u_{1}\cdot \nabla \tilde{\theta}, \tilde{\theta} \rangle_{L^{2}}= 0.\]
Hence, we obtain
\begin{eqnarray*}
\diff \ \left( ||\tilde{\omega}||^{2}_{L^2}+||\tilde{\theta}||^{2}_{L^2}\right) &+& 2 \displaystyle\sum_{i=1}^{\infty} \left( \langle  \mathcal{L}_{\xi_{i}} \tilde{\omega}, \tilde{\omega} \rangle_{L^2}+\langle  \mathcal{L}_{\xi_{i}} \tilde{\theta}, \tilde{\theta} \rangle_{L^{2}} \right) \diff B^{i}_{t} \\
&\lesssim &  \left( 1+ ||\theta_{1}||_{H^{3}}+||\theta_2||_{H^3}+||\omega_{1}||_{H^2}+||\omega_{2}||_{H^{2}}\right)(||\tilde{\omega}||^{2}_{L^2}+||\tilde{\theta}||^{2}_{L^2}) \diff t.
\end{eqnarray*}
Now, by defining
$$Y_t = - \int_0^t \left( 1+ ||\theta_{1}||_{H^{3}} + ||\theta_2||_{H^3}+||\omega_{1}||_{H^2}+||\omega_{2}||_{H^{2}} \right) \diff  s, $$
one rewrites the equation in Gr\"onwall's type form
$$ \diff \ \left( \exp(Y_{t}) \left( ||\tilde{\omega}||^{2}_{L^2} + ||\tilde{\theta}||^{2}_{L^2} \right) \right) \lesssim  -\exp(Y_{t}) \displaystyle\sum_{i=1}^{\infty} \left( \langle \mathcal{L}_{\xi_i} \tilde{\omega}, \tilde{\omega}\rangle_{L^2} + \langle \mathcal{L}_{\xi_i} \tilde{\theta}, \tilde{\theta} \rangle_{L^2} \right)  \diff  B_t^i,$$
and therefore upon integration
$$  \exp(Y_{\bar{\tau}}) \left( ||\tilde{\omega}_{\bar{\tau}}||^{2}_{L^2} + ||\tilde{\theta}_{\bar{\tau}}||^{2}_{L^2}  \right) \lesssim  - \displaystyle\sum_{i=1}^{\infty}  \int_0^{\bar{\tau}} \exp(Y_s) \left( \langle \mathcal{L}_{\xi_i} \tilde{\omega}_s, \tilde{\omega}_s \rangle_{L^2} + \langle \mathcal{L}_{\xi_i} \tilde{\theta}_s, \tilde{\theta}_s \rangle_{L^2} \right)  \diff  B_s^i,$$
for any bounded stopping time $\bar{\tau} \leq \tau$. Hence
$$  \exp(Y_{t \wedge \tau}) \left( ||\tilde{\omega}_{t \wedge \tau}||^{2}_{L^2}+ ||\tilde{\theta}_{t \wedge \tau}||^{2}_{L^2} \right)$$
$$ \lesssim  - \displaystyle\sum_{i=1}^{\infty}  \int_0^{t \wedge \tau} \exp(Y_s) \left( \langle \mathcal{L}_{\xi_i} \tilde{\omega}_s, \tilde{\omega}_s \rangle_{L^2} + \langle \mathcal{L}_{\xi_i} \tilde{\theta}_s, \tilde{\theta}_s \rangle_{L^2} \right)  \diff  B_s^i.$$
By taking expectation, one can obtain
$$\mathbb{E} \left[ \exp(Y_{t \wedge \tau}) \left( ||\tilde{\omega}_{t \wedge \tau}||^{2}_{L^2} + ||\tilde{\theta}_{t \wedge \tau}||^{2}_{L^2}  \right) \right] $$
$$\lesssim  - \displaystyle\sum_{i=1}^{\infty}  \mathbb{E} \left[ \int_0^{t \wedge \tau}   \exp(Y_s) \left( \langle \mathcal{L}_{\xi_i} \tilde{\omega}_s, \tilde{\omega}_s \rangle_{L^2} + \langle \mathcal{L}_{\xi_i} \tilde{\theta}_s, \tilde{\theta}_s \rangle_{L^2} \right)  \diff  B_s^i \right] \leq 0.$$
We conclude
$$\mathbb{E} \left[ \exp(Y_{t \wedge \tau}) \left( ||\tilde{\omega}_{t \wedge \tau}||^{2}_{L^2} + ||\tilde{\theta}_{t \wedge \tau}||^{2}_{L^2}  \right) \right] = 0.$$
This implies that for every $t>0,$
$$  \exp(Y_{t \wedge \tau}) \left( ||\tilde{\omega}_{t \wedge \tau}||^{2}_{L^2}  + ||\tilde{\theta}_{t \wedge \tau}||^{2}_{L^2}  \right) = 0, \quad a.s.$$
Since $Y_{t \wedge \tau}$ is finite we obtain \[ ||\tilde{\omega}_{t \wedge \tau}||^{2}_{L^2} + ||\tilde{\theta}_{t \wedge \tau}||^{2}_{L^2}  = 0, \quad a.s. \]
and thus,
\[ \omega_{1, t\wedge\tau}= \omega_{2, t\wedge\tau}, \text{ \ and \  }  \theta_{1, t\wedge\tau}=\theta_{2,t\wedge\tau}, \quad a.s.  \]
\end{proof}
\subsection{Existence and uniqueness of maximal solutions}
Fix $r>0$ to be determined later and choose a $C^{\infty}$ non-increasing function $\eta_{r}:[0,\infty)\to [0,1]$ such that
\[ \eta_{r}(x)=
\begin{cases}
1, \ \text{for \ } |x| \leq r, \\
0, \ \text{for \ } |x| \geq 2r. 
\end{cases}
\]
Consider the following truncated stochastic Boussinesq equations
\begin{eqnarray}
\label{eq:vor:tru}  \diff \omega_{r} + \eta_{r}(||\nabla u||_{L^\infty})\mathcal{L}_{u_{r}} \omega_{r} \diff t+\displaystyle\sum_{i=1}^{\infty}  \mathcal{L}_{\xi_{i}} \omega_{r} \diff  B^{i}_{t}&=& \frac{1}{2} \displaystyle\sum_{i=1}^{\infty}  \mathcal{L}^{2}_{\xi_{i}}\omega_{r} \diff t+\partial_{x} \theta_{r} \diff t,  \\
\label{eq:theta:tru}
\diff  \theta_{r} +\eta_{r}(||\nabla \theta||_{{L^\infty}}) \mathcal{L}_{u_{r}} \theta_{r} \diff t + \displaystyle\sum_{i=1}^{\infty}  \mathcal{L}_{\xi_{i}}\theta_{r} \diff B^{i}_{t}&=& \frac{1}{2} \displaystyle\sum_{i=1}^{\infty}  \mathcal{L}^{2}_{\xi_{i}}\theta_{r} \diff t.
\end{eqnarray}
 
\begin{lemma}\label{globalimplieslocal}
Fix $r>0$ and $(\omega_{0},\theta_{0})\in H^{2}(\mathbb{T}^{2},\mathbb{R})\times H^{3}(\mathbb{T}^{2},\mathbb{R})$. Let $\omega_{r}:\mathbb{T}^{2}\times\Xi \times [0,\infty)\to \mathbb{R}$, $\theta_{r}:\mathbb{T}^{2}\times\Xi \times [0,\infty)\to \mathbb{R}$ be a global solution of \eqref{eq:vor:tru}-\eqref{eq:theta:tru} in $H^{2}(\mathbb{T}^{2},\mathbb{R})\times H^{3}(\mathbb{T}^{2},\mathbb{R})$. Consider the stopping time
\begin{equation}\label{tau:stopping}
\tau_{r}:= \displaystyle \inf \left\{t\geq 0 : ||\omega||_{H^{2}}+||\theta||_{H^{3}} \geq \dfrac{r}{C} \right\}, 
\end{equation}
where $C$ is chosen in such a way that the following inequality holds:
\[ ||\nabla u||_{L^\infty}+||\nabla \theta||_{L^\infty} \leq C(||\omega||_{H^{2}}+||\theta||_{H^{3}}). \]
Then, if we let $\omega:\mathbb{T}^{2} \times \Xi\times [0,\tau_{r}] \to \mathbb{R}$, $\theta:\mathbb{T}^{2} \times \Xi\times [0,\tau_{r}]\to \mathbb{R}$ be the restriction of $(\omega_{r},\theta_{r})$ to $\tau_r$, we have that $(\omega,\theta)$ is a local solution in $H^{2} \times H^{3}$ to the stochastic Boussinesq equations \eqref{eq:vorticity:ito}-\eqref{eq:theta:ito}.
\end{lemma}
\begin{proof}[Proof of Lemma \ref{globalimplieslocal}]
The proof is straightforward by construction.  For any $t\in[0,\tau_{r}]$ we have that
\[ ||\nabla u||_{L^\infty}+||\nabla \theta||_{L^\infty} \leq C(||\omega||_{H^{2}}+||\theta||_{H^{3}}) \leq r, \]
and therefore, $\eta_{r}(||\nabla u||_{\infty}) = \eta_{r}(||\nabla \theta||_{\infty})=1$.
\end{proof}\vspace{0.2cm}
Let us now state the result which will be the pillar for proving existence and uniqueness of maximal solutions of the stochastic Boussinesq equations \eqref{eq:vorticity:ito}-\eqref{eq:theta:ito}.  
\begin{proposition} \label{mainproposition}
Given $r>0$ and $(\omega_{0},\theta_{0})\in H^{2}(\mathbb{T}^{2},\mathbb{R})\times H^{3}(\mathbb{T}^{2},\mathbb{R})$, there exists a unique global solution $(\omega,\theta)$ in  $H^{2}\times H^{3}$ to the truncated equations \eqref{eq:vor:tru}-\eqref{eq:theta:tru}.
\end{proposition}
The rest of Section \ref{proofs:main:results} and Section \ref{comp:lim:proc} is devoted to proving Proposition \ref{mainproposition}. However, we first analyse how Proposition \ref{mainproposition} implies our main result stated in Theorem \ref{mainth}, showing existence and uniqueness of maximal solutions of equations \eqref{eq:vorticity:ito}-\eqref{eq:theta:ito}.
\begin{proof}[Proof of Theorem \ref{mainth}]
The proof is rather standard and merely constructive. By choosing $r=n\in\mathbb{N}$ in Lemma \ref{globalimplieslocal}, we have that $(\tau_{n},\omega_{n},\theta_{n})$ in $H^{2}(\mathbb{T}^{2},\mathbb{R})\times H^{3}(\mathbb{T}^{2},\mathbb{R})$ are local solutions of the stochastic Boussinesq equations \eqref{eq:vorticity:ito}-\eqref{eq:theta:ito}. Let us define $\tau_{max}:=\displaystyle \lim_{n\to\infty} \tau_{n}$, $ \omega:= \omega_{n},$ and $\theta:=\theta_{n}$ on $[0,\tau_{n})$. The statement that either $\tau_{max}=\infty$ or $\displaystyle\lim\sup_{s\nearrow \tau_{max}}\left(||\omega(s)||_{H^{2}}+ ||\theta(s)||_{H^{3}}\right)= \infty$ is easy to check. Indeed, assume that $\tau_{max}<\infty$. Then by continuity of $(\omega,\theta)$, there exists some stopping time $\tau'_{n}<\tau_{n}$ satisfying $|\tau'_{n}-\tau_{n}| \leq \frac{1}{n}$ and $||\omega(\tau'_{n})||_{H^{2}}+ ||\theta(\tau'_{n})||_{H^{3}}\geq \frac{n-1}{C}.$ Hence,
\[ \displaystyle\lim\sup_{s\nearrow \tau_{max}}\left(||\omega(s)||_{H^{2}}+ ||\theta(s)||_{H^{3}}\right) \geq  \displaystyle\lim\sup_{n\to\infty}\left(||\omega(\tau'_{n})||_{H^{2}}+ ||\theta(\tau'_{n})||_{H^{3}}\right)=\infty. \]
Now let us show that $(\tau_{max},\omega,\theta)$ is a maximal solution. Assume by contradiction, that there exists another solution $(\tau',\omega',\theta')$ such that $\omega'=\omega$ and $\theta=\theta'$ on  $[0,\tau'\wedge \tau_{max})$, with $\tau'> \tau_{max}$ on a set with positive probability. This is only possible if $\tau_{max}<\infty$ and therefore, on the set $\{\tau' > \tau_{max}\},$
\begin{eqnarray*}
 \infty= \displaystyle\lim\sup_{n\to\infty}\left(||\omega(\tau'_{n})||_{H^{2}}+ ||\theta(\tau'_{n})||_{H^{3}}\right) &=& \displaystyle\lim\sup_{n\to\infty}\left(||\omega'(\tau'_{n})||_{H^{2}}+ ||\theta'(\tau'_{n})||_{H^{3}}\right) \\
 &=& ||\omega'(\tau_{max})||_{H^{2}}+ ||\theta'(\tau_{max})||_{H^{3}}< \infty,
\end{eqnarray*}
since $\omega',\theta'$ are continuous on $[0,\tau')$, leading to a contradiction. We conclude that $\tau' \leq \tau_{max}$ and therefore, $(\tau_{max},\omega,\theta)$ is a maximal solution. Suppose now that there exists another maximal solution $(\tau,\omega^{\star},\theta^{\star})$ with the same initial conditions $(\omega_{0},\theta_{0})$. Due to the uniqueness result of Proposition \ref{prop:uniqueness}, one deduces that $\omega=\omega^{\star}$, $\theta=\theta^{\star}$ on $[0, \tau\wedge\tau_{max})$. By a similar argument as before, we cannot have $\tau_{max}<\tau$ on any set with positive measure, so $\tau \leq \tau_{max}$. From the third property of maximal solutions in Definition \ref{def:maximal}, we obtain $\tau=\tau_{max}$, and hence $\omega=\omega^{\star}, \theta=\theta^{\star}$ on $[0,\tau_{max})$.
\end{proof}

\subsection{Uniqueness of solutions of the truncated Boussinesq equations}
In this subsection we show uniqueness of solutions of the truncated Boussinesq equations \eqref{eq:vor:tru}-\eqref{eq:theta:tru}, and therefore we prove the uniqueness part of Proposition \ref{mainproposition}. The proof follows closely the same strategy as for the uniqueness result in Proposition \ref{prop:uniqueness}. However, since we have to perform $H^{2} $ and $H^3$ estimates, it is more involved. First of all, let  $\omega_{r,1},\omega_{r,2}:\mathbb{T}^{2}\times \Xi \times [0,\infty)\to \mathbb{R}$ and $\theta_{r,1},\theta_{r,2}:\mathbb{T}^{2}\times \Xi\times [0,\infty)\to \mathbb{R}$ be two global solutions of \eqref{eq:vor:tru}-\eqref{eq:theta:tru} in $H^{2}$ and $H^{3}$ respectively. Define the differences $\tilde{\omega} = \omega_{r,1}-\omega_{r,2},$ $\tilde{\theta} = \theta_{r,1} - \theta_{r,2},$ and $\tilde{u}=u_{r,1}-u_{r,2}$. We also define the truncation functions $\eta_r(||\nabla\tilde{u}||_{L^{\infty}})=\eta_r(||\nabla u_{1}||_{L^{\infty}})-\eta_{r}(||\nabla u_{2}||_{L^{\infty}})$ and $\eta_{r}(||\nabla\tilde{\theta}||_{L^{\infty}})=\eta_{r}(||\nabla \theta_{1}||_{L^{\infty}})-\eta_{r}(||\nabla \theta_{2}||_{L^{\infty}})$. To simplify notation, we omit the $r$ parameter dependence along the proof. First let us estimate the evolution of $||\tilde{\omega}||_{L^2}, ||\tilde{\theta}||_{L^2},$ and then we will estimate the evolution of $||\Delta\tilde{\omega}||_{L^2}, ||\Lambda^{3}\tilde{\theta}||_{L^2}$. We have
\makeatletter
 \def\@eqnnum{{\normalsize \normalcolor (\theequation)}}
  \makeatother
 { \small	
\begin{eqnarray}\label{eq:trunc:dif:vor}
 \diff \tilde{\omega} + \eta(||\nabla u_{1}||_{L^{\infty}})\mathcal{L}_{u_1}\omega_{1}\diff t-\eta(||\nabla u_{2}||_{L^{\infty}})\mathcal{L}_{u_2}\omega_{2}\diff t+ \displaystyle\sum_{i=1}^{\infty}  \mathcal{L}_{\xi_{i}} \tilde{\omega} \diff B^{i}_{t}&=& \frac{1}{2} \displaystyle\sum_{i=1}^{\infty}  \mathcal{L}^{2}_{\xi_{i}}\tilde{\omega} \diff t+\partial_{x} \tilde{\theta} \diff t,  \\
 \diff \tilde{\theta} + \eta(||\nabla \theta_{1}||_{L^{\infty}})\mathcal{L}_{u_1}\theta_{1}\diff t-\eta(||\nabla \theta_{2}||_{L^{\infty}})\mathcal{L}_{u_2}\theta_{2}\diff t + \displaystyle\sum_{i=1}^{\infty}  \mathcal{L}_{\xi_{i}}\tilde{\theta} \diff B^{i}_{t}&=& \frac{1}{2} \displaystyle\sum_{i=1}^{\infty}  \mathcal{L}^{2}_{\xi_{i}}\tilde{\theta} \diff t.
\end{eqnarray}}
Therefore, by using It\^o's lemma and the cancellation property \eqref{eq:cancellation1}, we obtain that
\begin{eqnarray*}
\frac{1}{2}  \diff  ||\tilde{\omega}||^{2}_{L^{2}} + \langle \eta(||\nabla u_{1}||_{L^{\infty}})\mathcal{L}_{u_1}\omega_{1}-\eta(||\nabla u_{2}||_{L^{\infty}})\mathcal{L}_{u_2}\omega_{2},\tilde{\omega} \rangle_{L^2} \ \diff t + \displaystyle\sum_{i=1}^{\infty} \langle  \mathcal{L}_{\xi_{i}} \tilde{\omega}, \tilde{\omega} \rangle_{L^{2}}\diff B^{i}_{t}&=& \langle \partial_{x} \tilde{\theta}, \tilde{\omega} \rangle_{L^2} \diff t,  \\
\frac{1}{2}  \diff  ||\tilde{\theta}||^{2}_{L^2} + \langle \eta(||\nabla \theta_{1}||_{L^{\infty}})\mathcal{L}_{u_1}\theta_{1}-\eta(||\nabla \theta_{2}||_{L^{\infty}})\mathcal{L}_{u_2}\theta_{2}, \tilde{\theta} \rangle_{L^2} \ \diff t + \displaystyle\sum_{i=1}^{\infty} \langle  \mathcal{L}_{\xi_{i}} \tilde{\theta}, \tilde{\theta} \rangle_{L^{2}} \diff B^{i}_{t}&=& 0. 
\end{eqnarray*}
To estimate the nonlinear terms we rewrite them as follows
\begin{eqnarray*}
\langle \eta(||\nabla u_{1}||_{L^{\infty}})\mathcal{L}_{u_1}\omega_{1}-\eta(||\nabla u_{2}||_{L^{\infty}})\mathcal{L}_{u_2}\omega_{2},\tilde{\omega} \rangle_{L^2}&=& \langle \left(\eta(||\nabla u_{1}||_{L^{\infty}})-\eta(||\nabla u_{2}||_{L^{\infty}})\right) \mathcal{L}_{u_{1}}\omega_{1}, \tilde{\omega} \rangle_{L^{2}} \\ &+& \langle \eta(||\nabla u_{2}||_{L^{\infty}})\mathcal{L}_{\tilde{u}}\omega_{1},\tilde{\omega}\rangle_{L^{2}}+  \langle \eta(||\nabla u_{2}||_{L^{\infty}})\mathcal{L}_{u_{2}}\tilde{\omega}, \tilde{\omega} \rangle_{L^{2}}, 
\end{eqnarray*}
and
\begin{eqnarray*}
\langle \eta(||\nabla \theta_{1}||_{L^{\infty}})\mathcal{L}_{u_1}\theta_{1}-\eta(||\nabla\theta_{2}||_{L^{\infty}})\mathcal{L}_{u_2}\theta_{2},\tilde{\theta} \rangle_{L^2} &=& \langle \left(\eta(||\nabla \theta_{1}||_{L^{\infty}})-\eta(||\nabla \theta_{2}||_{L^{\infty}})\right) \mathcal{L}_{u_{1}}\theta_{1}, \tilde{\theta} \rangle_{L^{2}} \\ &+& \langle \eta(||\nabla \theta_{2}||_{L^{\infty}})\mathcal{L}_{\tilde{u}}\theta_{1},\tilde{\theta}\rangle_{L^{2}}+ \langle \eta(||\nabla \theta_{2}||_{L^{\infty}})\mathcal{L}_{u_{2}}\tilde{\theta}, \tilde{\theta} \rangle_{L^{2}}. 
\end{eqnarray*}
Now notice that, on the set where $\tau_{2}\leq \tau_{1}$, the nonlinear terms are zero if $||\omega_{1}||_{H^2}+||\theta_{1}||_{H^{3}} \geq r/C$ (this is simply a direct consequence of the definition of the stopping times $\tau_{1},\tau_{2}$ provided in \eqref{tau:stopping}). Therefore, the nonlinear terms can be bounded as follows:
\begin{equation}\label{est:nonlinear:diff1}
  |\langle \eta(||\nabla u_{1}||_{L^{\infty}})\mathcal{L}_{u_1}\omega_{1}-\eta(||\nabla u_{2}||_{L^{\infty}})\mathcal{L}_{u_2}\omega_{2},\tilde{\omega} \rangle_{L^2} |   \lesssim  \left(1+||\omega_{1}||_{H^{2}} \right)||\tilde{\omega}||^{2}_{H^{2}},
\end{equation}
where  we have used that $\eta$ is Lipschitz, so
\begin{eqnarray*}
|\eta(||\nabla u_{1}||_{L^{\infty}})  -\eta(||\nabla u_{2}||_{L^{\infty}})| &\leq & C \left( ||\nabla u_{1}||_{L^\infty}-||\nabla u_{2}||_{L^\infty}\right) \\
&\leq & C ||\nabla\tilde{u}||_{L^\infty} \leq C ||\tilde{\omega}||_{H^{2}}.
\end{eqnarray*}
Similarly
\begin{equation}\label{est:nonlinear:diff2}
|\langle \eta(||\nabla \theta_{1}||_{L^{\infty}})\mathcal{L}_{u_1}\theta_{1}-\eta(||\nabla \theta_{2}||_{L^{\infty}})\mathcal{L}_{u_2}\theta_{2},\tilde{\theta} \rangle_{L^2}| \lesssim \left(1+||\theta_{1}||_{H^{3}}\right)\left( ||\tilde{\omega}||^{2}_{H^{2}}+||\tilde{\theta}||^{2}_{H^{3}}\right),
\end{equation}
where once again we have needed
\begin{eqnarray*}
|\eta(||\nabla \theta_{1}||_{L^{\infty}})-\eta(||\nabla \theta_{2}||_{L^{\infty}})| &\leq & C \left( ||\nabla \theta_{1}||_{L^\infty}-||\nabla \theta_{2}||_{L^\infty}\right) \\
&\leq & C ||\nabla\tilde{\theta}||_{L^\infty} \leq C ||\tilde{\theta}||_{H^{3}}.
\end{eqnarray*}
Putting \eqref{est:nonlinear:diff1} and \eqref{est:nonlinear:diff2} together, we deduce
\[ \diff ||\tilde{\omega}||^{2}_{L^{2}}+  \diff ||\tilde{\theta}||^{2}_{L^2}+\displaystyle\sum_{i=1}^{\infty} \langle  \mathcal{L}_{\xi_{i}} \tilde{\omega}, \tilde{\omega}\rangle_{L^{2}} \diff B^{i}_{t} +\displaystyle\sum_{i=1}^{\infty} \langle  \mathcal{L}_{\xi_{i}} \tilde{\theta}, \tilde{\theta} \rangle_{L^{2}} \ \diff B^{i}_{t} \lesssim \left(1+||\omega_{1}||_{H^{2}}+||\theta_{1}||_{H^3}\right) \left( ||\tilde{\omega}||^{2}_{H^2}+||\tilde{\theta}||^{2}_{H^{3}} \right) \diff t.\]
Also, by the same reason as before, we have that on the set $\{\tau_{1}\leq \tau_{2}\}$, the nonlinear terms are zero if $||\omega_{2}||_{H^{2}}+||\theta_{2}||_{H^{3}} \geq r/C$. Hence, by applying similar arguments to the ones above, one can conclude that
\[ \diff ||\tilde{\omega}||^{2}_{L^{2}}+  \diff ||\tilde{\theta}||^{2}_{L^2}+\displaystyle\sum_{i=1}^{\infty} \langle  \mathcal{L}_{\xi_{i}} \tilde{\omega}, \tilde{\omega}\rangle_{L^{2}} \diff B^{i}_{t} +\displaystyle\sum_{i=1}^{\infty} \langle  \mathcal{L}_{\xi_{i}} \tilde{\theta}, \tilde{\theta} \rangle_{L^{2}} \ \diff B^{i}_{t} \lesssim \left(1+||\omega_{2}||_{H^{2}}+||\theta_{2}||_{H^3}\right) \left( ||\tilde{\omega}||^{2}_{H^2}+||\tilde{\theta}||^{2}_{H^{3}} \right) \diff t.\]
Next, let us estimate the evolution of $||\Delta\tilde{\omega}||^{2}_{L^{2}}$. Taking the Laplace operator on equation \eqref{eq:trunc:dif:vor} and $L^{2}$ inner product against $\Delta\tilde{\omega},$ we obtain
\begin{eqnarray*} 
\diff ||\Delta\tilde{\omega}||^{2}_{L^2} &+& \langle \Delta\left( \eta(||\nabla u_{1}||_{L^{\infty}})\mathcal{L}_{u_1}\omega_{1}-\eta(||\nabla u_{2}||_{L^{\infty}})\mathcal{L}_{u_2}\omega_{2}\right), \Delta \tilde{\omega} \rangle_{L^2} \ \diff t +  \displaystyle\sum_{i=1}^{\infty} \langle \Delta \mathcal{L}_{\xi_{i}} \tilde{\omega}, \Delta\tilde{\omega}\rangle_{L^{2}} \diff B^{i}_{t} \\&=& \frac{1}{2} \displaystyle\sum_{i=1}^{\infty} \langle \Delta \mathcal{L}^2_{\xi_i} \tilde{\omega}, \Delta \tilde{\omega}\rangle_{L^2} \ \diff t  + \frac{1}{2}\displaystyle \sum_{i=1}^{\infty} \langle \Delta\mathcal{L}_{\xi_{i}}\tilde{\omega},\Delta \mathcal{L}_{\xi_{i}}\tilde{\omega} \rangle_{L^2} \ \diff t + \langle \Delta \partial_{x}\tilde{\theta}, \Delta \tilde{\omega} \rangle_{L^2} \ \diff t.
\end{eqnarray*}
First of all note that 
\[ | \langle \Delta \mathcal{L}_{\tilde{u}}\omega_{1}, \Delta \tilde{\omega} \rangle_{L^{2}}| \lesssim ||\omega_{1}||_{H^{2}} ||\tilde{\omega}||^{2}_{H^{2}},
\]
\[ |  \langle \Delta \mathcal{L}_{u_{2}}\tilde{\omega}, \Delta \tilde{\omega} \rangle_{L^{2}} | \lesssim ||\omega_{2}||_{H^{2}}||\tilde{\omega}||^{2}_{H^{2}}. \]
The same kind of estimates holds for the terms $\langle \Delta \mathcal{L}_{\tilde{u}}\omega_{2}, \Delta \tilde{\omega} \rangle_{L^{2}},\langle \Delta \mathcal{L}_{u_{1}}\tilde{\omega}, \Delta \tilde{\omega} \rangle_{L^{2}}$. Again, on the set $\{\tau_{2}\leq \tau_{1}\}$
\[
|\langle \left(\eta(||\nabla u_{1}||_{L^{\infty}})-\eta(||\nabla u_{2}||_{L^{\infty}})\right) \Delta \mathcal{L}_{u_{1}}\omega_{1}, \Delta \tilde{\omega} \rangle_{L^{2}}| \leq ||\omega_{1}||_{H^{2}}||\tilde{\omega}||^{2}_{H^{2}}. \]
In the same way, on the set  $\{\tau_{1}\leq \tau_{2}\}$
\[ |\langle \left(\eta(||\nabla u_{1}||_{L^{\infty}})-\eta(||\nabla u_{2}||_{L^{\infty}})\right) \Delta \mathcal{L}_{u_{2}}\omega_{2}, \Delta \tilde{\omega} \rangle_{L^{2}}| \leq ||\omega_{2}||_{H^{2}} ||\tilde{\omega}||^{2}_{H^{2}}. \]
Hence putting everything together we obtain:
\begin{eqnarray*} 
\diff ||\Delta\tilde{\omega}||^{2}_{L^2}  &+&  \displaystyle\sum_{i=1}^{\infty} \langle \Delta \mathcal{L}_{\xi_{i}} \tilde{\omega}, \Delta\tilde{\omega}\rangle_{L^{2}}\diff B^{i}_{t}\lesssim \left( 1+||\omega_{1}||_{H^{2}}+||\omega_{2}||_{H^{2}}\right) \left( ||\tilde{\omega}||^{2}_{H^{2}} +||\tilde{\theta}||^{2}_{H^{3}}\right) \diff t \\ &+&\frac{1}{2} \displaystyle\sum_{i=1}^{\infty} \langle \Delta \mathcal{L}^2_{\xi_i} \tilde{\omega}, \Delta \tilde{\omega}\rangle_{L^2} \ \diff t  + \frac{1}{2}\displaystyle \sum_{i=1}^{\infty} \langle \Delta\mathcal{L}_{\xi_{i}}\tilde{\omega},\Delta \mathcal{L}_{\xi_{}}\tilde{\omega} \rangle_{L^2} \ \diff t.
\end{eqnarray*}
The terms on the second line above can be bounded by invoking estimate \eqref{eq:cancellation2} with $k=2,$ thus obtaining
\begin{equation*} 
\diff ||\Delta\tilde{\omega}||^{2}_{H^2} +  \displaystyle\sum_{i=1}^{\infty} \langle \Delta \mathcal{L}_{\xi_{i}} \tilde{\omega}, \Delta\tilde{\omega}\rangle_{L^{2}} \diff B^{i}_{t}\lesssim \left( 1+||\omega_{1}||_{H^{2}}+||\omega_{2}||_{H^{2}}\right) \left( ||\tilde{\omega}||^{2}_{H^{2}}+||\tilde{\theta}||^{2}_{H^{3}}\right) \diff t. 
\end{equation*}
Finally, we derive the estimate for $||\Lambda^{3}\tilde{\theta}||^{2}_{L^{2}}$. To avoid repetition and simplify the exposition, we omit further details on this $H^{3}$ evolution computation. Without much effort, one realises that (take into account \eqref{eq:cancellation2} with $k=3$)
\[\diff ||\Lambda^{3}\tilde{\theta}||^{2}_{L^2} +  \displaystyle\sum_{i=1}^{\infty} \langle \Lambda^{3} \mathcal{L}_{\xi_{i}} \tilde{\theta}, \Lambda^{3}\tilde{\theta}\rangle_{L^{2}} \ \diff B^{i}_{t} \lesssim  \left( 1+||\omega_{1}||_{H^{2}}+||\omega_{2}||_{H^{2}}+||\theta_{1}||_{H^{3}}+||\theta_{2}||_{H^{3}}\right) \left( ||\tilde{\omega}||^{2}_{H^{2}}+||\tilde{\theta}||^{2}_{H^{3}}\right) \diff t. \]
Combining the estimates for $\tilde{\omega}$ and $\tilde{\theta}$ we conclude
\[ 
\begin{split}
\diff ||\Delta\tilde{\omega}||^{2}_{L^2} &+ \diff ||\Lambda^{3}\tilde{\theta}||^{2}_{L^2} + \displaystyle\sum_{i=1}^{\infty} \left( \langle \Delta \mathcal{L}_{\xi_{i}} \tilde{\omega}, \Delta\tilde{\omega}\rangle_{L^{2}}+\langle \Lambda^{3} \mathcal{L}_{\xi_{i}} \tilde{\theta}, \Lambda^{3}\tilde{\theta}\rangle_{L^{2}} \right) \diff B^{i}_{t}  \\
&\lesssim  \left( 1+||\omega_{1}||_{H^{2}}+||\omega_{2}||_{H^{2}}+||\theta_{1}||_{H^{3}}+||\theta_{2}||_{H^{3}}\right) \left( ||\tilde{\omega}||^{2}_{H^{2}}+||\tilde{\theta}||^{2}_{H^{3}}\right) \diff t.
\end{split}
\]
Now, it is enough to repeat the arguments in the proof of Proposition \ref{prop:uniqueness} to finish the proof.
\begin{remark}
It is also worth mentioning that in order to make full sense of some of the terms in the previous computation we would actually require more regularity. This could be made rigorous by introducing some mollifiers or a Fourier truncation type function $\mathcal{S}_{r}$ such that
\[ \widehat{\mathcal{S}{_r} f}(\xi)=1_{B_{r}}\widehat{f}(\xi), \]
where $B_{r}$ is a ball of radius $r$ centred at the origin and $1_{B_{r}}$ represents the indicator function. However, we do not carry out this argument here, since it would give rise to several lengthy and tedious computations that are quite standard.
\end{remark}
\subsection{Global solutions of the hyper-regularised truncated Boussinesq equations}\label{globalregular}
We are left to show global existence of solutions of the truncated Boussinesq equations \eqref{eq:vor:tru}-\eqref{eq:theta:tru}. To that purpose, let us consider the following hyper-regularised truncated equations
\begin{eqnarray}
\label{eq:vor:reg}  \diff  \omega^{\nu}_{r} + \eta_{r}(\omega^{\nu}_{r})\mathcal{L}_{u^{\nu}_{r}} \omega^{\nu}_{r} \diff t+\displaystyle\sum_{i=1}^{\infty}  \mathcal{L}_{\xi_{i}} \omega^{\nu}_{r} \diff  B^{i}_{t}&=& \nu \Delta^{5} \omega^{\nu}_{r} \diff t +\frac{1}{2} \displaystyle\sum_{i=1}^{\infty}  \mathcal{L}^{2}_{\xi_{i}}\omega^{\nu}_{r} \diff t+\partial_{x} \theta^{\nu}_{r} \diff t,  \\
\label{eq:theta:reg}
\diff  \theta^{\nu}_{r} +\eta_{r}(\theta^{\nu}_{r}) \mathcal{L}_{u^{\nu}_{r}} \theta^{\nu}_{r} \diff t + \displaystyle\sum_{i=1}^{\infty}  \mathcal{L}_{\xi_{i}}\theta^{\nu}_{r} \diff  B^{i}_{t}&=& \nu \Delta^{7} \theta^{\nu}_{r} \diff t + \frac{1}{2} \displaystyle\sum_{i=1}^{\infty}  \mathcal{L}^{2}_{\xi_{i}}\theta^{\nu}_{r} \diff t, \\
\label{eq:initia:reg}
\omega^{\nu}_{r}(x,0)=\omega_{0}&,&  \ \theta^{\nu}_{r}(x,0)=\theta_{0},
\end{eqnarray}
where $\nu>0$ and $\text{div} \ u^{\nu}_{r}= 0$. The above equation is understood in the mild sense (see \ref{Duhamel}), since the terms $\Delta^{5}\omega^{\nu}_{r}$, $\Delta^{7}\theta^{\nu}_{r}$ cannot be made sense of in the classical one. The artificial dissipations have been added in order to make the computations of the higher order terms rigorous (as for instance $\Delta \mathcal{L}^{2}_{\xi_{i}}\omega$ or $\Lambda^{3}\mathcal{L}^{2}_{\xi_{i}}\theta$). The rest of this section is devoted to proving the following result.
\begin{proposition}\label{Global:sol:reg:truncated}
For every $\nu,r >0$, and initial data $(\omega_{0},\theta_{0})\in H^{2}(\mathbb{T}^{2},\mathbb{R})\times H^{3}(\mathbb{T}^{2},\mathbb{R})$, there exists a unique global strong solution ($\omega^{\nu}_{r}, \theta^{\nu}_{r})$ in the class $L^{2}(\Xi ; C([0,T]; H^{2}(\mathbb{T}^{2},\mathbb{R}))\times H^{3}(\mathbb{T}^{2},\mathbb{R})))$, for all $T>0$. Moreover, their paths will gain extra regularity, namely $C([\delta,T]; H^{4}(\mathbb{T}^{2},\mathbb{R})\times H^{5}(\mathbb{T}^{2},\mathbb{R}))$, for every $T>\delta>0$. 
\end{proposition}
\begin{proof}[Proof of Proposition \ref{Global:sol:reg:truncated}]
The proof consists in first constructing a local solution via a fix point iteration argument. After showing a proper a priori estimate, we will be able to infer that this solution can be extended to a global one. For the sake of exposition clarity, the proof will be divided into several steps. We will also omit the subscripts $\nu$ and $r$ throughout the proof.   \\ \\
\textsl{Step 1: Formulation of the mild equation.} Given $U_{0}=(\omega_{0},\theta_{0})\in L^{2}(\Xi; H^{2}(\mathbb{T}^{2},\mathbb{R})\times H^{3}(\mathbb{T}^{2},\mathbb{R}))$, consider the mild truncated formulation equation 
\[U(t)= (\Upsilon U)(t),\]
where
\begin{equation*}
 (\Upsilon U)(t)=e^{tA}U_{0}-\int
_{0}^{t}e^{\left(  t-s\right)  A}(B_{\eta}(U)(s)-GU(s)) \
\diff  s +\int_{0}^{t}e^{\left(  t-s\right)  A} LU(s) \ \diff  s-
\displaystyle\sum_{i=1}^\infty \int_{0}^{t}e^{\left(  t-s\right)  A}R_{i}U(s) \ \diff  B_{s}^{i},
\end{equation*}
where $A:=(\nu\Delta^{5},\nu \Delta^{7}),$ $B_{\eta}U:=\left(\eta(||\nabla u||_{L^\infty})(u\cdot\nabla)\omega,\eta(||\nabla \theta||_{L^\infty})(u\cdot\nabla)\theta\right),$ $GU:=(\partial_{x}\theta,0)$, \\ $LU:= \left( \dfrac{1}{2}\displaystyle\sum_{i=1}^{\infty} \mathcal{L}^{2}_{\xi_{i}}\omega,\dfrac{1}{2}\displaystyle\sum_{i=1}^{\infty} \mathcal{L}^{2}_{\xi_{i}}\theta \right),$ and $R_{i}U:= \left( \mathcal{L}_{\xi_{i}}\omega,\mathcal{L}_{\xi_{i}}\theta \right)$. \\ \\
\textsl{Step 2: Construction of the local solution.} Let  $\mathcal{W}_{T}=L^{2}(\Xi; C([0,T];H^{2}(\mathbb{T}^{2},\mathbb{R})\times H^{3}(\mathbb{T}^{2},\mathbb{R})))$.  Since we want to prove that the map $\Upsilon$ is a contraction on the space $\mathcal{W}_{T}$, first we need to check that the map $\Upsilon$ applied to an element of $\mathcal{W}_{T}$ returns indeed an element of the same space. So let $U \in \mathcal{W}_T$. We check the different terms: \\
\begin{itemize}[leftmargin=*]
\item $e^{tA}$ is bounded in the spaces $H^{2}(\mathbb{T}^{2},\mathbb{R})$ and $H^{3}(\mathbb{T}^{2},\mathbb{R})$. Therefore $e^{tA}U_{0}$ is in $\mathcal{W}_{T}$. \\
\item The operator $B_{\eta}(U)$ is in $L^{2}(\Xi;C([0,T];L^{2}(\mathbb{T}^{2},\mathbb{R})\times L^{2}(\mathbb{T}^{2},\mathbb{R})))$ since the map $U\to B_{\eta}(U)$ from $H^{2}(\mathbb{T}^{2},\mathbb{R})\times H^{3}(\mathbb{T}^{2},\mathbb{R})$ to $L^{2}(\mathbb{T}^{2},\mathbb{R})\times L^{2}(\mathbb{T}^{2},\mathbb{R})$ is Lipschitz continuous. Hence by applying estimate \eqref{estimation:semigroup1}, we obtain that	
\[ \int_{0}^{t} e^{(t-s)A}B_{\eta}(U)(s) \ \diff s \in\mathcal{W_{T}}. \] \\
\item Since the operator $G$ satisfies
\[ ||GU||_{L^{2}} \lesssim ||U||_{H^1}, \]
by the same argument as above it is straightforward to infer that
\[ \int_{0}^{t} e^{(t-s)A} GU(s) \ \diff  s\in \mathcal{W}_{T}. \] \\
\item By condition \eqref{assumption1}, $\displaystyle\sum_{i=1}^{\infty} \mathcal{L}^{2}_{\xi_{i}}U\in L^{2}(\Xi; C([0,T];L^{2}(\mathbb{T}^{2},\mathbb{R})\times L^{2}(\mathbb{T}^{2},\mathbb{R})))$, so again by estimate (\ref{estimation:semigroup1}) we have
\[ \int_{0}^{t} e^{(t-s)A} LU(s) \ \diff s \in\mathcal{W}_{T}. \] \\
\item To manipulate the stochastic term  $R_{i}U$, we just need to combine \eqref{assumption2},
\[ \displaystyle\sum_{i=1}^{\infty} || \mathcal{L}_{\xi_{i}} U||^{2}_{L^{2}} \lesssim  ||U||^{2}_{H^{2}}, \]
with estimate (\ref{estimation:semigroup2}) to get 
\[ \sum_{i=1}^\infty \int_{0}^{t} e^{(t-s)A} R_{i}U(s) \ \diff  B^{i}_{s} \in\mathcal{W}_{T}. \]  \\
\end{itemize}
Checking  the Lipschitz continuity of the map $\Upsilon$ in $\mathcal{W}_{T}$ is a tedious but simple task which is left for the interested reader. \\ \\
\textsl{Step 3: From local to global. A priori estimate.} It is clear from the above construction that the lifespan of the solution $U$ depends only on the norm of the initial condition $U_{0}$ in $L^{2}(\Xi; H^{2}(\mathbb{T}^{2},\mathbb{R})\times H^{3}(\mathbb{T}^{2},\mathbb{R}))$ . Therefore, in order to extend the solution to a global one, it is sufficient to show that for a given $T>0$ and initial $U_{0}$, we have that
\begin{equation}\label{eq:assertion}
\sup_{t\in[0,T]} \mathbb{E}\left[ ||U(t)||^{2}_{H^{2}\times H^{3}} \right] \leq C(T).
\end{equation}
 If this holds true, we could patch together each local solution to cover any time interval. Hence, let us prove that this a priori estimate is indeed satisfied. Taking into account each term in the equation $U(t)=(\Upsilon U)(t)$, we derive
\begin{align*}
\mathbb{E}\left[  \left\Vert U(t)  \right\Vert _{H^{2}\times H^{3}}^{2}\right]
&  \lesssim \mathbb{E}\left[  \left\Vert e^{tA}U_{0}\right\Vert _{H^{2}\times H^{3}}
^{2}\right]  +\mathbb{E}\left[  \left\Vert \int_{0}^{t}e^{\left(  t-s\right)
A}\left(B_{\eta} (U)(s)-GU(s)\right) \ \diff  s\right\Vert _{H^{2}\times H^{3}}^{2}\right] \\
&  +\mathbb{E}\left[  \left\Vert \int_{0}^{t}e^{\left(  t-s\right)  A}LU(s) \ \diff  s\right\Vert _{H^{2}\times H^{3}}^{2}\right]  +\mathbb{E}\left[  \left\Vert \sum_{i=1}^\infty \int_{0}^{t}e^{\left(  t-s\right)  A} R_{i}U(s)  \ \diff  B^{i}_{s} \right\Vert _{H^{2}\times H^{3}}^{2}\right].
\end{align*}

Applying estimates (\ref{estimation:semigroup1})  and (\ref{estimation:semigroup2}), together with \eqref{assumption3}-\eqref{assumption2}, we obtain
\[
\mathbb{E}\left[  \left\Vert U(t)  \right\Vert _{H^{2}\times H^{3}}^{2}\right] 
\lesssim \mathbb{E}\left[  \left\Vert U_{0}\right \Vert _{H^{2}\times H^{3}}^{2}\right]
+ \mathbb{E}\left[  \int_{0}^{t} \displaystyle\max \left(\frac{1}{(t-s)^{2/5}},\frac{1}{(t-s)^{3/6}}\right) \left\Vert
U(s) \right\Vert _{H^{2}\times H^{3}}^{2} \diff  s\right].
\]
Assertion \eqref{eq:assertion} is concluded by performing a variation of a Gr\"onwall's type argument (cf. \cite{Henry}).\\ \\
\textsl{Step 4: Higher regularity}. By using general properties of the semigroup $e^{tA}$ (cf. \cite{Goldestein,Pazy}), one can prove that for positive times $T>\delta>0$, each term in the mild equation enjoys higher regularity, namely, $U\in L^{2}(\Xi; C([\delta,T];H^{4}(\mathbb{T}^2, \mathbb{R})\times H^{5}(\mathbb{T}^2, \mathbb{R}))),$ for every $T>\delta>0$. 
\end{proof}

\newpage
\section{Compactness argument}\label{comp:lim:proc}
In this section, we will show that the family of solutions $\{ (\omega_r^\nu, \theta_r^\nu) \}_{\nu>0}$ to the hyper-regularised equations \eqref{eq:vor:reg}-\eqref{eq:theta:reg} is compact in some sense, which will enable us to extract a subsequence converging strongly to a solution of the truncated stochastic Boussinesq system \eqref{eq:vor:tru}-\eqref{eq:theta:tru}. The central idea for proving this is to show compactness of the probability laws of this family. Consequently, we have to demonstrate that these laws are tight in a suitable metric space. Before proceeding any further, let us first give a glimpse of the steps we will be following.  \\ \\
Let $T>0$ and define the Polish space $E$ by
\begin{equation}\label{polish}
E = L^2 ([0,T];H^{\beta}(\mathbb{T}^{2},\mathbb{R}) \times H^{\beta}(\mathbb{T}^{2},\mathbb{R}) ) \cap C_w([0,T];H^{2}(\mathbb{T}^{2},\mathbb{R}) \times H^{3}(\mathbb{T}^{2},\mathbb{R})), \quad \beta > 1.
\end{equation}
Assume that the laws of $\{(\omega^\nu_r, \theta^\nu_r)\}_{\nu>0}$ are tight in $E.$ Then, Theorem \ref{prohorov} can be applied and one can extract a weakly converging subsequence, which we will denote by $\{(\omega^{1/n}_r, \theta^{1/n}_r)\}_{n \in \mathbb{N}}$ without loss of generality. However, we need a stronger type of convergence so that we can take limits properly. To this purpose we use the Gy\"ongy-Krylov Lemma \ref{strong}, which guarantees that the sequence $\{(\omega^{1/n}_r, \theta^{1/n}_r)\}_{n \in \mathbb{N}}$ converges in probability if some diagonal assumption holds. This latter assumption can be checked by applying Theorem \ref{rep} to the sequence $\{(\omega^{1/n}_r, \theta^{1/n}_r)\}_{n \in \mathbb{N}}$ and finding a probability space $(\Omega, \mathcal{A}, \mathbb{P})$ and a family of random measures $\{ ({\omega'}_r^{1/n}, {\theta'}_r^{1/n}) \}_{n \in \mathbb{N}},$ where the convergence is almost surely. Moreover, the laws of $\{ ({\omega'}_r^{1/n}, {\theta'}_r^{1/n}) \}_{n \in \mathbb{N}}$ can be easily shown to satisfy equations \eqref{eq:vor:reg}-\eqref{eq:theta:reg} weakly. Although we have not gone into extreme detail at some of the points in this paragraph, this is based on standard and classical stochastic partial differential equations arguments. A more exhaustive explanation can be found in \cite{CriHolFla},\cite{GlaVic},\cite{GlaZia}.\\ 

\subsubsection*{Passage to the limit.} We know that $\{ ({\omega'}_r^{1/n}, {\theta'}_r^{1/n}) \}_{n \in \mathbb{N}}$ converges almost surely in $E$ as $n \rightarrow \infty.$ Let us denote its limit by $({\omega'}_r, {\theta'}_r )$. We claim that $({\omega'}_r, {\theta'}_r )$ satisfies equations \eqref{eq:vor:tru}-\eqref{eq:theta:tru} in the weak sense explained in Definition \ref{localsol}. We will show this now and later we will take charge of proving this limit is actually regular enough so the equations are satisfied in the strong sense. We integrate against test functions and take limits as $n \rightarrow \infty$ on the equations
\makeatletter
 \def\@eqnnum{{\normalsize \normalcolor (\theequation)}}
  \makeatother
 { \small	
\begin{eqnarray*}
\diff  \omega^{1/n}_{r} + \eta_{r}(||\nabla u^{1/n}||_{L^{\infty}})\mathcal{L}_{u^{1/n}_{r}} \omega^{1/n}_{r} \diff t + \displaystyle\sum_{i=1}^{\infty}  \mathcal{L}_{\xi_{i}} \omega^{1/n}_{r} \diff  B^{i}_{t} &=& (1/n) \Delta^{5} \omega^{1/n}_{r} \diff t + \frac{1}{2} \displaystyle\sum_{i=1}^{\infty}  \mathcal{L}^{2}_{\xi_{i}}\omega^{1/n}_{r} \diff t + \partial_{x} \theta^{1/n}_{r} \diff t,  \\
 \diff  \theta^{1/n}_{r} + \eta_{r}(||\nabla \theta^{1/n}||_{L^{\infty}}) \mathcal{L}_{u^{1/n}_{r}} \theta^{1/n}_{r} \diff t + \displaystyle\sum_{i=1}^{\infty}  \mathcal{L}_{\xi_{i}}\theta^{1/n}_{r} \diff  B^{i}_{t} &=& (1/n) \Delta^{7} \theta^{1/n}_{r} \diff t + \frac{1}{2} \displaystyle\sum_{i=1}^{\infty}  \mathcal{L}^{2}_{\xi_{i}}\theta^{1/n}_{r} \diff t.
\end{eqnarray*}}
Let us analyse each term carefully: \\
\begin{itemize}[leftmargin=*]
\item It is straightforward to check that the term $\partial_{x}\theta^{1/n}_{r}$ converges to $\partial_{x}\theta_{r},$ weakly as $n \rightarrow \infty.$ \\
\item The viscosity term $(1/n)\Delta^{5}\omega^{1/n}_{r}$. Indeed, note that $(1/n) \Delta^{5} \omega^{1/n}_{r}$ converges weakly to zero (as $n  \rightarrow \infty$), by using pathwise convergence in  
\[L^2([0,T];L^2(\mathbb{T}^{2},\mathbb{R})) \supset L^2 ([0,T];H^{\beta}(\mathbb{T}^{2},\mathbb{R})) \] plus the embedding $H^{\beta}(\mathbb{T}^{2},\mathbb{R}) \subset C(\mathbb{T}^{2},\mathbb{R}),$ which implies that $\omega_r^{1/n}$ is equibounded, i.e. we can apply bounded convergence theorem. \\
%\textcolor{blue}{No entiendo porque tienes que usar el embedding con el $H^\beta$. Otra cosa importante que está mal en varios sitios mas abajo, es que la convergencia no es solo en espacio. En realidad quieres ver que :
%\[ \frac{1}{n} \int_{0}^{t} \left \langle \Delta^{5}\omega^{1/n}_{r},\phi \right\rangle_{L^2} \]
%converge.}
\item{The single Lie-derivative term $\mathcal{L}_{\xi_i} \omega_r^{1/n}.$ First note that
$$\left \langle \sum_{i=1}^\infty \mathcal{L}_{\xi_i} \omega_r^{1/n} - \sum_{i=1}^\infty \mathcal{L}_{\xi_i} \omega_r , \phi \right \rangle_{L^2}$$
$$\leq ||\nabla \phi ||_{L^\infty} \left( \sum_{i=1}^{\infty} ||\xi_i||_{L^2} \right)||\omega_r^{1/n} - \omega_r||_{L^2} \rightarrow 0,$$ 
by using integration by parts, the pathwise convergence in $L^2([0,T];L^2(\mathbb{T}^{2},\mathbb{R}))$ and assumption (\ref{assumption3}). Then apply bounded convergence theorem. \\}
\item{The double Lie-derivative terms $\mathcal{L}^2_{\xi_i} \omega_r^{1/n}$.} Note 
\[ \left \langle \sum_{i=1}^\infty \mathcal{L}^2_{\xi_i} \omega_r^{1/n} - \sum_{i=1}^\infty \mathcal{L}^2_{\xi_i} \omega_r, \phi \right \rangle_{L^2}
= \left \langle  \omega_r^{1/n} - \omega_r, \sum_{i=1}^\infty \mathcal{L}^2_{\xi_i} \phi \right \rangle_{L^2}, \]
since $\mathcal{L} = - \mathcal{L}^*$ (where $\mathcal{L}^*$ denotes the adjoint operator of $\mathcal{L}$ under the $L^2$ pairing). Use again convergence in $L^2([0,T];L^2(\mathbb{T}^{2},\mathbb{R}))$ and assumption (\ref{assumption3}). Then apply bounded convergence theorem. \\ 
\item{The nonlinear term. One needs to show
\begin{eqnarray} \label{weaknon}
\int_0^T \int_{\mathbb{T}^2} \eta_{r}(||\nabla u^{1/n}||_{L^{\infty}})  \mathcal{L}_{u^{1/n}_{r}} \omega^{1/n}_{r} \phi \diff V \diff s \rightarrow \int_0^T \int_{\mathbb{T}^2} \eta_{r}(||\nabla u||_{L^{\infty}}) \mathcal{L}_{u_{r}} \omega_{r} \phi \diff V \diff s, \quad a.s.
\end{eqnarray}
First note 
$$\left \langle \mathcal{L}_{u^{1/n}_{r}} \omega^{1/n}_{r} ,  \phi  \right \rangle_{L^2}  = - \left \langle  \omega^{1/n}_{r} ,  \mathcal{L}_{u^{1/n}_{r}} \phi \right \rangle_{L^2},$$
so
$$ \left |\left \langle \mathcal{L}_{u^{1/n}_{r}} \omega^{1/n}_{r} - \mathcal{L}_{u_{r}} \omega_{r},  \phi  \right \rangle_{L^2} \right |$$
$$ \leq \left |\left \langle \mathcal{L}_{(u^{1/n}_{r} - u_r)} \omega^{1/n}_{r},  \phi  \right \rangle_{L^2} \right |
+  \left |\left \langle \mathcal{L}_{u_r} ( \omega^{1/n}_{r} - \omega_{r}),  \phi  \right \rangle_{L^2} \right |$$
$$ = \left |\left \langle  \omega^{1/n}_{r},  \mathcal{L}_{(u^{1/n}_{r} - u_r)} \phi  \right \rangle_{L^2} \right |
+  \left |\left \langle   \omega^{1/n}_{r} - \omega_{r}, \mathcal{L}_{u_r} \phi  \right \rangle_{L^2} \right | \rightarrow 0,$$
as $n \rightarrow \infty,$ since convergence of $\omega^{1/n}_r$ in $L^2([0,T]; L^2(\mathbb{T}^{2},\mathbb{R}))$ to $\omega_r$ implies convergence of $u^{1/n}_r$ in $L^2([0,T];H^{1}(\mathbb{T}^{2},\mathbb{R}))$ to $u_r.$ By the bounded convergence theorem, this convergence can be extended to $L^1([0,T],\mathbb{R}).$ Also note that
$$\left| \left| \eta_{r}(||\nabla u^{1/n}||_{L^{\infty}})  \mathcal{L}_{u^{1/n}_{r}} \omega^{1/n}_{r} \right | \right|_{L^2} 
\leq \left | \left| \eta_{r}(||\nabla u^{1/n}||_{L^{\infty}}) \right| \right|_{L^\infty}  \left | \left |\mathcal{L}_{u^{1/n}_{r}} \omega^{1/n}_{r} \right | \right|_{L^2} . $$ 
Therefore, to show (\ref{weaknon}), we only need bounded convergence theorem. Let us verify the assumptions. \\ 
\begin{enumerate}
\item{Since $L^2 ([0,T];H^{\beta}(\mathbb{T}^{2},\mathbb{R})) \cap C_w([0,T];H^{2}(\mathbb{T}^{2},\mathbb{R})) \subset C_w([0,T]; C(\mathbb{T}^{2},\mathbb{R})),$ 
the sequence
$$\left \langle (\omega^{1/n}_{r} - \omega_{r}), \mathcal{L}_{(u^{1/n}_{r} - u_r)} \phi \right \rangle_{L^2} $$ 
is equibounded in $L^2([0,T], \mathbb{R}),$ $\mathbb{P}$ almost surely. \\}
\item{We are left to prove $\eta_{r}(||\nabla u^{1/n}||_{L^{\infty}}) \rightarrow \eta_{r}(||\nabla u||_{L^{\infty}}),$ as $n \rightarrow \infty,$ in probability with respect to time (note that these functions are bounded by one). For this convergence, use that $\omega_r^{1/n}$ converges to $\omega_r$ strongly with respect to time since $L^2([0,T]; H^{\beta}(\mathbb{T}^{2},\mathbb{R})) \subset L^2([0,T];C(\mathbb{T}^{2},\mathbb{R})).$ Moreover, since $\eta_r$ is bounded continuous, also $\eta_r (||\nabla u^{1/n}||_{L^{\infty}})$ converges in $L^2([0,T]; C(\mathbb{T}^{2},\mathbb{R})).$ Finally, strong convergence implies convergence in probability. Hence, bounded convergence can be applied and (\ref{weaknon}) is guaranteed.}
\end{enumerate}}
\end{itemize}\vspace{0.5cm}
Thus, we have shown that solutions of equation \eqref{eq:vor:reg} converge to solutions of equation \eqref{eq:vor:tru} in the weak limit on $E.$ By an almost identical procedure, we can also show that solutions of equation \eqref{eq:theta:reg} converge to solutions of its nonregularised version, namely equation \eqref{eq:theta:tru}, as $n \rightarrow \infty.$ As we pointed before, the first step to carry out all these arguments is to show that the laws of 
$\{(\omega_r^\nu, \theta_r^\nu)\}_{\nu>0}$ are tight in the Polish space $E$. The following result asserts that this holds true if some estimates are satisfied.
\begin{proposition} \label{compacity}
Assume that for some $\alpha >0$ and $N, N^{\star} \in \mathbb{N},$ there exist constants $C_1, C_2,$ such that
\begin{equation}\label{tight1}
\quad \mathbb{E} \left [ \sup_{t \in [0,T]} ||\omega_r^\nu (t)||^2_{H^{2}} \right ] +\mathbb{E} \left [ \sup_{t \in [0,T]} ||\theta_r^\nu(t)||^2_{H^{3}} \right ] \leq C_1, 
\end{equation}
\begin{equation}\label{tight2}
 \quad \mathbb{E} \left [ \int_0^T \int_0^T \frac{||\omega_r^\nu (t) - \omega_r^\nu(s)||^2_{H^{-N}}}{|t-s|^{1+2\alpha}} \diff t \diff s \right ] + \mathbb{E} \left [ \int_0^T \int_0^T \frac{||\theta_r^\nu (t) - \theta_r^\nu(s)||^2_{H^{-N^{\star}}}}{|t-s|^{1+2\alpha}} \diff t \diff s \right ] \leq C_2,
\end{equation}
 uniformly in $\nu$.  Then $\{(\omega_r^\nu, \theta_r^\nu)\}_{\nu>0}$ is tight in the Polish space $E$ defined in \eqref{polish}.
\end{proposition}
%Before starting with the proof the write down a useful Lemma.  \textcolor{blue}{Con respecto a este lemma, tengo serias dudas de que sea eso lo que necesitas probar. Creo que es innecesario. Yo  lo he entendido de otra forma. Me explico: te defines un espacio E' que va a ser un subepsacio compacto de E (tu Polish principal) (ahi es donde utilizas el Aubin-Lion). Luego como quieres ver que tu familia es tight en E para probar el teorema, te basta probar que $\mathbb{P}(\omega,\theta)>1-\epsilon$ en algún compacto (en este caso coges E'), que es lo que haces en la prueba del teorema. Si estas de acuerdo borra el Lemma.}
%\begin{lemma} \label{tightness}
%Tightness in the space 
%$$E' = C([0,T]; H^{2}(\mathbb{T}^{2},\mathbb{R}) \times H^{3}(\mathbb{T}^{2},\mathbb{R})) \cap H^{\alpha} ([0,T]; H^{-N}(\mathbb{T}^{2},\mathbb{R}) \times H^{-N}(\mathbb{T}^{2},\mathbb{R})). $$
%is equivalent to tightness in the space
%$$E = L^2 ([0,T];H^{\beta}(\mathbb{T}^{2},\mathbb{R}) \times H^{\beta}(\mathbb{T}^{2},\mathbb{R})) \cap C_w([0,T];H^{2}(\mathbb{T}^{2},\mathbb{R}) \times H^{3}(\mathbb{T}^{2},\mathbb{R})).$$
\begin{proof}[Proof of Proposition \ref{compacity}]
By applying Lemma \ref{compactlemma} with $X=H^{2}(\mathbb{T}^{2},\mathbb{R})\times H^{3}(\mathbb{T}^{2},\mathbb{R}), Y=H^{\beta}(\mathbb{T}^{2},\mathbb{R})\times H^{\beta}(\mathbb{T}^{2},\mathbb{R}), Z= H^{-N}(\mathbb{T}^{2},\mathbb{R})\times H^{-N^{\star}}(\mathbb{T}^{2},\mathbb{R}),$ $p=2,$ and $0<\alpha<1$, we deduce that 
\begin{equation}
 E_{0}:= L^{2}([0,T];X )\cap H^{\alpha}([0,T];Z) 
 \hookrightarrow  L^{2}([0,T]; Y)\subset E,
 \end{equation}
for any $1<\beta<2$. We choose this range for $\beta$ to obtain compactness of the embedding of $X$ into $Y$ and so that $Y\subset C(\mathbb{T}^{2},\mathbb{R})$.  Then the family of laws of $\{(\omega_r^\nu, \theta_r^\nu)\}_{\nu>0}$ is supported on the space $E_{0}$ by hypothesis \eqref{tight1}-\eqref{tight2}. All we need to show is that this family is tight in $E$. For $R_{1},\ldots,R_{6}>0$ define the set $B_{R_{1},\ldots,R_{6}}$ as
\begin{eqnarray*}
\bigg\{ (\omega_r^\nu, \theta_r^\nu): \sup_{t \in [0,T]} ||\omega_r^\nu(t)||_{H^{2}}^2  &\leq & R_1, \int_0^T ||\omega_r^\nu (t)||_{H^{-N}}^2 \diff t \leq R_2,  \int_0^T \int_0^T \frac{||\omega_r^\nu (t)-\omega_r^\nu (s)||^2_{H^{-N}}}{|t-s|^{1+2\alpha}} \diff t \diff s \leq R_3 , \\
 \sup_{t \in [0,T]} ||\theta_r^\nu (t)||_{H^{3}}^2 & \leq &   R_4, \int_0^T ||\theta_r^\nu(t)||_{H^{-N^{\star}}}^2 \diff t \leq R_5,  \int_0^T \int_0^T \frac{||\theta_r^\nu(t)-\theta_r^\nu(s)||^2_{H^{-N^{\star}}}}{|t-s|^{1+2\alpha}} \diff t \diff s \leq R_6  \bigg\},
\end{eqnarray*}
which is compact in $L^{2}([0,T]; H^{\beta}(\mathbb{T}^{2},\mathbb{R})\times H^{\beta}(\mathbb{T}^{2},\mathbb{R}))$ and therefore in $E$.   It suffices to prove that for every given $\epsilon,$ there exist $R_{1},\ldots, R_{6}>0$ such that  
\[ \mathbb{P} ((\omega_r^\nu, \theta_r^\nu) \in B^{c}_{R_{1},\ldots,R_{6}}) \leq \epsilon. \]
Invoking Lemma \ref{desigualdad}, we have that
\begin{equation*}
\mathbb{P} \left( \sup_{t \in [0,T]} ||\omega_r^\nu(t)||^2_{H^{2}} > R_1 \right) \leq  \frac{\mathbb{E} \left [ \displaystyle\sup_{t \in [0,T]} ||\omega_r^\nu (t)||^2_{H^{2}} \right] }{R_1} \leq \frac{C}{R_1},
\end{equation*}
and this is smaller than $\epsilon/6$ if we choose $R_{1}$ sufficiently large. Similarly, one can deduce that
\[ \mathbb{P} \left( \int_0^T \int_0^T \frac{||\omega_r^\nu(t)-\omega_r^\nu(s)||_{H^{-N}}}{|t-s|^{1+2\alpha}}  \diff t \diff s > R_3 \right) \leq \frac{C}{R_{3}} \leq \epsilon/6, \]
if $R_{3}$ is large enough. Since $||f(t)||^{2}_{H^{-N}} \lesssim ||f(t)||^{2}_{H^{2}},$ 
\begin{eqnarray*}
\mathbb{P} \left( \int_0^T ||\omega_r^\nu(t)||^2_{H^{-N}} \diff t > R_2  \right) &\leq & \mathbb{P} \left( T \displaystyle\sup_{t\in[0,T]} ||\omega_r^\nu(t)||^2_{H^{-N}} > R_2  \right) \\ 
&\lesssim &  \mathbb{P} \left( T \displaystyle\sup_{t\in[0,T]} ||\omega_r^\nu(t)||^2_{H^{2}} > R_2  \right)\leq \frac{C}{R_{2}},
\end{eqnarray*}
which can also be made arbitrarily small by choosing $R_2$ large enough. Identical procedure applies to the sets
\[ \mathbb{P} \left( \sup_{t \in [0,T]} ||\theta_r^\nu(t)||^2_{H^{3}} > R_4 \right),  \mathbb{P} \left( \int_0^T ||\theta_r^\nu(t)||^2_{H^{-N^{\star}}} \diff t > R_5  \right),  \mathbb{P} \left( \int_0^T \int_0^T \frac{||\theta_r^\nu(t)-\theta_r^\nu(s)||_{H^{-N^{\star}}}}{|t-s|^{1+2\alpha}} \diff t \diff s > R_6 \right),\]
thanks to Lemma \ref{desigualdad} and hypothesis \eqref{tight1}-\eqref{tight2}. We conclude that there exist large enough $R_{1},\ldots,R_{6}>0$ such that
\[  \mathbb{P} ((\omega_r^\nu, \theta_r^\nu) \in B^{c}_{R_{1},\ldots,R_{6}}) \leq \epsilon \]
as required.
\end{proof}
After having proven Proposition \ref{compacity}, we are left to show that its hypothesis \eqref{tight1}-\eqref{tight2} hold. First, we will demonstrate that condition \eqref{tight1} implies condition \eqref{tight2}. Indeed, since $\omega_r^\nu$ and $\theta_r^\nu$ satisfy equations \eqref{eq:vor:reg}-\eqref{eq:theta:reg}, respectively, we have that
\begin{eqnarray*}
\omega^{\nu}_{r}(t)-\omega^{\nu}_{r}(s) &+& \int_s^t \eta_{r}(||\nabla u||_{L^{\infty}})\mathcal{L}_{u^{\nu}_{r}} \omega^{\nu}_{r} (\gamma) \diff \gamma + \int_s^t \displaystyle\sum_{i=1}^{\infty}  \mathcal{L}_{\xi_{i}} \omega^{\nu}_{r} (\gamma) \diff B^{i}_{\gamma} \\
&=& \int_s^t \nu \Delta^{5} \omega^{\nu}_{r} (\gamma) \diff \gamma +\frac{1}{2} \int_s^t \displaystyle\sum_{i=1}^{\infty}  \mathcal{L}^{2}_{\xi_{i}}\omega^{\nu}_{r} (\gamma) \diff \gamma +\int_s^t \partial_{x} \theta^{\nu}_{r} (\gamma) \diff \gamma,
\end{eqnarray*}
and
\begin{eqnarray*}
\theta^{\nu}_{r}(t) - \theta^{\nu}_{r}(s) &+& \int_s^t \eta_{r}(||\nabla\theta||_{L^{\infty}}) \mathcal{L}_{u^{\nu}_{r}} \theta^{\nu}_{r} (\gamma) \diff \gamma + \int_s^t \displaystyle\sum_{i=1}^{\infty}  \mathcal{L}_{\xi_{i}}\theta^{\nu}_{r} (\gamma) \diff B^{i}_{\gamma} \\
&=& \int_s^t \nu \Delta^{7} \theta^{\nu}_{r} (\gamma) \diff \gamma + \frac{1}{2} \int_s^t  \displaystyle\sum_{i=1}^{\infty}  \mathcal{L}^{2}_{\xi_{i}}\theta^{\nu}_{r} (\gamma) \diff \gamma .
\end{eqnarray*}
Hence by applying Minkowski's and Jensen's inequality, we obtain that
%Hence by applying Hlder \textcolor{blue}{Juraria que es Minkowski inequality, no Hoelder} \textcolor{red}{No estoy de acuerdo} \textcolor{blue}{Puede que me equivoque creo que utilizas Minskowski en la primera y Jensen inequality en la segunda xD Adjunto pruebas! \footnote{$ \left[\int_{S_2}\left|\int_{S_1}F(x,y)\, \mu_1(\mathrm{d}x)\right|^p \mu_2(\mathrm{d}y)\right]^{\frac{1}{p}} \le \int_{S_1}\left(\int_{S_2}|F(x,y)|^p\,\mu_2(\mathrm{d}y)\right)^{\frac{1}{p}}\mu_1(\mathrm{d}x)$ and $\varphi\left(\frac{1}{b-a}\int_a^b  f(x)\, dx\right) \le \frac{1}{b-a} \int_a^b \varphi(f(x)) \,dx.$ for $\varphi=f^{2}$}}
%on the time parameter, we obtain
%$$\mathbb{E} \left[ \left|\left|\int_s^t f(\gamma) \diff \gamma \right|\right|^2_\alpha \right] \leq \mathbb{E} \left[ \left( \int_s^t ||f(\gamma)||_\alpha \diff \gamma \right)^2 \right] $$
%$$\leq 
%\mathbb{E} \left[ (t-s)  \int_s^t ||f(\gamma)||_\alpha^2 \diff \gamma  \right] = (t-s) \int_s^t \mathbb{E} \left[ ||f(\gamma)||_\alpha^2  \right] \diff \gamma.$$
%\textcolor{red}{La desigualdad de arriba la meteria en un lema en la intro}
\begin{eqnarray*}
\mathbb{E} \left[ ||\omega^{\nu}_{r}(t)-\omega^{\nu}_{r}(s)||^2_{H^{-N}} \right]
& \lesssim & (t-s)   \int_s^t \mathbb{E} \left[\eta_{r}(||\nabla u||_{L^{\infty}})||\mathcal{L}_{u^{\nu}_{r}} \omega^{\nu}_{r}(\gamma)||^2_{H^{-N}}\right] \diff \gamma \\ &+&
(t-s)\int_s^t \mathbb{E} \left[ ||\nu \Delta^{5} \omega^{\nu}_{r} (\gamma)||^2_{H^{-N}} \right] \diff \gamma \\
&+& (t-s) \int_s^t \displaystyle\sum_{i=1}^{\infty} \mathbb{E} [|| \mathcal{L}^{2}_{\xi_{i}}\omega^{\nu}_{r} (\gamma)||^2_{H^{-N}}] \diff \gamma
 +(t-s) \int_s^t \mathbb{E} \left[||\partial_{x} \theta^{\nu}_{r} (\gamma)||^2_{H^{-N}}\right] \diff \gamma \\
 &+& \mathbb{E} \left[ \left| \left| \int_s^t \displaystyle\sum_{i=1}^{\infty}  \mathcal{L}_{\xi_{i}} \omega^{\nu}_{r} (\gamma) \diff B^{i}_{\gamma} \right| \right|^2_{H^{-N}} \right].
\end{eqnarray*}
Therefore, by using that $\norm{f}_{H^{-N}} \lesssim \norm{f}_{L^{2}}$, we have that
\begin{eqnarray*}
\mathbb{E} \left[ ||\omega^{\nu}_{r}(t)-\omega^{\nu}_{r}(s)||^2_{H^{-N}} \right]
& \lesssim &  (t-s) \int_s^t \mathbb{E} \left[\eta_{r}(||\nabla u||_{L^{\infty}})||\mathcal{L}_{u^{\nu}_{r}} \omega^{\nu}_{r}(\gamma)||^2_{H^{-N}}\right] \diff \gamma \\ &+& 
(t-s)\int_s^t \mathbb{E} \left[||\nu \Delta^{5} \omega^{\nu}_{r} (\gamma)||^2_{H^{-N}}\right] \diff \gamma \\
&+& (t-s) \int_s^t \displaystyle\sum_{i=1}^{\infty} \mathbb{E} \left[|| \mathcal{L}^{2}_{\xi_{i}}\omega^{\nu}_{r} (\gamma)||^2_{L^2}\right] \diff \gamma
 +(t-s) \int_s^t \mathbb{E} \left[||\partial_{x} \theta^{\nu}_{r} (\gamma)||^2_{L^2}\right] \diff \gamma \\
&+& \mathbb{E} \left[ \left| \left| \int_s^t \displaystyle\sum_{i=1}^{\infty}  \mathcal{L}_{\xi_{i}} \omega^{\nu}_{r} (\gamma) \diff B^{i}_{\gamma} \right| \right|^2_{L^2} \right].
\end{eqnarray*}
Mimicking the same estimates, we get
\begin{eqnarray*}
\mathbb{E} \left[ ||\theta^{\nu}_{r}(t)-\theta^{\nu}_{r}(s)||^2_{H^{-N^{\star}}} \right]
&\lesssim & (t-s) \int_s^t \mathbb{E} \left[\eta_{r}(||\nabla\theta||_{L^{\infty}})||\mathcal{L}_{u^{\nu}_{r}} \theta^{\nu}_{r}(\gamma)||^2_{H^{-N^{\star}}}\right] \diff \gamma  \\ &+& 
(t-s)\int_s^t \mathbb{E} \left[||\nu \Delta^{7} \theta^{\nu}_{r} (\gamma)||^2_{H^{-N^{\star}}}\right] \diff \gamma \\
&+& (t-s) \int_s^t \displaystyle\sum_{i=1}^{\infty} \mathbb{E} \left[|| \mathcal{L}^{2}_{\xi_{i}}\theta^{\nu}_{r} (\gamma)||^2_{L^2}\right] \diff \gamma \\
 &+& \mathbb{E} \left[ \left| \left| \int_s^t \displaystyle\sum_{i=1}^{\infty}  \mathcal{L}_{\xi_{i}} \theta^{\nu}_{r} (\gamma) \diff B^{i}_{\gamma} \right| \right|^2_{L^2} \right].
\end{eqnarray*}
In order to calculate each term, we use the following bounds:
\begin{lemma}\label{lemma:est:triv} The following estimates hold true
\begin{eqnarray}
||\mathcal{L}_{u^{\nu}_{r}} \omega^{\nu}_{r}||_{H^{-8}} & \lesssim &||\omega_r^\nu||_{L^\infty} ||\omega_r^\nu||_{H^{2}}, \label{est:straight1}\\
||\Delta^5 \omega_r^\nu||_{H^{-8}} & \lesssim & ||\omega_r^\nu||_{H^2}, \label{est:straight2} \\
||\mathcal{L}_{u^{\nu}_{r}} \theta^{\nu}_{r}||_{H^{-11}} & \lesssim &||\omega_r^\nu||_{L^\infty} ||\omega_r^\nu||_{H^{3}},  \label{est:straight3}\\
||\Delta^7 \theta_r^\nu||_{H^{-11}} & \lesssim & ||\theta_r^\nu||_{H^3}.  \label{est:straight4}
\end{eqnarray}
\end{lemma}
\begin{proof}[Proof of Lemma \ref{lemma:est:triv}]
The statement can be checked by direct computations.
\end{proof}
Using \eqref{est:straight1} we have that
\[ \eta_{r}(||\nabla u||_{L^{\infty}}) \norm{\mathcal{L}_{u^{\nu}_{r}} \omega^{\nu}_{r}}^{2}_{H^{-8}} \lesssim \norm{\omega^{\nu}_{r}}^{2}_{H^{2}}, \]
and therefore,
\begin{equation}\label{est:term1}
 \int_s^t \mathbb{E} \left[\eta_{r}(||\nabla u ||_{L^{\infty}})||\mathcal{L}_{u^{\nu}_{r}} \omega^{\nu}_{r}(\gamma)||^2_{H^{-8}}\right] \diff \gamma \lesssim  \int_{s}^{t} \mathbb{E}\left[\norm{\omega^{\nu}_{r}(\gamma)}^{2}_{H^{2}}\right] \diff \gamma \leq C,
\end{equation}
where we have used \eqref{tight1}. In the same way, by using \eqref{tight1} and \eqref{est:straight2}, it is easy to infer that
\begin{equation}\label{est:term2}
\int_s^t \mathbb{E} \left[||\nu \Delta^{5} \omega^{\nu}_{r} (\gamma)||^2_{H^{-8}}\right] \diff \gamma \lesssim  \int_s^t \mathbb{E}\left[||\omega_r^\nu(\gamma)||^{2}_{H^2}\right] \diff \gamma \leq C. 
\end{equation}
Applying \eqref{tight1}, we get 
\begin{equation}\label{est:term4}
\int_s^t \mathbb{E} \left[|| \partial_{x} \theta^{\nu}_{r} (\gamma)||^2_{L^{2}}\right] \diff \gamma \lesssim  \int_s^t \mathbb{E}\left[||\theta_r^\nu(\gamma)||^{2}_{H^3}\right] \diff \gamma \leq C. 
\end{equation}
Now using \eqref{assumption1}, we have that
\[  \norm{\displaystyle\sum_{i=1}^{\infty}\mathcal{L}^{2}_{\xi_{i}}\omega^{\nu}_{r} (\gamma)}^{2}_{L^{2}} \lesssim \norm{\omega^{\nu}_{r}(\gamma)}^{2}_{H^{2}}, \]
and hence by \eqref{tight1},
\begin{equation}\label{est:term3}
\int_s^t \displaystyle\sum_{i=1}^{\infty} \mathbb{E} \left[ \norm{\mathcal{L}^{2}_{\xi_{i}}\omega^{\nu}_{r} (\gamma)}^2_{L^2}\right] \diff \gamma \lesssim \int_{s}^{t}\mathbb{E}\left[ \norm{\omega^{\nu}_{r}(\gamma)}^{2}_{H^{2}}\right] \diff \gamma \leq C.
\end{equation}
Finally, the stochastic term can be controlled by using \eqref{assumption2},
\begin{eqnarray}\label{est:term5}
\mathbb{E} \left[ \left| \left| \int_s^t \displaystyle\sum_{i=1}^{\infty}  \mathcal{L}_{\xi_{i}} \omega^{\nu}_{r} (\gamma) \diff B^{i}_{\gamma} \right| \right|^2_{L^2} \right]
&=& \displaystyle\sum_{i=1}^{\infty} \int_s^t   \mathbb{E} \left[ ||\mathcal{L}_{\xi_{i}} \omega^{\nu}_{r} (\gamma) ||^2_{L^2} \right] \diff \gamma   \nonumber \\ 
&\lesssim & \int_s^t   \mathbb{E} \left[ || \omega^{\nu}_{r} (\gamma) ||^2_{H^{2}} \right] \diff \gamma \leq C. 
\end{eqnarray}
Combining estimates \eqref{est:term1}-\eqref{est:term5}, we deduce that
\[ \mathbb{E} \left[ \norm{\omega^{\nu}_{r}(t)-\omega^{\nu}_{r}(s)}^2_{H^{-8}} \right] \leq C(t-s).\]
Likewise, by using \eqref{est:straight3}-\eqref{est:straight4} we can conclude that
$$\mathbb{E} \left[ \norm{\theta^{\nu}_{r}(t)-\theta^{\nu}_{r}(s)}^2_{H^{-11}} \right] \leq C(t-s).$$
Hence for $0<\alpha < 1/2,$ 
\begin{eqnarray*}
\mathbb{E} \left[ \int_0^T \int_0^T \frac{||\omega_r^\nu (t) - \omega_r^\nu(s)||^2_{H^{-8}}}{|t-s|^{1+2\alpha}} \diff t \diff s \right] &+& \mathbb{E} \left[ \int_0^T \int_0^T \frac{||\theta_r^\nu (t) - \theta_r^\nu(s)||^2_{H^{-11}}}{|t-s|^{1+2\alpha}} \diff t \diff s \right] \\
&\leq &\mathbb{E} \left[ \int_0^T \int_0^T \frac{C}{|t-s|^{2 \alpha}} \diff t \diff s \right] \leq C_{1}(T),
\end{eqnarray*}
as required.
\begin{remark}
Notice that we have needed to use the spaces $H^{-8}$ and $H^{-11}$ in order to deal with the dissipative terms.
\end{remark}
We are left to prove that hypothesis \eqref{tight1} holds true. This fact is collected in the following lemma.
\begin{lemma}\label{lemma:est:apriori:full}
There exists a universal constant $C$ such that
\begin{equation}\label{est:apriori:full}
\mathbb{E} \left [ \sup_{t \in [0,T]} ||\omega_r^\nu (t)||^2_{H^{2}} \right ] +\mathbb{E} \left [ \sup_{t \in [0,T]} ||\theta_r^\nu(t)||^2_{H^{3}} \right ] \leq C.
\end{equation}
\end{lemma}
\begin{proof}[Proof of Lemma \ref{lemma:est:apriori:full}]
Taking two derivatives on equation \eqref{eq:vor:reg}, we have that
\begin{eqnarray*}
\Delta \omega^{\nu}_{r} (t) &=& \Delta \omega^{\nu}_{r} (0)  - \int_0^t \eta_{r}(||\nabla u||_{L^{\infty}}) \Delta \mathcal{L}_{u^{\nu}_{r}} \omega^{\nu}_{r}(s) \diff s - \int_0^t \displaystyle\sum_{i=1}^{\infty} \Delta  \mathcal{L}_{\xi_{i}} \omega^{\nu}_{r} (s) \diff B^{i}_{s} \\ &+& \int_0^t \nu \Delta^{6} \omega^{\nu}_{r} (s) \diff s 
+  \frac{1}{2} \int_0^t \displaystyle\sum_{i=1}^{\infty} \Delta \mathcal{L}^{2}_{\xi_{i}} \omega^{\nu}_{r}  (s) \diff s + \int_0^t \Delta \partial_{x} \theta^{\nu}_{r} (s) \diff s.  \\
\end{eqnarray*}
Dotting against $\Delta \omega^{\nu}_{r},$ applying It\^o's formula, and integrating over $\mathbb{T}^2,$ %for $f(X,t) = X^2,$ 
one obtains 
\begin{eqnarray*}
\frac{1}{2} \int_{\mathbb{T}^{2}} |\Delta \omega^{\nu}_{r} (t)|^2 \ \diff V &=& \frac{1}{2}\int_{\mathbb{T}^{2}} |\Delta \omega^{\nu}_{r} (0)|^2 \ \diff V  - \int_0^t \langle \eta_{r}(||\nabla u||_{L^{\infty}}) \Delta \mathcal{L}_{u^{\nu}_{r}} \omega^{\nu}_{r}(s), \Delta \omega^\nu_r (s) \rangle_{L^2} \diff s \\ &-& \displaystyle\sum_{i=1}^{\infty} \int_0^t  \langle \Delta  \mathcal{L}_{\xi_{i}} \omega^{\nu}_{r} (s), \Delta \omega_r^\nu (s) \rangle_{L^2} \diff B^{i}_{s} 
 +  \int_0^t \langle \nu \Delta^{6} \omega^{\nu}_{r} (s), \Delta \omega^\nu_r (s) \rangle_{L^2} \diff s  \\ 
&+&  \int_0^t \langle \Delta \partial_{x} \theta^{\nu}_{r} (s), \Delta \omega^\nu_r (s) \rangle_{L^2} \diff s 
+ \frac{1}{2} \displaystyle\sum_{i=1}^{\infty}\int_0^t  \langle \Delta \mathcal{L}^{2}_{\xi_{i}} \omega^{\nu}_{r} (s), \Delta \omega^\nu_r (s) \rangle_{L^2} \diff s  \\
&+&\frac{1}{2} \sum_{i=1}^\infty \int_0^t \langle \Delta \mathcal{L}_{\xi_i} \omega^\nu_r(s), \Delta \mathcal{L}_{\xi_i} \omega^\nu_r (s) \rangle_{L^2} \diff s.
\end{eqnarray*}
Let us estimate term by term: \\
\begin{itemize}[leftmargin=*]
\item The dissipative term $\langle \nu \Delta^{6} \omega^{\nu}_{r} , \Delta \omega^\nu_r  \rangle_{L^2}$ cannot be used to absorb any other singular terms, since we want our estimates to be independent of $\nu$. Applying integration by parts
\[ \langle \nu \Delta^{6} \omega^{\nu}_{r} , \Delta \omega^\nu_r \rangle_{L^2} = - \nu \int_{\mathbb{T}^2} |\Delta^{7/2}\omega^{\nu}_{r}|^2 \diff V, \]
so we see that the dissipative term has the correct sign, and we can just drop it. \\
\item The sum of the last two terms  
$$\frac{1}{2} \displaystyle\sum_{i=1}^\infty \langle \Delta \mathcal{L}_{\xi_i} \omega^\nu_r, \Delta \mathcal{L}_{\xi_i} \omega^\nu_r \rangle_{L^2} + \frac{1}{2} \displaystyle\sum_{i=1}^{\infty} \langle \Delta \mathcal{L}^{2}_{\xi_{i}} \omega^{\nu}_{r}, \Delta \omega^\nu_r \rangle_{L^2} $$ 
can be bounded thanks to \eqref{eq:cancellation2} as
\[   \frac{1}{2} \sum_{i=1}^\infty \langle \Delta \mathcal{L}_{\xi_i} \omega^\nu_r, \Delta \mathcal{L}_{\xi_i} \omega^\nu_r \rangle_{L^2} 
+ \frac{1}{2} \displaystyle\sum_{i=1}^{\infty} \langle \Delta \mathcal{L}^{2}_{\xi_{i}} \omega^{\nu}_{r}, \Delta \omega^\nu_r \rangle_{L^2} \lesssim ||\omega^{\nu}_{r} ||^{2}_{H^{2}}.\] \\
\item The $H^{2}$ estimate for the nonlinear term is quite standard. It is easy to show that
\begin{equation}\label{est:nonlinear:u}
 \left| \int_{\mathbb{T}^2}  \Delta \mathcal{L}_{u^{\nu}_{r}} \omega^{\nu}_{r} \Delta \omega^\nu_r \ \diff V \right| \lesssim ||\nabla u^{\nu}_{r} ||_{L^\infty}||\omega^{\nu}_{r}||^2_{H^{2}}. 
\end{equation}
Indeed, by Leibniz chain rule we have that
\begin{eqnarray*}
\int_{\mathbb{T}^2}  \Delta \mathcal{L}_{u^{\nu}_{r}} \omega^{\nu}_{r} \Delta \omega^\nu_r \ \diff V &=&  \int_{\mathbb{T}^{2}}  ( \Delta u^{\nu}_{r} \cdot \nabla)\omega^{\nu}_{r} \Delta \omega^\nu_r \ \diff V + \int_{\mathbb{T}^2}  (u^{\nu}_{r} \cdot \nabla \Delta \omega^{\nu}_{r}) \Delta \omega^{\nu}_{r} \ \diff V \\
&+&   2\int_{\mathbb{T}^2}  \displaystyle\sum_{|\alpha|=1}\left(D^{\alpha}u^{\nu}_{r} \cdot  D^{\alpha}\nabla \omega^{\nu}_{r}\right)\Delta \omega^\nu_r \ \diff V.
\end{eqnarray*}
The second integral on the right-hand side above is zero due to the incompressibility condition.  The first integral can by bounded as follows
\begin{eqnarray*}\left| \int_{\mathbb{T}^{2}}   (\Delta u^{\nu}_{r} \cdot \nabla) \omega^{\nu}_{r} \Delta \omega^\nu_r \ \diff V \right| &\lesssim & \norm{\Delta u^{\nu}_{r}}_{L^{4}} \norm{\nabla \omega^\nu_r}_{L^{4}} \norm{\Delta \omega^\nu_r}_{L^{2}} \\
&\lesssim & \norm{\nabla u^{\nu}_{r}}_{L^{\infty}}^{1/2}\norm{\nabla \Delta u^{\nu}_{r}}^{1/2}_{L^{2}}\norm{\omega^{\nu}_{r}}^{1/2}_{L^{\infty}}\norm{\Delta \omega^{\nu}_{r}}^{1/2}_{L^{2}}\norm{\Delta\omega^{\nu}_{r}}_{L^{2}} \\
&\lesssim & \norm{\nabla u^{\nu}_{r}}_{L^{\infty}}\norm{\omega^{\nu}_{r}}^{2}_{H^{2}}, 
\end{eqnarray*}
where we have used the Gagliardo-Nirenberg inequality \eqref{Sob:ine3} and the Biot-Savart mapping \eqref{BiotIne}.
We can also estimate the last integral with the aid of H\"older's inequality 
\[ \left|\int_{\mathbb{T}^{2}} \displaystyle\sum_{|\alpha|=1} \left(D^{\alpha}u^{\nu}_{r} \cdot D^{\alpha} \nabla \omega^{\nu}_{r}\right)\Delta \omega^\nu_r \ \diff V \right| \lesssim \norm{\nabla u^{\nu}_{r}}_{L^{ \infty}} \norm{\omega^{\nu}_{r}}^{2}_{H^{2}},\]
thus proving our claim. \\
\item The term $\langle \Delta \partial_{x} \theta^{\nu}_{r}, \Delta \omega^\nu_r \rangle_{L^2}$ can be easily bounded by applying H\"older's inequality:
\begin{equation}\label{thetaterm}
\left| \int_{\mathbb{T}^{2}} \Delta \partial_{x}\theta^{\nu}_{r} \Delta \omega^\nu_r  \ \diff V \right| \lesssim ||\omega^\nu_r||_{H^{2}} ||\theta^\nu_r ||_{H^{3}}. 
\end{equation}
\end{itemize}
On the other hand, by taking three derivatives in equation \eqref{eq:theta:reg}, dotting against $\Lambda^{3}\theta^{\nu}_{r}$, using It\^o's formula, and integrating over $\mathbb{T}^2$,
\begin{eqnarray*}
\frac{1}{2} \int_{\mathbb{T}^{2}} |\Lambda^{3} \theta^{\nu}_{r} (t)|^2 \ \diff V &=& \frac{1}{2} \int_{\mathbb{T}^{2}}  |\Lambda^{3} \theta^{\nu}_{r} (0)|^2 \ \diff V - \int_0^t \langle \eta_{r}(||\nabla\theta||_{L^{\infty}}) \Lambda^{3} \mathcal{L}_{u^{\nu}_{r}} \theta^{\nu}_{r}(s), \Lambda^{3} \theta^\nu_r (s) \rangle_{L^2} \diff s \\
&-& \displaystyle \sum_{i=1}^{\infty} \int_0^t  \langle \Lambda^{3}  \mathcal{L}_{\xi_{i}} \theta^{\nu}_{r} (s), \Lambda^{3} \theta_r^\nu (s) \rangle_{L^2} \diff B^{i}_{s} + \int_0^t \langle \nu \Lambda^{3} \Delta^{7} \theta^{\nu}_{r} (s), \Lambda^{3} \theta^\nu_r (s) \rangle_{L^2} \diff s  \\
&+& \frac{1}{2} \sum_{i=1}^\infty \int_0^t \langle \Lambda^{3} \mathcal{L}_{\xi_i} \theta^\nu_r(s), \Lambda^{3} \mathcal{L}_{\xi_i} \theta^\nu_r (s) \rangle_{L^2} \diff s
+ \displaystyle \frac{1}{2} \sum_{i=1}^{\infty}\int_0^t  \langle \Lambda^{3} \mathcal{L}^{2}_{\xi_{i}} \theta^{\nu}_{r} (s), \Lambda^{3} \theta^\nu_r (s) \rangle_{L^2} \diff s.
\end{eqnarray*}
Let us analyse each term separately: \\
\begin{itemize}[leftmargin=*]
\item The dissipative term $\int_0^t \langle \nu \Lambda^{3} \Delta^{7} \theta^{\nu}_{r} (s), \Lambda^{3} \theta^\nu_r (s) \rangle_{L^2} \diff s$ can be ignored. Indeed, applying integration by parts
\[ \langle \nu \Lambda^{3}\Delta^{7} \theta^{\nu}_{r} , \Lambda^{3} \theta^\nu_r  \rangle_{L^2} = -\nu \int_{\mathbb{T}^2} |\Delta^{5} \theta_r^\nu |^2 \diff V, \]
and thus we see that it has the correct sign.  \\
\item The sum of the last two terms can be bounded thanks to \eqref{eq:cancellation2} as
\[  \frac{1}{2} \sum_{i=1}^\infty \langle \Lambda^{3} \mathcal{L}_{\xi_i} \theta^\nu_r, \Lambda^{3} \mathcal{L}_{\xi_i} \theta^\nu_r \rangle _{L^2}
+ \frac{1}{2} \displaystyle\sum_{i=1}^{\infty} \langle \Lambda^{3} \mathcal{L}^{2}_{\xi_{i}} \theta^{\nu}_{r}, \Lambda^{3} \theta^\nu_r \rangle_{L^2} \lesssim ||\theta^{\nu}_{r} ||^{2}_{H^{3}}.\] \\
\item The nonlinear term can be estimated as in the deterministic case,
\[ \left| \int_{\mathbb{T}^2}  \Lambda^{3} \mathcal{L}_{u^{\nu}_{r}} \theta^{\nu}_{r} \Lambda^{3} \theta^\nu_r \ \diff V\right| \lesssim  (||\nabla u^{\nu}_{r} ||_{L^\infty} + ||\nabla\theta^{\nu}_{r}||_{L^{\infty}}) (||\omega^{\nu}_{r}||^{2}_{H^{2}}+||\theta^{\nu}_{r}||^2_{H^{3}}). \] 
We omit the proof to avoid redundancy, since it is quite similar to the $H^{2}$ estimate for the nonlinear term \eqref{est:nonlinear:u}. 
\end{itemize}\vspace{0.5cm}
To conclude the proof, we just need to bound the local martingale terms. This is done by estimating there quadratic variation and using the Burkholder-Davis-Gundy inequality \eqref{BGD:ineq}. Indeed, let us denote
\makeatletter
 \def\@eqnnum{{\normalsize \normalcolor (\theequation)}}
  \makeatother
 { \small	
\begin{equation*}
M_t = \int_0^t \displaystyle\sum_{i=1}^\infty \left(  \langle \mathcal{L}_{\xi_{i}} \omega^{\nu}_{r} (s),  \omega_r^\nu (s) \rangle_{L^2}  + \langle \mathcal{L}_{\xi_{i}} \theta^{\nu}_{r} (s),  \theta_r^\nu (s) \rangle_{L^2}  +  \langle \Delta  \mathcal{L}_{\xi_{i}} \omega^{\nu}_{r} (s), \Delta \omega_r^\nu (s) \rangle_{L^2} + \langle \Lambda^{3} \mathcal{L}_{\xi_{i}} \theta^{\nu}_{r} (s), \Lambda^{3} \theta_r^\nu (s) \rangle_{L^2} \right) \diff B_s^i. 
\end{equation*}}
We will denote the quantities $\omega_r^\nu,\theta^\nu_{r}$ by $\omega,\theta,$ respectively, to make the notation in the following estimates less cumbersome, but implicitly taking into account that they indeed depend on $\nu$ and $r.$ From the aforementioned estimates we can infer
\[ ||\omega(t)||^{2}_{H^{2}} +||\theta(t)||^{2}_{H^{3}} \lesssim ||\omega_{0}||^{2}_{H^{2}} +||\theta_{0}||^{2}_{H^3} +  M_t + \eta(r) \int_0^t \left(||\omega(s)||^{2}_{H^{2}}+||\theta(s)||^{2}_{H^{3}} \right) \diff s, \]
%\textcolor{red}{Incomplete. Missing $\theta$ part.}
%By using the integrating factor $exp (-\int_0^t \kappa(r) ||\omega_s||_{H^{2}} \diff s )$\textcolor{blue}{que es integrating factor? aplicar gronwall?} 
and thus by Gr\"onwall's inequality
\[ \displaystyle\sup_{s\in[0,t]}\left( ||\omega(s)||^{2}_{H^{2}} +||\theta(s)||^{2}_{H^{3}}\right) \lesssim \text{exp}(\eta(r)t) \left( ||\omega_{0}||^{2}_{H^{2}}+||\theta_0||^{2}_{H^{3}} +\displaystyle\sup_{s\in[0,t]}\left|M_s\right|\right). \]
Consequently, by taking expectations,
\begin{equation}\label{est:expectation}
\mathbb{E}  \left[ \sup_{s \in [0,t]}\left(||\omega(s)||_{H^{2}}^4+||\theta(s)||^{4}_{H^{3}}\right) \right] \lesssim \text{exp}(\eta(r)t) \left( ||\omega_0||^4_{H^{2}}+||\theta_0||^{4}_{H^{3}} + \mathbb{E} \left[ \sup_{s \in [0,t]} |M_s|^2 \right] \right).\end{equation}
Invoking Burkholder-Davis-Gundy inequality \eqref{BGD:ineq}, the term $|M_s|$ can be controlled by
\begin{equation}\label{est:burk-dav-gun}
\mathbb{E} \left[ \sup_{s \in [0,t]} |M_{s}|^{2} \right] \lesssim  \mathbb{E} \left[\left[M\right]_{t}\right], 
\end{equation}
where $\left[M_t\right]$  is the quadratic variation of the process $M_t,$ given by
$$\left[M\right]_t = \int_0^t \displaystyle\sum_{i=1}^\infty  \left( \langle \mathcal{L}_{\xi_{i}} \omega(s),  \omega(s) \rangle_{L^2} + \langle \mathcal{L}_{\xi_{i}} \theta (s),  \theta (s) \rangle_{L^2}  +  \langle \Delta  \mathcal{L}_{\xi_{i}} \omega(s), \Delta \omega(s) \rangle_{L^2} + \langle \Lambda^{3} \mathcal{L}_{\xi_{i}} \theta(s), \Lambda^{3} \theta(s) \rangle_{L^2} \right)^2 \diff s.$$
% $\mathbb{E}  [\sup_{s \in [0,t]}||\omega_s||_{H^{2}}^2]$ can be controlled, uniformly in time. For this we only need a control in expectation of $[M]_t$ in terms of $||\omega||_{H^{2}}$ and to apply Gr\"onwall.
It follows that
\begin{equation}
\left[M\right]_{t}\leq  \displaystyle\sum_{i=1}^{\infty} \int_{0}^{t} C_{i} \left(\norm{\omega(s)}^{4}_{H^2} + \norm{\theta(s)}^{4}_{H^3} \right) \ \diff s, 
\end{equation}
where $C_{i}=C(\norm{\xi_{i}}_{H^{3}})$
since the highest order terms cancel, namely
$$\int_{\mathbb{T}^2} (\xi_i \cdot \nabla \Delta \omega) \Delta \omega \diff V = - \frac{1}{2} \int_{\mathbb{T}^2} \text{div} \ (\xi_i)  |\Delta \omega|^2 \diff V = 0.$$
Therefore, by making use of assumption \eqref{assumption3} we obtain
\begin{equation}\label{est:martingale}
\mathbb{E} \left[[M]_t\right] \lesssim \int_0^t \mathbb{E} \left[ \sup_{s \in [0,\gamma] } \left(|| \omega (\gamma) ||^{4}_{H^{2}}+||\theta(\gamma)||^{4}_{H^{3}}\right)\right] \diff \gamma.
\end{equation}
%Last, the term
%$$\int_0^t \langle \Delta \partial_{x} \theta^{\nu}_{r}, \Delta \omega^\nu_r \rangle  $$
%can be estimated by $||\omega^\nu_r||_{H^{2}} ||\theta^\nu_r ||_{H^{3}}.$
Hence, from estimates \eqref{est:expectation}, \eqref{est:burk-dav-gun}, \eqref{est:martingale}, and Gr\"onwall's inequality we have that
$$\mathbb{E} \left [ \sup_{t \in [0,T]} ||\omega (t)||^4_{H^{2}} \right ] + \mathbb{E} \left [ \sup_{t \in [0,T]} ||\theta(t)||^4_{H^{3}} \right ] \leq C(T).$$
Finally, bound \eqref{est:apriori:full} follows by a simple application of Jensen's inequality.
\end{proof}
\newpage

\section{The blow-up criterion}\label{blowup:crit}
In this subsection we prove an analogue of the well-known Beale-Kato-Majda criterion for the Euler equation, but this time for the stochastic Boussinesq system. However, let us first discuss some key differences between the deterministic and the stochastic models. \\
\subsection{The deterministic blow-up criterion.}\label{subsec:det:blowup}
The deterministic Boussinesq equations in vorticity form are given by
\begin{empheq}[left = \empheqlbrace]{align}
\label{eq:vorticity:deter} & \partial_{t}\omega + u\cdot \nabla\omega = \partial_{x}\theta,  \\
\label{eqtheta:deter} & \partial_{t}\theta + u\cdot \nabla \theta= 0, 
\end{empheq}
where $\omega= \nabla^{\perp}\cdot u$ and $\text{div} \ u= 0$.  Local existence and uniqueness of strong solutions in $H^{s-1}(\mathbb{T}^2,\mathbb{R})\times H^{s}(\mathbb{T}^2,\mathbb{R})$ for $s>2$ can be shown by obtaining a priori estimates and applying a Picard type theorem. Let us now assume that $(\omega,\theta)$ are local solutions and let $T^*>0.$ If 
 \begin{equation}\label{bkm:det1}
  \int_{0}^{T^{\star}}\norm{\nabla u(t)}_{L^\infty} \ \diff t < \infty,
\end{equation}
then the solution  can be extended to $[0,T^{\star}]$.  Indeed, to do so, we can compute the $H^{s-1}$ norm of the vorticity as follows
\[ \frac{1}{2}\frac{\diff}{\diff t} \norm{D^{s-1}\omega}^{2}_{L^{2}} = \int_{\mathbb{T}^{2}} D^{s-1}\partial_{x}\theta D^{s-1}\omega \ \diff V - \int_{\mathbb{T}^{2} }D^{s-1}\omega D^{s-1}(u\cdot\nabla\omega)  \diff V. \]
Using Leibniz rule, Gauss's theorem, and well-known calculus inequalities, we have that
\[ \frac{\diff}{\diff t} \norm{D^{s-1}\omega}^{2}_{L^{2}} \lesssim \norm{D^{s}\theta}^{2}_{L^{2}}+ \norm{D^{s-1}\omega}^{2}_{L^{2}}+ \norm{D^{s-1}\omega}^{2}_{L^2}\norm{\nabla u}_{L^\infty}. \]
Similarly, for the potential temperature $\theta$, we obtain
\[ \frac{\diff}{\diff t} \norm{D^{s}\theta}^{2}_{L^2} \lesssim \norm{D^{s}\theta}_{L^2}\left(\norm{D^{s}u}_{L^2}\norm{\nabla\theta}_{L^{\infty}}+\norm{D^{s}\theta}_{L^2}\norm{\nabla u }_{L^{\infty}}\right). \]
Moreover, one gets the upper bound
\[ \frac{1}{p}\frac{\diff}{\diff t}\norm{\nabla \theta}^{p}_{L^{p}} \lesssim \norm{\nabla u}_{L^{\infty}}\norm{\nabla\theta}^{p}_ {L^p}, \]
for any $p\in[2,\infty]$. By Gr\"onwall's lemma, the inequality 
\[ \norm{\nabla \theta}^{p}_{L^{p}} \lesssim \norm{\nabla\theta_{0}}^{p}_{L^{p}} \text{exp}\left(\int_{0}^{t} \norm{\nabla u(s)}_{L^\infty} \ \diff s \right), \]
holds for any $p\in[2,\infty]$. Therefore, letting $E(t)= \norm{D^{s-1}\omega}^{2}_{L^2}+\norm{D^{s}\theta}^{2}_{L^{2}}$, one derives the energy inequality
\[ \frac{\diff }{\diff t}E(t) \lesssim \norm{\nabla\theta_{0}}_{L^{p}} \text{exp} \left(\int_{0}^{t} \norm{\nabla u(s)}_{L^\infty} \ \diff s \right)   E(t)+E(t)\left(\norm{\nabla u}_{L^\infty}+1\right).\]
Invoking Gr\"onwall's lemma one gets that if (\ref{bkm:det1}) holds, then $E(t)<\infty$ on $[0,T^{\star}]$. 
\begin{remark}
Furthermore, one can also prove the same result under the alternative assumption
\[ \int_{0}^{T^{\star}} \norm{\nabla\theta(t)}_{L^{\infty}} \diff  t < \infty. \]
Nevertheless, we omit the proof, since the strategy is quite similar although a little more involved. One needs to use more sophisticated calculus inequalities, like for instance, a logarithmic type Sobolev inequality (cf. \cite{BreWai}).
\end{remark}
\begin{remark}
Notice that we cannot expect (as for 3D Euler), a Beale-Kato-Majda criterion \cite{BKM84} stating that if
\[\int_{0}^{T^*} \norm{\omega(t)}_{L^\infty} \diff t < \infty, \]
then the corresponding solution stays regular on $[0,T^*],$ i.e. a blow-up criterion in terms of the vorticity. This is due to the fact that we cannot control properly $\theta$ by using solely the vorticity. Note that if we only had $u$ in the equations, this could be done easily by using a logarithmic inequality like in 3D Euler. However, for the 2D Boussinesq equations it is unknown whether controlling $\norm{\omega}_{L^\infty}$ is enough for global regularity.
\end{remark}
\subsection{The stochastic blow-up criterion.} A priori, one would expect this type of results might be recovered for the stochastic Boussinesq equations. Nevertheless, an immediate analysis reveals that several complications arise, namely: \\
\begin{itemize}[leftmargin=*]
\item When computing the evolution of the $H^{s-1}$ norm of the vorticity and the $H^{s}$ norm of the potential temperature, various extra terms appear in the stochastic case. For instance, the term $\displaystyle\sum_{i=1}^{\infty} \mathcal{L}^{2}_{\xi_{i}} \omega \diff t$ (a second differential operator) needs to be treated carefully. \\ 
\item Another term one must take into account is the It\^o correction, which appears constantly while carrying out computations. \\
\end{itemize}
By using the algebraic in \cite{CriHolFla}, one can manage to manipulate these differential operators and explore some cancellations involving the most singular terms. We also want to point out that when deriving the estimates for the stochastic blow-up criterion, a new term appears in the stochastic Boussinesq case, which seems to make the derivation of a stochastic version of the deterministic Boussinesq criterion hopeless (see Remark \ref{destroyhope}). However, we are able to show a slightly weaker yet very useful version, which reads as follows.
\begin{theorem}[Blow-up criterion for stochastic Boussinesq]\label{theorem:blowup}
Let us define the stopping times $\tau^2$ and $\tau^{\infty}$ by
$$\tau^2 = \lim_{n \rightarrow \infty} \tau_n^2, \qquad \tau_n^2 = \inf \left\{t \geq 0: \norm{\omega(\cdot,t)}_{H^2} + \norm{\theta(\cdot,t)}_{H^3} \geq n  \right\},$$
$$\tau^{\infty} = \lim_{n \rightarrow \infty} \tau_n^{\infty}, \qquad \tau_n^{\infty} = \inf \left\{t \geq 0: \int_0^t ( \norm{\nabla u(\cdot,s)}_{L^\infty}  + \norm{\nabla \theta(\cdot,s)}_{L^\infty} ) \ \diff s\geq n  \right\}.$$
Then $\tau^2 = \tau^{\infty},$ $\mathbb{P}$ almost surely.
\end{theorem}
\begin{proof}[Proof of Theorem \ref{theorem:blowup}] We prove the above equality by showing both $\tau^{2}\leq \tau^{\infty}$ and $\tau^{\infty}\leq \tau^{2}$ in two different steps. \\ \\
\textsl{Step 1: $\tau^{2}\leq \tau^{\infty}$.}
This inequality is straightforward to check. It follows from the Sobolev embedding inequality (\ref{Sob:ine2}) and Biot-Savart mapping (\ref{BiotIne}) that
\[ \norm{\nabla u}_{L^\infty} + \norm{\nabla \theta}_{L^\infty} \lesssim \norm{\omega}_{H^2}+ \norm{\theta}_{H^3} .\]
\textsl{Step 2: $\tau^{\infty}\leq \tau^{2}$.}
Consider the hyper-regularised truncated equations introduced in Subsection \ref{globalregular} given by 
\begin{eqnarray}
\label{eq:vor:reg2} \diff  \omega^{\nu}_{r} + \eta_{r}(||\nabla u||_{L^{\infty}})\mathcal{L}_{u^{\nu}_{r}} \omega^{\nu}_{r} \diff t+\displaystyle\sum_{i=1}^{\infty}  \mathcal{L}_{\xi_{i}} \omega^{\nu}_{r} \diff  B^{i}_{t}&=& \nu \Delta^{5} \omega^{\nu}_{r} \diff t +\frac{1}{2} \displaystyle\sum_{i=1}^{\infty}  \mathcal{L}^{2}_{\xi_{i}}\omega^{\nu}_{r} \diff t+\partial_{x} \theta^{\nu}_{r} \diff t,  \\
\label{eq:theta:reg2}  
 \diff  \theta^{\nu}_{r} +\eta_{r}(||\nabla\theta||_{L^{\infty}}) \mathcal{L}_{u^{\nu}_{r}} \theta^{\nu}_{r} \diff t + \displaystyle\sum_{i=1}^{\infty}  \mathcal{L}_{\xi_{i}}\theta^{\nu}_{r} \diff  B^{i}_{t} &=& \nu \Delta^{7} \theta^{\nu}_{r} \diff t + \frac{1}{2} \displaystyle\sum_{i=1}^{\infty}  \mathcal{L}^{2}_{\xi_{i}}\theta^{\nu}_{r} \diff t,  
 \end{eqnarray}
with initial data $\omega^{\nu}_{r}(x,0)=\omega_{0}, \ \theta^{\nu}_{r}(x,0)=\theta_{0}.$ To simplify notation we will omit subscripts $\nu$ and $r$ over the proof. We need to compute the evolution of the $H^{2}$ norm of the vorticity and the $H^{3}$ norm of the potential temperature. Therefore, we can write that
\begin{eqnarray*}
\frac{1}{2} \diff  \norm{\omega}^{2}_{L^{2}} &+& \eta(||\nabla u||_{L^{\infty}})\langle \mathcal{L}_{u} \omega , \omega\rangle_{L^{2}} \diff t + \displaystyle\sum_{i=1}^{\infty} \langle  \mathcal{L}_{\xi_{i}} \omega, \omega \rangle_{L^{2}} \ \diff B^{i}_{t} \\
&=& \langle \Delta^{5}\omega,\omega \rangle_{L^{2}} \diff t +\frac{1}{2} \displaystyle\sum_{i=1}^{\infty} \langle \mathcal{L}^2_{\xi_i} \omega, \omega \rangle_{L^2} \diff t + \frac{1}{2} \displaystyle\sum_{i=1}^{\infty}  \langle \mathcal{L}_{\xi_i} \omega, \mathcal{L}_{\xi_i} \omega \rangle_{L^2} \diff t+ \langle \partial_{x} \theta, \omega \rangle_{L^2} \diff t,
\end{eqnarray*}
and
\begin{eqnarray*}
\frac{1}{2} \diff  \norm{\theta}^{2}_{L^{2}} &+& \eta(||\nabla \theta||_{L^{\infty}}) \langle \mathcal{L}_{u} \theta,\theta\rangle_{L^{2}} \diff t + \displaystyle\sum_{i=1}^{\infty} \langle  \mathcal{L}_{\xi_{i}} \theta, \theta \rangle_{L^{2}} \ \diff  B^{i}_{t} \\
&=& \langle \Delta^{7}\theta,\theta \rangle_{L^{2}} \diff t + \frac{1}{2} \displaystyle\sum_{i=1}^{\infty} \langle \mathcal{L}^2_{\xi_i} \theta, \theta \rangle_{L^2} \diff t + \frac{1}{2} \displaystyle\sum_{i=1}^{\infty}  \langle \mathcal{L}_{\xi_i} \theta, \mathcal{L}_{\xi_i} \theta \rangle_{L^2} \diff t.
\end{eqnarray*}
Integrating by parts, using the divergence-free condition, H\"older's inequality, and the cancellation \eqref{eq:cancellation1} one obtains that
\begin{equation}\label{EstimateL2}
 \diff  \norm{\omega}^{2}_{L^{2}}+ \diff \norm{\theta}^{2}_{L^{2}}+2 \displaystyle\sum_{i=1}^{\infty} \left(\langle \mathcal{L}_{\xi_{i}} \omega, \omega \rangle_{L^{2}} +  \langle \mathcal{L}_{\xi_{i}} \theta, \theta \rangle_{L^{2}} \right) \ \diff  B^{i}_{t} \lesssim \left(1+\norm{\nabla \theta}_{L^\infty}\right) \norm{\omega}^{2}_{L^{2}}  \ \diff t.
\end{equation} 
The evolution of the $H^{2}$ norm of the vorticity and the $H^{3}$ norm of the potential temperature is given by
\begin{eqnarray*}
\frac{1}{2} \diff  \norm{\omega}^{2}_{H^{2}} &+& \eta(||\nabla u||_{L^{\infty}})\langle \Delta\mathcal{L}_{u} \omega , \Delta \omega\rangle_{L^{2}} \diff t + \displaystyle\sum_{i=1}^{\infty} \langle \Delta \mathcal{L}_{\xi_{i}} \omega, \Delta\omega \rangle_{L^{2}} \ \diff  B^{i}_{t} \\
&=& \langle \Delta^{6}\omega, \Delta \omega \rangle_{L^{2}} \diff t +\frac{1}{2} \displaystyle\sum_{i=1}^{\infty} \langle \Delta \mathcal{L}^2_{\xi_i} \omega, \Delta \omega \rangle_{L^2} \diff t + \frac{1}{2} \displaystyle\sum_{i=1}^{\infty}  \langle \Delta \mathcal{L}_{\xi_i} \omega, \Delta \mathcal{L}_{\xi_i} \omega \rangle_{L^2} \diff t+ \langle \Delta \partial_{x} \theta,\Delta \omega \rangle_{L^2} \diff t,
\end{eqnarray*}
and
\begin{eqnarray*}
\frac{1}{2} \diff  \norm{\theta}^{2}_{H^{3}} &+& \eta(||\nabla \theta||_{L^{\infty}}) \langle \Lambda^{3} \mathcal{L}_{u} \theta,\Lambda^{3}\theta\rangle_{L^{2}} \diff t + \displaystyle\sum_{i=1}^{\infty} \langle \Lambda^{3} \mathcal{L}_{\xi_{i}} \theta, \Lambda^{3} \theta \rangle_{L^{2}} \ \diff  B^{i}_{t} \\
&=& \langle \Lambda^{3} \Delta^{7}\theta, \Lambda^{3} \theta \rangle_{L^{2}} \diff t + \frac{1}{2} \displaystyle\sum_{i=1}^{\infty} \langle \Lambda^{3}\mathcal{L}^2_{\xi_i} \theta, \Lambda^{3} \theta \rangle_{L^2} \diff t + \frac{1}{2} \displaystyle\sum_{i=1}^{\infty}  \langle \Lambda^{3} \mathcal{L}_{\xi_i} \theta, \Lambda^{3} \mathcal{L}_{\xi_i} \theta \rangle_{L^2} \diff t,
\end{eqnarray*}
respectively. Let us estimate each term above separately: \\
\begin{itemize}[leftmargin=*]
\item $|\langle \Delta\mathcal{L}_{u} \omega , \Delta \omega\rangle_{L^{2}}| \lesssim \norm{\nabla u}_{L^{\infty}} \norm{\omega}^{2}_{H^{2}}$ , \\
\item $\langle \Delta^{6}\omega, \Delta \omega \rangle_{L^{2}}=-\norm{\Delta^{\frac{7}{2}}\omega}^{2}_{L^{2}} \leq 0$, \\ 
\item $|\langle \Delta \partial_{x} \theta,\Delta \omega \rangle_{L^2}| \lesssim  \norm{\theta}_{H^{3}}\norm{\omega}_{H^{2}},$ \\
\item $\langle \Lambda^{3} \Delta^{7}\theta, \Lambda^{3} \theta \rangle_{L^{2}}=-\norm{\Delta^{5}\theta}^{2}_{L^{2}}\leq 0$, \\
\item $| \langle \Lambda^{3} \mathcal{L}_{u} \theta,\Lambda^{3}\theta\rangle_{L^{2}}| \lesssim  \left(\norm{\nabla u}_{L^{\infty}}+ \norm{\nabla\theta}_{L^{\infty}}+1\right)\left(\norm{\omega}^{2}_{H^{2}}+\norm{\theta}^{2}_{H^{3}}\right).$ \\
\end{itemize}
By using all the previous estimates together with inequality (\ref{eq:cancellation2}), and Young's inequality, we have
\begin{eqnarray} \label{blowestimada}
\diff  \norm{\omega}^{2}_{H^{2}}+\diff  \norm{\theta}^{2}_{H^{3}} &+& 2\displaystyle\sum_{i=1}^{\infty} \left(\langle \mathcal{L}_{\xi_i} \omega, \omega \rangle_{L^2} + \langle  \mathcal{L}_{\xi_i} \theta,  \theta \rangle_{L^2} + \langle \Delta\mathcal{L}_{\xi_i} \omega, \Delta \omega \rangle_{L^2} + \langle \Lambda^{3} \mathcal{L}_{\xi_i} \theta, \Lambda^{3} \theta \rangle_{L^2} \right)  \diff  B_t^i  \nonumber \\ 
&\lesssim &\left( \norm{\nabla u}_{L^{\infty}}+ \norm{\nabla\theta}_{L^{\infty}}+1\right)(\norm{\omega}^{2}_{H^{2}}+\norm{\theta}^{2}_{H^{3}}) \ \diff t.
\end{eqnarray}
In order to deal with the stochastic term, we rewrite the last equation using It\^{o}'s formula for the logarithmic function (cf. \cite{CriHolFla})
\begin{equation}\label{stochasticitoformula}
\diff  \ \text{log } \left(\norm{\omega}^{2}_{H^2}+\norm{\theta}^{2}_{H^{3}}\right)= \frac{\diff  \left(\norm{\omega}^{2}_{H^2}+ \norm{\theta}^{2}_{H^{3}}\right)}{\norm{\omega}^{2}_{H^2}+ \norm{\theta}^{2}_{H^{3}}}-\frac{\diff  N_{t}}{\left(\norm{\omega}^{2}_{H^2}+ \norm{\theta}^{2}_{H^{3}}\right)^{2}}.
\end{equation}
Here we have assumed, without loss of generality, that $\norm{\omega}^{2}_{H^2}+ \norm{\theta}^{2}_{H^{3}}$ is nonzero and
\[N_{t}:= 2\displaystyle\sum_{i=1}^{\infty} \int_{0}^{t} \left( \langle\mathcal{L}_{\xi_i} \omega(s), \omega(s) \rangle_{L^2} + \langle \mathcal{L}_{\xi_i} \theta(s), \theta(s) \rangle_{L^2} + \langle \Delta \mathcal{L}_{\xi_i} \omega(s), \Delta \omega(s) \rangle_{L^2} + \langle \Lambda^{3} \mathcal{L}_{\xi_i} \theta(s), \Lambda^{3} \theta(s) \rangle_{L^2}\right)^2  \diff s.\]
By applying estimate \eqref{blowestimada}, we have that
\[ \diff  \ \text{log} \left(\norm{\omega}^{2}_{H^2}+\norm{\theta}^{2}_{H^{3}}\right) \lesssim  \frac{\left(1+\norm{\nabla u}_{L^{\infty}}+\norm{\nabla\theta}_{L^{\infty}}\right)\left(\norm{\omega}^{2}_{H^2}+\norm{\theta}^{2}_{H^3}\right)}{\norm{\omega}^{2}_{H^2}+\norm{\theta}^{2}_{H^{3}}} \diff t+ \diff M_{t}, 
\]
for the local martingale
\[ M_t = 2 \displaystyle\sum_{i=1}^{\infty} \int_0^t \frac{\langle \mathcal{L}_{\xi_i} \omega(s), \omega(s) \rangle_{L^2} + \langle \mathcal{L}_{\xi_i} \theta(s), \theta(s) \rangle_{L^2}+\langle \Delta \mathcal{L}_{\xi_i} \omega(s), \Delta \omega(s) \rangle_{L^2} + \langle \Lambda^{3} \mathcal{L}_{\xi_i} \theta(s), \Lambda^{3} \theta(s) \rangle_{L^2}}{\norm{\omega(s)}^{2}_{H^2}+\norm{\theta(s)}^{2}_{H^{3}}} \ \diff  B^{i}_{s}.\]
Thus, integrating in time we obtain that
\begin{eqnarray}\label{eq:gronwallform}
\text{log} \left(\norm{\omega(t)}^{2}_{H^2} + \norm{\theta(t)}^{2}_{H^{3}} \right) &\lesssim & \text{log } \left(\norm{\omega_0}^{2}_{H^2}+\norm{\theta_0}^{2}_{H^{3}}\right) + \int_0^t \left(1+\norm{\nabla u(s)}_{L^{\infty}}+\norm{\nabla\theta(s)}_{L^{\infty}}\right) \diff s \nonumber \\ &+& \int_{0}^{t} \diff M_{s}.
\end{eqnarray}
At this point, it is only left to find a good control of the stochastic integral in \eqref{eq:gronwallform}. This is done by invoking Burkholder-Davis-Gundy inequality. To that purpose, we show how to estimate the quadratic variation of the aforementioned stochastic integral, namely,
\makeatletter
 \def\@eqnnum{{\normalsize \normalcolor (\theequation)}}
  \makeatother
 { \small	
\begin{eqnarray*}
\left[ \int_{0}^{t} \ \diff  M_{s} \right]_{t}  &=& 4\displaystyle\sum_{i=1}^{\infty}\int_{0}^{t} \frac{\left(\langle \mathcal{L}_{\xi_i} \omega(s), \omega(s) \rangle_{L^2} + \langle \mathcal{L}_{\xi_i} \theta(s), \theta(s) \rangle_{L^2}+\langle \Delta \mathcal{L}_{\xi_i} \omega(s), \Delta \omega(s) \rangle_{L^2} + \langle \Lambda^{3} \mathcal{L}_{\xi_i} \theta(s), \Lambda^{3} \theta(s) \rangle_{L^2}\right)^{2}}{\left(\norm{\omega(s)}^{2}_{H^2}+\norm{\theta(s)}^{2}_{H^{3}}\right)^{2}} \ \diff s \\
&\lesssim & \int_{0}^{t} \frac{\norm{\omega(s)}^{4}_{H^{2}}+\norm{\theta(s)}^{4}_{H^{3}}}{\left(\norm{\omega(s)}^{2}_{H^2}+\norm{\theta(s)}^{2}_{H^{3}}\right)^{2}} \ \diff s \\
&\lesssim & t,
\end{eqnarray*}}
where we have used (\ref{assumption3}) and standard calculations to bound all the numerator terms in the first integral. Making use of inequality  \eqref{BGD:ineq}, we obtain
\begin{equation}\label{martingale:est:2}
\mathbb{E} \left[ \displaystyle\sup_{s\in[0,t]} \left|\int_{0}^{s} \ \diff M_{s}  \right| \right] \lesssim \sqrt{t}. 
\end{equation}
Taking expectation on \eqref{eq:gronwallform} and using the estimate \eqref{martingale:est:2}, we derive
\begin{equation}
\mathbb{E}\left[  \displaystyle\sup_{s\in[0,\tau^{\infty}_{n}\wedge m ]} \displaystyle\log\left(\norm{\omega(s)}^{2}_{H^{2}}+\norm{\theta(s)}^{2}_{H^{3}}\right) \right] \lesssim \displaystyle\log\left(\norm{\omega_{0}}^{2}_{H^{2}}+\norm{\theta_{0}}^{2}_{H^{3}}\right) + m(n+1) + \sqrt{t} < \infty,
\end{equation}
for any $n,m \in \mathbb{N}$. 
So we have proven that for any $n,m\in\mathbb{N},$ 
\[ \mathbb{E} \left[  \text{log} \left(\displaystyle\sup_{s\in[0,\tau^{\infty}_{n}\wedge m ]} 
 \left( \norm{\omega(s)}_{H^2}+\norm{\theta(s)}_{H^{3}}\right)^{2} \right) \right]< \infty,
\]
which implies in particular that $\displaystyle\sup_{s\in[0,\tau^{\infty}_{n}\wedge m ]} \left(\norm{\omega(s)}_{H^2}+\norm{\theta(s)}_{H^{3}}\right)$ is a finite measure random variable $\mathbb{P}$ almost surely, this is
\[ \mathbb{P}\left(\displaystyle\sup_{s\in[0,\tau^{\infty}_{n}\wedge m ]} \left(\norm{\omega(s)}_{H^2}+\norm{\theta(s)}_{H^{3}}\right)< \infty \right) =1. \]
Recall that we have omitted the subscripts $\nu,r$ during the proof. However, by using Fatou's lemma we can take limits on $\omega^{\nu}_{r},\theta^{\nu}_{r}$ as $\nu$ goes to zero and $r$ tends to infinity, hence recovering the same result on the limit. To finish the proof we just need to notice that if
\[ \mathbb{P}\left(\displaystyle\sup_{s\in[0,\tau^{\infty}_{n}\wedge m ]} \left(\norm{\omega(s)}_{H^2}+\norm{\theta(s)}_{H^{3}}\right)< \infty \right) =1. \]
for any $n,m\in\mathbb{N}$, then $\tau^{\infty} \leq \tau^{2}$ (c.f. \cite{CriHolFla}).
\end{proof}

\begin{remark}\label{destroyhope}
The question of whether one can improve this blow-up criterion is quite natural. For instance, one could wonder whether it would be possible to recover the deterministic blow-up criterion, where it suffices to control
\[ \int_{0}^{T*} \norm{\nabla u(t)}_{L^{\infty}} \diff t \ \text{ \ or \ } \  \int_{0}^{T*} \norm{\nabla \theta(t)}_{L^{\infty}} \diff t.\]
As sketched in the previous discussion on the deterministic Boussinesq equations (see \ref{subsec:det:blowup}), this follows by performing several $L^{p}$ estimates on the evolution of the potential temperature in the deterministic case (see \ref{subsec:det:blowup}). If one tries to adapt this idea to the stochastic Boussinesq equations, the It\^{o} correction terms destroy any hope. Indeed, if $d X_t = \mu_t \diff t + \sigma_t \diff B_t $ and $f(t,x)$ is a smooth enough function, then
\[ \diff f(t,X_t) = \frac{\partial f}{\partial t} \diff t + (\nabla_x f)^T \diff X_t + \frac{1}{2} Tr [\sigma^T_t Hess_x(f) \sigma_t ] \diff t. \]
Now let $f = x^p/p.$ We obtain
\[ \diff f (X_{t}) = X_{t}^{p-1} \diff X_{t} + \frac{1}{2} Tr [\sigma^T_t Hess_x(f) \sigma^T_t ] \diff t. \]
Therefore, in the stochastic case, by computing the evolution of the $L^{p}$ norm of the gradient of the potential temperature, we have that 
\begin{eqnarray*}
\frac{1}{p} \diff ||\nabla \theta||^{p}_{L^p} &=& - \int_{\mathbb{T}^{2}} \eta_r (||\nabla \theta||_{L^\infty}) ((\nabla u \cdot \nabla) \theta ) \cdot \nabla\theta |\nabla \theta|^{p-2} \ \diff V \diff t \\ 
&-& \sum_{i=1}^\infty \int_{\mathbb{T}^{2}} \nabla \mathcal{L}_{\xi_i} \theta \cdot \nabla \theta |\nabla \theta|^{p-2} \diff V \diff B_t^i  + \nu \int_{\mathbb{T}^{2}}  \nabla \Delta^5 \theta \cdot \nabla \theta |\nabla \theta|^{p-2} \diff V \diff t \\
&+& \frac{1}{2} \sum_{i=1}^\infty \int_{\mathbb{T}^{2}} \nabla \mathcal{L}^2_{\xi_i} \theta \cdot \nabla \theta |\nabla \theta|^{p-2} \diff V \diff t + \frac{p-1}{2} \sum_{i=1}^\infty \int_{\mathbb{T}^{2}} \nabla \mathcal{L}_{\xi_i} \theta \cdot \nabla \mathcal{L}_{\xi_i} \theta |\nabla \theta|^{p-2} \diff V \diff t.
\end{eqnarray*}
The main complication comes from the last two terms. It is easy to check that those integrals contain several high order singular terms we are not able to deal with if $p\neq 2$. This could be due to the special structure enjoyed by Hilbert spaces (case of $p=2$). 
\end{remark}

\newpage
\section{Conclusions}\label{conclusions}
In this paper we have established local well-posedness in Sobolev spaces of a stochastic Boussinesq model. The model itself comes from modifying the variational principle for deterministic Boussinesq to include cylindrical multiplicative noise, following the ideas in \cite{Principal}. The rich properties of this stochastic formulation enable this type of results, since the proposed method fundamentally preserves the ``geometric structure'' of the deterministic Boussinesq system. Our result contributes to validating the methods for adding stochasticity to the equations proposed in \cite{Principal}, as physical. Moreover, it makes this stochastic Boussinesq model a good candidate for real weather simulations. Note that thanks to the estimates in the appendix, our well-posedness results could be extended to $(\omega_{0},\theta_{0})\in H^{s-1}(\mathbb{T}^{2},\mathbb{R})\times H^{s}(\mathbb{T}^{2},\mathbb{R}),$ $s>2.$ \\\\
We have also constructed a stochastic blow-up criterion, which is an extension of the one in the deterministic case. However, this criterion is not as sharp as the deterministic one, since It\^o correction terms destroy the estimates on $||\nabla \theta||_{L^p}$ which permit a bound of the type $||\nabla \theta||_{L^\infty} \lesssim ||\nabla u||_{L^\infty}$ in the deterministic case. Our criterion is very useful for numerical simulations, since one can track the evolution of $||\nabla \theta||_{L^\infty} + ||\nabla u||_{L^\infty}$ to see whether solutions are likely to blow up in finite time. \\\\
Finally, we have derived general Lie derivatives estimates which serve to tackle well-posedness of a broader range of stochastic fluid dynamics equations with cylindrical multiplicative noise. We plan on commenting further on this topic and its various possible applications in a forthcoming paper.  \\\\
We propose a few research lines regarding some problems which are left for future research:\\
\begin{itemize}[leftmargin=*]
\item{One could prove that our blow-up criterion is indeed sharp, in the sense that there exist initial data for which the norm $||\nabla \theta||_{L^\infty}$ cannot be controlled by $||\nabla u||_{L^\infty}$, and therefore providing an example where the first norm blows up but not the second.  Numerical simulations have been carried out investigating this phenomenon and this might be part of a future paper, which would complete the theoretical results provided here.} \\
\item{Similar well-posedness results, as well as extensions of deterministic blow-up criteria, could be derived for other relevant stochastic physical equations such as MHD \cite{Principal}, slice atmospheric models \cite{DieAyt}, electromagnetic field equations \cite{Holm18}, etc.} \\ 
\item{The problem of finite time blow-up versus global existence of smooth solutions in the deterministic case is open and tremendously challenging. An interesting and ambitious question to ponder about is whether the presence of noise could prevent or mitigate the singularities in the stochastic model. Examining this kind of problems could help develop new intuition which might shed some light into the original deterministic problem.}
\end{itemize}
\newpage
\begin{appendix}
\section{The generalised Lie derivatives estimates}\label{appendixA}
We collect in this appendix the proof of Proposition \ref{Liecancellations}, dealing with the bounds on the Lie derivatives. The proof is derived from a more general result for linear operators of order one which turns out to be quite useful. We will provide the proof of this statement and comment on its various possible applications. The idea is to extend the results in \cite{CriHolFla}, by modifying their argument to be more general. More precisely, we provide an extension of their result to higher or fractional order differential operators and general linear differential operators of first order (i.e. not only 3D Lie derivatives). This shows that the special cancellations taking place in \cite{CriHolFla} not only occur due to the particularities of the Laplace operator and the Lie derivative noise type, but due to something more essential. The main idea behind our proof presented in this appendix relies on the fact that commutators of differential operators become slightly less singular operators. \\ \\
We first claim that the following inequality holds for every smooth enough vector field $f$,
\begin{equation}\label{Liegeneral1}
\langle \mathcal{Q}^{2}f,f \rangle_{L^2} + \langle \mathcal{Q}f,\mathcal{Q}f \rangle_{L^2} \lesssim ||f||^{2}_{L^{2}}.
\end{equation}
Here $Q$ is a linear differential operator of first order with bounded smooth coefficients. Indeed, this follows after a straightforward computation, since
\begin{equation}\label{estimate:adjoint:1}
\langle \mathcal{Q}^2 f, f \rangle_{L^{2}} = \langle \mathcal{Q} f, \mathcal{Q}^{\star} f \rangle_{L^2} = - \langle \mathcal{Q} f, \mathcal{Q} f \rangle_{L^{2}} + \langle \mathcal{Q} f, E f\rangle_{L^2}, 
\end{equation}
where $\mathcal{Q}^{\star}$ denotes the adjoint operator of $\mathcal{Q}$ under the $L^{2}$ pairing. Note that we have used
\begin{equation}\label{estimate:adjoint:2}
\mathcal{Q}^* = - \mathcal{Q} + E 
\end{equation}
where $E$ is a zero order operator, which follows from the general theory of differential operators. The last term on the right-hand side of \eqref{estimate:adjoint:1} can be rewritten as
\begin{eqnarray*}
\langle \mathcal{Q}f, Ef \rangle_{L^2} &=& - \langle f, \mathcal{Q} E f\rangle_{L^2} + \langle f, E^{2}f \rangle_{L^2} \\
&=& - \langle f, E\mathcal{Q}f\rangle_{L^2} - \langle f,T_{0} f\rangle_{L^2} + \langle f, E^{2}f \rangle_{L^2} \\
&=& - \langle Ef,\mathcal{Q}f \rangle_{L^2} - \langle f, T_{0} f\rangle_{L^2} + \langle f,E^{2}f\rangle_{L^2},
\end{eqnarray*}
since
\[  \mathcal{Q}E-E\mathcal{Q}= [\mathcal{Q},E]= T_{0}, \]
where $T_{0} $ is a zero order differential operator and the fact that $\langle Ef,g \rangle_{L^2} = \langle f,Eg \rangle_{L^2},$ for any $L^{2}$ integrable smooth vector fields $f,g$. Hence
\[   | \langle \mathcal{Q}^2 f, f \rangle_{L^{2}} + \langle \mathcal{Q} f, \mathcal{Q}f \rangle_{L^{2}}  | =  (1/2) | \langle f, T_{0} f\rangle_{L^2} + \langle f,E^{2}f\rangle_{L^2} |  \lesssim ||f||^{2}_{L^2}. \]
Next, let us show that for every smooth enough $f$,
\begin{equation}\label{Liegeneral2}
\langle \mathcal{P}\mathcal{Q}^{2}f,\mathcal{P} f \rangle_{L^{2}} + \langle \mathcal{P}\mathcal{Q}f, \mathcal{P}\mathcal{Q} f \rangle_{L^2} \lesssim ||f||^{2}_{H^{k}},
\end{equation}
where $\mathcal{P}$ is a pseudodifferential operator of order $k \in [1,\infty)$, and $Q$ is a linear differential operator of first order with smooth bounded coefficients. First, let us define
\[ T_{1} =  \mathcal{P} \mathcal{Q} - \mathcal{Q} \mathcal{P}=[\mathcal{P}, \mathcal{Q}]. \]
The classical theory of pseudodifferential operators states that the resulting commutator is of order $k$ (c.f. \cite{Taylorr} , \cite{Hormander}). Following the same idea, let us define
\[ T_{2} =  T_{1} \mathcal{Q} - \mathcal{Q} T_{1} = [T_{1}, \mathcal{Q}], \]
which is also an operator of order $k$ for the same reason. Hence, we have
\makeatletter
 \def\@eqnnum{{\normalsize \normalcolor (\theequation)}}
  \makeatother
 { \small	
\begin{eqnarray*}
\langle \mathcal{P} \mathcal{Q}^2 f,\mathcal{P} f \rangle_{L^{2}} &=& \langle (\mathcal{Q} \mathcal{P} + T_1) \mathcal{Q} f, \mathcal{P} f \rangle_{L^2}= \langle \mathcal{Q} \mathcal{P} \mathcal{Q} f, \mathcal{P} f \rangle_{L^2} + \langle T_1 \mathcal{Q} f, \mathcal{P} f \rangle_{L^2} \\
&=& \langle \mathcal{P} \mathcal{Q} f, \mathcal{Q}^{\star} \mathcal{P} f \rangle_{L^2} + \langle T_1 \mathcal{Q} f, \mathcal{P} f \rangle_{L^2}= -\langle \mathcal{P} \mathcal{Q} f, \mathcal{Q} \mathcal{P} f \rangle_{L^2} + \langle \mathcal{P} \mathcal{Q} f, E \mathcal{P} f \rangle_{L^2} + \langle T_1 \mathcal{Q} f, \mathcal{P} f \rangle_{L^2} \\
&=&-\langle \mathcal{P} \mathcal{Q} f, \mathcal{P} \mathcal{Q} f \rangle_{L^2} + \langle \mathcal{P}  \mathcal{Q} f, T_{1} f \rangle_{L^2}  + \langle T_1 \mathcal{Q} f, \mathcal{P} f \rangle_{L^2} + \langle \mathcal{P} \mathcal{Q} f, E \mathcal{P} f \rangle_{L^2},
\end{eqnarray*}} 
where we have used the definition of $T_1$ and \eqref{estimate:adjoint:2}. Therefore,
\begin{equation}\label{estimate:adjoint:3}
\langle \mathcal{P} \mathcal{Q}^2 f,\mathcal{P} f \rangle_{L^{2}} + \langle \mathcal{P} \mathcal{Q} f, \mathcal{P} \mathcal{Q} f \rangle_{L^2} = \langle \mathcal{P}  \mathcal{Q} f, T_{1} f \rangle_{L^2}  + \langle T_1 \mathcal{Q}  f, \mathcal{P} f \rangle_{L^2} + \langle \mathcal{P} \mathcal{Q} f, E \mathcal{P} f \rangle_{L^2}. 
\end{equation}
Once again, manipulating the above equality \eqref{estimate:adjoint:3}, we obtain
\makeatletter
 \def\@eqnnum{{\normalsize \normalcolor (\theequation)}}
  \makeatother
 { \small	
\begin{eqnarray*}
 \langle \mathcal{P} \mathcal{Q}^2 f,\mathcal{P} f \rangle_{L^{2}} + \langle \mathcal{P} \mathcal{Q} f, \mathcal{P} \mathcal{Q} f \rangle_{L^2} &=& \langle \mathcal{Q} \mathcal{P}   f, T_{1} f \rangle_{L^2}  + \langle T_{1} f, T_{1} f\rangle_{L^{2}} +\langle T_1 \mathcal{Q} f, \mathcal{P} f \rangle_{L^2} + \langle \mathcal{P} \mathcal{Q} f, E \mathcal{P} f \rangle_{L^2} \\
&=& - \langle \mathcal{P} f,\mathcal{Q} T_{1} f \rangle_{L^2} + \langle T_{1}f, T_{1}f\rangle_{L^{2}} + \langle T_1 \mathcal{Q} f, \mathcal{P} f \rangle_{L^2} \\
&+& \langle \mathcal{P} \mathcal{Q} f, E \mathcal{P} f \rangle_{L^2} + \langle \mathcal{P} f, E T_1 f \rangle_{L^2}  \\
&=& \langle (T_{1}\mathcal{Q} - \mathcal{Q}T_{1})f,\mathcal{P} f \rangle_{L^2} + \langle T_{1}f, T_{1}f\rangle_{L^{2}} + \langle \mathcal{P} \mathcal{Q} f, E \mathcal{P} f \rangle_{L^2}+ \langle \mathcal{P} f, E T_1 f \rangle_{L^2}  \\
&=& \langle T_{2} f, \mathcal{P} f \rangle_{L^2} + \langle T_{1}f, T_{1}f\rangle_{L^{2}} + \langle \mathcal{P} f, E T_1 f \rangle_{L^2}+\langle \mathcal{P} \mathcal{Q} f, E \mathcal{P} f \rangle_{L^2} ,
\end{eqnarray*}}
Notice that the last term on the right-hand side in the last equality seems to be singular as well. However, one can manage it as follows:
\makeatletter
 \def\@eqnnum{{\normalsize \normalcolor (\theequation)}}
  \makeatother
 { \small	
\begin{eqnarray*}
\langle \mathcal{P}\mathcal{Q}f, E\mathcal{P}f\rangle_{L^2} &=& \langle (\mathcal{Q}\mathcal{P}+T_{1})f, E\mathcal{P}f\rangle_{L^2}= \langle \mathcal{Q}\mathcal{P}f, E\mathcal{P}f\rangle_{L^{2}}+\langle T_{1}f, E\mathcal{P}f\rangle_{L^{2}}\\ &=&-\langle \mathcal{P}f, \mathcal{Q}E\mathcal{P}f\rangle_{L^2}+\langle \mathcal{P}f,E^{2}\mathcal{P}f\rangle_{L^2}+\langle T_{1}f,E\mathcal{P}f\rangle_{L^{2}}\\ &=& -\langle \mathcal{P}f, E\mathcal{Q}\mathcal{P}f\rangle_{L^2}-\langle \mathcal{P}f,T_0\mathcal{P}f\rangle_{L^2}+\langle \mathcal{P}f,E^{2}\mathcal{P}f\rangle_{L^2}+\langle T_{1}f, E\mathcal{P}f\rangle_{L^2} \\ &=&-\langle E\mathcal{P}f, \mathcal{Q}\mathcal{P}f\rangle_{L^{2}}-\langle \mathcal{P}f,T_0 \mathcal{P}f\rangle_{L^2}+\langle \mathcal{P}f,E^{2}\mathcal{P}f\rangle_{L^2}+\langle T_{1}f, E\mathcal{P}f\rangle_{L^2},
\end{eqnarray*}}
where we have used \eqref{estimate:adjoint:2} and the commutators constructed above. Hence,
\[ 2 \langle \mathcal{P}\mathcal{Q}f, E\mathcal{P}f\rangle_{L^2} = - \langle \mathcal{P}f,T_0\mathcal{P}f\rangle_{L^2}+\langle \mathcal{P}f,E^{2}\mathcal{P}f\rangle_{L^2}+2\langle T_{1}f, E\mathcal{P}f\rangle_{L^2}.\] 
Finally, by applying H\"{o}lder's inequality, plus the fact that $T_{1},T_{2},\mathcal{P}$ are differential operators of order $k,$ and $E,T_{0}$ are zero order operators, we conclude that
\begin{align*}
 \bigg|\langle \mathcal{P}\mathcal{Q}^{2}f,\mathcal{P} f \rangle_{L^{2}} + \langle \mathcal{P}\mathcal{Q}f, \mathcal{P}\mathcal{Q} f \rangle_{L^2} \bigg|
 &= \biggl| \langle T_{2} f, \mathcal{P} f \rangle_{L^2} + \langle T_{1}f, T_{1}f\rangle_{L^{2}}+\langle \mathcal{P} f, E T_1 f \rangle_{L^2}  \\
&-\frac{1}{2}\langle \mathcal{P}f,T_{0}\mathcal{P}f\rangle_{L^2}   +\frac{1}{2}\langle \mathcal{P}f,E^{2}\mathcal{P}f\rangle_{L^2}+\langle T_{1}f, E\mathcal{P}f\rangle_{L^2}\biggl|  \lesssim  \norm{f}^{2}_{H^{k}}.
\end{align*}

\begin{remark}
It is easy to see that \eqref{eq:cancellation1},\eqref{eq:cancellation2} represent a particular case of inequalities \eqref{Liegeneral1},\eqref{Liegeneral2}. Indeed, let $\mathcal{Q}=\mathcal{L}_{\xi_{i}}$ and $f$ be a smooth scalar function. Then we have that $\mathcal{Q}^{\star}=-\mathcal{Q},$ yielding \eqref{eq:cancellation1}. On the other hand, inequality \eqref{eq:cancellation2} follows by taking $\mathcal{Q}=\mathcal{L}_{\xi_{i}}$, $\mathcal{P}=\Lambda^{k},$ and $f$ a smooth scalar function. It is also worth noting that we have proven our estimates for smooth vector fields $f$ taking values in $\mathbb{T}^{2},$ but they extend to the whole space $\mathbb{R}^{2}$ without modifying the argument. Moreover, since all the commutator properties are also available for compact manifolds $M$, these estimates are also valid in that context.
\end{remark}

\begin{remark}
It is also important to note that the Lie derivative estimates in \cite{CriHolFla} can be extended to higher fractional order differential operators $\mathcal{P}$ and general first-order linear operators $\mathcal{Q}$, hence proving well-posedness results and blow-up criteria for a broader and much more general noise type.
\end{remark}
\end{appendix}

\newpage
%%%%%%%%%%%%%%%%%%%%%%%%%%%%%%%%%%%
\bigskip
\begin{center}
{\bf Acknowledgements}
\end{center}
The authors are indebted to D. Holm and A. C\'{o}rdoba for useful discussions. The first author has been partially supported by the grant  MTM2017-83496-P from the Spanish Ministry of Economy and Competitiveness and through the Severo Ochoa Programme for Centres of Excellence in R\&D  (SEV-2015-0554). The second author has been supported by the Mathematics of Planet Earth Centre of Doctoral Training (MPE CDT).

%%%%%%%%%%%%%%%%%%%%%%%%%%%%%%%%%%%
%\bibliographystyle{alpha}
%\bibliography{VladBib}

\newcommand{\etalchar}[1]{$^{#1}$}

%%%%%%%%%%%%%%%%%%%%%%%%%%%%%%%%%%%

\end{document}